\documentclass[11pt]{article}
\usepackage{amsmath,amssymb}
\usepackage{amsfonts}
\usepackage{amsthm}
\usepackage{eucal}
\usepackage{graphicx}
\usepackage{hyperref}
\usepackage{mathtools}

\DeclarePairedDelimiter\floor{\lfloor}{\rfloor}

\numberwithin{equation}{section}
\usepackage{xcolor}

\newtheorem{theorem}{Theorem}[section]

\newtheorem{lemma}[theorem]{Lemma}
\newtheorem{conjecture}[theorem]{Conjecture}

\newtheorem{remark}[theorem]{Remark}
\newtheorem{proposition}[theorem]{Proposition}

\begin{document}

%%%%%%%%%%%%%%%%%%%%%%%%%%%%%%%%%%%%%%%%
\title{On the non-degenerate and degenerate generic singularities formed by mean curvature flow}
\author{Zhou Gang\footnote{gzhou@math.binghamton.edu, partly supported by Simons Collaboration Grant 709542}}
\maketitle
\centerline{Department of Mathematical Sciences, Binghamton University, Binghamton, NY, 13850}
\setlength{\leftmargin}{.1in}
\setlength{\rightmargin}{.1in}
\normalsize \vskip.1in
\setcounter{page}{1} \setlength{\leftmargin}{.1in}
\setlength{\rightmargin}{.1in}
\large

\date

%\fixNumberingInArticle
%\NowFootNum
\normalsize
\section*{Abstract}
In this paper we study a neighborhood of generic singularities formed by mean curvature flow (MCF). We limit our consideration to the singularities modelled on $\mathbb{S}^3\times\mathbb{R}$ because, compared to the cases $\mathbb{S}^k\times \mathbb{R}^{l}$ with $l\geq 2$, the present case has the fewest possibilities to be considered. For various possibilities, we provide a detailed description for a small, but fixed, neighborhood of singularity, and prove that a small neighborhood of the singularity is mean convex, and the singularity is isolated. For the remaining possibilities, we conjecture that an entire neighborhood of the singularity becomes singular at the time of blowup, and present evidences to support this conjecture. A key technique is that, when looking for a dominating direction for the rescaled MCF, we need a normal form transformation, as a result, the rescaled MCF is parametrized over some chosen curved cylinder, instead over a standard straight one. 

This is a long paper. The introduction is carefully written to present the key steps and ideas.

\tableofcontents

\section{Introduction}\label{sec:mcf}
In the present paper we study mean curvature flow (MCF), an evolution of a family of $n-$dimensional hypersurfaces embedded in $\mathbb{R}^{n+1}$ satisfying the equation
\begin{align*}
\partial_{t}X_{t}=-h(X_{t}),
\end{align*} where $X_{t}$ is the immersion at time $t$, $h(X_{t})$ is the mean curvature vector at the point $X_{t}.$ It is well known that, in general, the MCF will form a singularity in a finite time. We are interested in studying a small space-and-time neighborhood of the generic singularities. Specifically, we consider initial hypersurfaces $\Sigma_0$ satisfy the condition 
\begin{align}\label{eq:generic}
\lambda(\Sigma_0)<\infty,
\end{align} where the functional $\lambda$ was defined in \cite{colding2012generic},
\begin{align*}
 \lambda(\Sigma_0):=\sup_{t_0,\ x_0} (4\pi t_0)^{-\frac{n}{2}}\int_{\Sigma_0} e^{-\frac{|x-x_0|^2}{4t_0}} d\mu,
\end{align*}and $\mu$ is the volume element.

Under the condition (\ref{eq:generic}), in \cite{ColdingMiniUniqueness} Colding and Minicozzi proved that, suppose that a singularity is formed at the time $T$ and the spatial origin, then for some $k+j=n$, with $j\geq 1,$ up to a rotation,
\begin{align}
\frac{1}{\sqrt{T-t}}X_t\rightarrow \mathbb{R}^k\otimes \mathbb{S}_{\sqrt{2j}}^j,\ \text{as}\ t\rightarrow T,\label{eq:limitCylinder}
\end{align}where $\mathbb{S}^{j}_{\sqrt{2j}}$ is the $j-$dimensional sphere with radius $\sqrt{2j}.$ 
Based on this, in our previous papers \cite{GZ2018,GZ2017} we studied the nondegenerate case, which is to be defined below, when the limit cylinder is $\mathbb{R}^3\otimes \mathbb{S}_{\sqrt{2}}$, and provided a detailed description for a small, but fixed, neighborhood of the singularity, and proved that it is mean convex and the singularity is isolated. We expect our techniques will work for the nondegenerate case when the limit cylinder is any of $\mathbb{R}^k\otimes \mathbb{S}_{\sqrt{2j}}^j$. The reason for choosing $k=3$ is that, compared to $k=1$, there are very few results for that case. 

In the present paper we study both the nondegenerate and degenerate cases when the limit cylinder is $\mathbb{R}^1\otimes \mathbb{S}^3_{\sqrt{6}}$. The reason for not choosing to study the cases $\mathbb{R}^{k}\otimes \mathbb{S}^{j}_{\sqrt{2j}}$ with $k\geq 2$ is that there are many more possibilities to be considered. This will be discussed around (\ref{eq:matrixB}) below, and we will address this problem in future.

In what follows we present the results and discuss some difficulties in proving them. 

By the results in \cite{ColdingMiniUniqueness}, in a possibly shrinking neighborhood of the singularity, the MCF takes the form, for some positive function $u: \ \mathbb{R}\times \mathbb{S}^3\times [0,T)\rightarrow \mathbb{R}^{+},$
\begin{align}
\left[
\begin{array}{ccc}
z\\
u(z, \omega, t)\omega
\end{array}
\right],
\end{align} where $\omega\in \mathbb{S}^3.$ The rescaled MCF is defined by the identity
\begin{align}\label{eq:crude}
\left[
\begin{array}{ccc}
y\\
v(y, \omega,\tau)\omega
\end{array}
\right]=\frac{1}{\sqrt{T-t}}\left[
\begin{array}{ccc}
z\\
u(z, \omega, t)\omega
\end{array}
\right]
\end{align} where $\tau$ and $y$ are rescaled time and spatial variables, defined as
\begin{align}
\tau:=|ln(T-t)|,\ \text{and}\ y:=\frac{z}{\sqrt{T-t}}=e^{-\frac{1}{2}\tau}z
\end{align} the functions $u$ and $v$ are related by the identity
\begin{align}
u(z,\omega,t)=\sqrt{T-t} v(y,\omega,\tau).\label{eq:uvidentity}
\end{align} By \cite{colding2012generic}, for each fixed $y\in \mathbb{R}$, $v(y,\omega,\tau)$ will converge to $\sqrt{6}$ as $\tau\rightarrow \infty.$ 

We will start with studying the rescaled MCF, look for some favorable information, then from there study the original MCF.

To find detailed information of the rescaled MCF, it is natural to decompose $v$ as,
\begin{align}
v(y,\omega,\tau)=\sqrt{6+\xi(y,\omega,\tau)}.\label{eq:decomVXi}
\end{align} Then $\xi$ satisfies the equation
\begin{align}
\partial_{\tau}\xi=-L\xi+NL(\xi)
\end{align} where the term $NL(\xi)$ is nonlinear in terms of $\xi$, and $L$ is a linear operator defined as
\begin{align}
L:=-\partial_{y}^2+\frac{1}{2}y\partial_{y}-1-\frac{1}{6}\Delta_{\mathbb{S}^3}=L_0-\frac{1}{6}\Delta_{\mathbb{S}^3}-1,\label{def:linearL}
\end{align} where the operator $L_0$ is naturally defined.

To understand the evolution of $\xi$, we study the spectrum of $L$. Its two components $L_0$ and $-\Delta_{\mathbb{S}^3}$ commute, thus it suffices to study them separately. The spectrum of $-\Delta_{\mathbb{S}^3}$ is
\begin{align}
\sigma(-\Delta_{\mathbb{S}^3})=&\Big\{ \ k(k+2)\ \Big|\ k=0,1,2,\cdots\Big\}.\label{eq:specOmega}
\end{align}The corresponding eigenfunctions, denoted by $f_{k,l}$, are the $k-$th order spherical harmonic functions and $l\in \{1,\cdots,\ \Omega_k\}$ for some $\Omega_k\geq 1.$ Thus $f_{k,l}$ satisfies the equation
\begin{align}
-\Delta_{\mathbb{S}^3}f_{k,l}=k(k+2) f_{k,l},\ \text{with}\ k=0,1,2,\cdots;\ \text{and}\  l=1,\cdots,\Omega_k.
\end{align} 
In what follows the explicit forms for $f_{k,l}$, $k=0$, $1$ are used often. Here the eigenfunctions are
\begin{align}
\begin{split}
f_{0,1}=&1, \\
f_{1,l}=&\omega_l,\ l=1,2,3,4,
\end{split}
\end{align} and $\omega_l$ is from $\omega=(\omega_1,\omega_2,\omega_3,\omega_4)$.
The operator $L_0$ is conjugate to the harmonic oscillator $\mathcal{L}_0:=e^{-\frac{1}{8}y^2} L_0 e^{-\frac{1}{8}y^2}$ with spectrum
\begin{align}
\sigma(\mathcal{L}_0)=\Big\{\frac{n}{2}\ \Big| \  n=0,1,2,\cdots\Big\}\label{eq:specL0}
\end{align} and the eigenfunction $e^{-\frac{1}{8}y^2}H_{n}(y).$ Here $H_n$ is the $n-th$ Hermite polynomial of the form $H_n(y)=\sum_{k=0}^{\floor*{\frac{n}{2}}}h_{n-2k,n} y^{n-2k}$ with $h_{n,n}\not=0.$ In the present paper the sign and size of $h_{n,n}$ are important, to facilitate later discussions we assume that $h_{n,n}=1.$ Hence, $H_n$ takes the form,
\begin{align}
    H_n(y)=y^n+\sum_{k=1}^{\floor*{\frac{n}{2}}}h_{n-2k,n} y^{n-2k}.\label{eq:hermit}
\end{align}

Returning to (\ref{eq:decomVXi}), we decompose $v$ according to the spectrum, for some real functions $\alpha_{n,k,l},$
\begin{align}
v(y,\omega,\tau)=\sqrt{6+\sum_{n,k,l}\alpha_{n,k,l}(\tau)H_{n}(y)f_{k,l}(\omega)}.\label{eq:ideas}
\end{align}  These functions satisfy the equation
\begin{align}
    \frac{d}{d\tau}\alpha_{n,k,l}=-\Big(\frac{k(k+2)}{6}+\frac{n-2}{2}\Big)\alpha_{n,k,l}+\text{NL}_{n,k,l},\label{eq:linearRate}
\end{align} where $\text{NL}_{n,k,l}$ are nonlinear in terms of $\alpha_{n,k,l}.$

We start with studying $\alpha_{n,k,l}$ with $$(n,k)=(0,0),\ (0,1),\ (1,0),\ (1,1), (2,0).$$Here $-(\frac{k(k+2)}{6}+\frac{n-2}{2})\geq 0$. This is adverse since, by (\ref{eq:linearRate}), at least on the linear level, $\alpha_{n,k,l}$ are not decaying, contrary to our expectation that all the directions decay since $\lim_{\tau\rightarrow \infty}v(y,\omega,\tau)= \sqrt{6}$.

Among those, the following ones are easier: $(n,k)=(0,0),\ (0,1)\ (1,0)$. They make $-(\frac{k(k+2)}{6}+\frac{n-2}{2})> 0$. Results from \cite{ColdingMiniUniqueness} imply that $|\alpha_{n,k,l}(\tau)|\rightarrow 0$ as $\tau\rightarrow \infty$. Rewrite (\ref{eq:linearRate}) to find that
\begin{align}
\alpha_{n,k,l}(\tau)=-\int_{\tau}^{\infty} e^{[\frac{n-2}{2}+\frac{k(k-1)}{6}](\tau-\sigma)} \text{NL}_{n,k,l}(\sigma)\ d\sigma.
\end{align}Thus these functions actually decay rapidly, provided that the nonlinear terms decay rapidly.

The ones with $(n,k)=(1,1)$ are not difficult either, since by choosing optimal tilts and centers for the coordinate system, we can make $\alpha_{1,1,l}=0.$ In fact by the identity 
$$\alpha_{1,1,l}(\tau)=-\int_{\tau}^{\infty} \text{NL}_{1,1,l}(\sigma)\ d\sigma,$$
we will prove these functions decay rapidly.

Consequently among the difficult directions, we need to focus on the function $\alpha_{2,0,1}$. To simplify the notation we define $$b(\tau):=\alpha_{2,0,1}(\tau).$$
Depending on how $b$ behaves, we have the so-called nondegenerate and degenerate cases. It satisfies the equation 
\begin{align}
\frac{d}{d\tau}b=-\frac{1}{3} b^2+\text{NL}_{2,0,0}.\label{eq:beqn}
\end{align} The equation $\frac{d}{d\tau}
b=-\frac{1}{3}b^2$ has a family solutions $b(\tau)=\frac{1}{c+\frac{1}{3}\tau}$. Based on this there are two (and only two) different cases:
\begin{itemize}
    \item[(A)] The nondegenerate case, where
    \begin{align}
 b(\tau)=3\tau^{-1}(1+o(1)),\ \text{as}\ \tau\rightarrow \infty.
 \end{align}
    
In this case our technical advantage allows us to study the rescaled MCF in the region 
          \begin{align}
           |y|\leq 10\tau^{\frac{1}{2}+\frac{1}{20}},\label{eq:nondRegion}
\end{align} and prove that the function $v$ satisfies the following estimates,
    \begin{align}
    \begin{split}\label{eq:profile}
        \Big|v-\sqrt{6+b(\tau)y^2}\Big|+
        \sum_{j+|i|=1,2}v^{j-1}|\partial_{y}^j\nabla_{E}^{i}v| \ll  \tau^{-\frac{3}{10}},
    \end{split}
    \end{align}
    where $\nabla$ denotes covariant differentiation on $\mathbb{S}^3$, and $(E_1,E_2,E_3)$ is an orthonormal basis of $T_{\omega}\mathbb{S}^3.$

  \eqref{eq:profile} is crucial for us to find a detailed description of a small neighborhood of the singularity of MCF. It implies that \begin{align}
  v(y,\omega,\tau)\gg 1,\ \text{when}\ 10\tau^{\frac{1}{2}+\frac{1}{20}}\geq |y|\gg \tau^{\frac{1}{2}},
  \end{align} in contrast, when $y=0,$ 
 \begin{align}
  v(0,\omega,\tau)=\sqrt{6}+o(1).
  \end{align} Before returning to MCF we recall that $z$ and $t$ are rescaled into $y:=\frac{z}{\sqrt{T-t}}=e^{\frac{1}{2}\tau}z$ and $\tau(t):=-\ln(T-t)$. Thus for each fixed small $z_0\not=0$, there exists a unique time $\tau_1$ such that,  $$|y_1|=e^{\frac{1}{2}\tau_1}|z_1|=\tau_1^{\frac{1}{2}+\frac{1}{20}}.$$
  We are ready to study MCF. By \eqref{eq:profile} and \eqref{eq:uvidentity},
  \begin{align}
      u(z_1,\omega,t(\tau_1))=\sqrt{T-t(\tau_1)} v(y_1, \omega,\tau_1)=\sqrt{T-t(\tau_1)} \tau_1^{\frac{1}{20}}\Big(1+o(1)\Big) ,
  \end{align} in contrast, recall that the MCF blows up at $z=0$ and the time $T$,
  \begin{align}
      u(0,\omega,t(\tau_1))=\sqrt{6}\sqrt{T-t(\tau_1)} \Big(1+o(1)\Big),
  \end{align} hence, recall that $\tau_1\gg 1,$
  \begin{align}
      \frac{u(0,\omega,t(\tau_1))}{u(z_1,\omega,t(\tau_1))}\ll 1.
  \end{align}
  These, together with the smoothness estimates provided by \eqref{eq:profile} and the identity
  \begin{align}
      v^{k-1}\partial_{y}^{j} \nabla_{E}^{i} v(y,\omega,\tau)=u^{k-1}\partial_{z}^{j} \nabla_{E}^{i} u(z,\omega, t),
  \end{align} and 
  the well known techniques of local smooth extension, see e.g. \cite{EckerBook}, imply that when the MCF forms a singularity at $z=0$, for any small $z\not=0$, a small neighborhood remains smooth.

 \item[(B)] If the first possibility does not work, then in the same region 
 the function $v$ is of the form
    \begin{align}
        v(y,\omega,\tau)=\sqrt{6+\eta(y,\omega,\tau)}
    \end{align}
   and $\eta$ is small:
    \begin{align}\label{eq:prenondege}
              \begin{split}
  \sum_{j+|k|\leq 2}\Big\|e^{-\frac{1}{8}y^2}\partial_{y}^{j}\nabla_{E}^{k}\eta(\cdot,\tau)\Big\|_{2}\lesssim & e^{-\frac{3}{10}\tau},\\
|\eta(y,\omega,\tau)|\leq & \tau^{-\frac{3}{10}}.
          \end{split}
\end{align}  

\end{itemize}

The results above are a part of the main Theorem \ref{THM:lIS1}.

We remark that instead of considering the region $|y|\leq 10\tau^{\frac{1}{2}+\frac{1}{20}}$, one can consider a neighborhood $|y|\leq M\tau^{\frac{1}{2}+\epsilon}$ for any sufficiently small $\epsilon>0$ and large $M.$ We choose $M=10$, $\epsilon=\frac{1}{20}$ so that we will not have too many statements like ``for some sufficiently constant $\epsilon$". In the present paper $10$ or $20$ signifies something sufficiently large.

We continue studying the degenerate cases. Here we have to limit our consideration to Type I blowup, specifically, let $A(p,t)$ be the second fundamental form at the point $p$ and time $t$, we require that, for some constant $c>0,$
\begin{align}
\max_{p}|A(p,t)|\leq \frac{c}{\sqrt{T-t}}.\label{eq:secFunForm1}
\end{align} This will be needed in the discussion around (\ref{eq:rescFunForm}) below.

We start with simplifying the problem by arguing that we only need to focus on the directions $H_m,$ $m\geq 3$ and $H_{n}\omega_l$, $n\geq 2$. It is easy to see the reason, when $k\geq 2$, \eqref{eq:ideas} implies that the linear decay rates of $\alpha_{n,k,l}$'s are not slower than $e^{-(\frac{n}{2}+\frac{1}{3})\tau},$ thus they play a diminishing role in MCF: if $|\alpha_{n,k,l}(\tau)|\lesssim e^{-(\frac{n}{2}+\frac{1}{3})\tau}$, then uniformly in the $z=e^{-\frac{1}{2}\tau}y$ variable, $$\Big| \alpha_{n,k,l}(\tau)H_{n}(y)f_{k,l}(\omega) \Big|\leq e^{-\frac{1}{3}\tau}\Big|f_{k,l}(\omega)\Big( z^n+\sum_{k=1}^{\floor*{\frac{n}{2}}} e^{-k\tau}h_{n-2k,n}z^{n-2k}\Big)\Big|\rightarrow 0,\ \text{as}\ \tau\rightarrow \infty,$$
where, the constants $h_{n-2k,n}$ are from \eqref{eq:hermit}.

The problem can be simplified one more time by observing that, by parameterizing the rescaled MCF properly, the parts $\alpha_{n,1,l}H_{n}\omega_l$ can be removed.
And moreover it is necessary to remove them. To illustrate this by an example, suppose that $v$ is of the form
\begin{align}
    v(y,\omega,\tau)=\sqrt{6+\alpha_{2,1,1}(\tau) H_{2}(y)\omega_1+ \alpha_{6,0,1}(\tau) H_{6}(y)+\text{OtherTerms}}\label{eq:badexam}
\end{align} and $\alpha_{2,1,1}$, $\alpha_{6,0,1}$ decay at the rates predicted by (\ref{eq:linearRate}), i.e. for some nonzero constants $c_2$ and $c_6$,
\begin{align*}
    \alpha_{2,1,1}(\tau)= c_2 e^{-\frac{1}{2}\tau} \Big(1 +o(1)\Big),\\
    \alpha_{6,0,1}(\tau)=c_6 e^{-2\tau}\Big(1+o(1)\Big).
\end{align*} The difficulty is that it is impossible to build a profile similar to that in \eqref{eq:nondRegion} and \eqref{eq:profile}: when $|\alpha_{3,0,1}H_6(y)|\gg 1$, $y$ must satisfy $|y|\gg e^{\frac{1}{3}\tau}$, then  $\alpha_{2,1,1}H_2 \omega_1$ can be even larger, adversely!

It turns out, see e.e. (\ref{eq:2repara}) below, to remove the main part of $\alpha_{n,1,l},$ it suffices to parametrize the rescaled MCF as follows: for some polynomial-valued vector $Q_{N}(y,\tau)\in \mathbb{R}^4$
\begin{align}
Q_{N}(y,\tau)=\sum_{n=2}^{N}H_{n}(y) e^{-\frac{n-1}{2}\tau}\Big(a_{n,1},a_{n,2},a_{n,3},a_{n,4}
\Big)^{T}
\end{align} with $a_{k,l}$ being constants, we parametrize the rescaled MCF as
\begin{align}
\Psi_{Q_{N},v}
:=\left[
\begin{array}{ccc}
y\\
Q_{N}(y,\tau)
\end{array}
\right]+v(y,\omega,\tau)\left[
\begin{array}{ccc}
-\partial_{y}Q_{N}(y,\tau)\cdot\omega\\
\omega
\end{array}
\right].
\end{align}

Based on the reasons presented above, we will look for dominating directions among $H_{n},\ n\geq 3$, and in the process remove the main part of the directions $H_{k} \omega_l,$ $2\leq k\leq n-1$ and $l=1,2,3,4,$ by finding certain vector-valued functions $Q_{n-1}\in \mathbb{R}^4$ in $\Psi_{Q_{n-1},v}$.

To illustrate the ideas, we present our strategy for studying the directions $H_3$ and $H_4$. We have to study the direction $H_3$ first, then $H_4,$ by reason similar to what was discussed around (\ref{eq:badexam}).

We start with studying the direction $H_3$ in the region 
\begin{align}
|y|\leq R(\tau):=10\tau^{\frac{1}{2}+\frac{1}{20}},
\end{align} so that we can use the results in (\ref{eq:prenondege}), also because we are not ready to study a much larger region before proving that what is in the square root is nonnegative.

For the rescaled MCF $\Psi_{0,v}$,
we decompose $v$ such that
\begin{align}
v(y,\omega,\tau)=\sqrt{6+\alpha_{3,0,1}(\tau)H_3(y)+\sum_{n=0}^2 \sum_{k=0}^N \sum_l \alpha_{n,k,l}(\tau) H_ n(y)f_{k,l}(\omega)+\eta(y,\omega,\tau)}\label{eq:trialVdecom}
\end{align}
where, $N$ is a large integer, and we prove that, for any $\delta_0>0,$ there exists a $C_{\delta_0}$ such that
\begin{align}
\begin{split}
\Big\|1_{\leq 5\tau^{\frac{1}{2}+\frac{1}{20}}}\eta(\cdot,\tau)\Big\|_{\mathcal{G}}\leq & C_{\delta_0} e^{-(1-\delta_0)\tau},
\end{split}
\end{align}
and the governing equations for $\alpha_{n,k,l}$ take the forms
\begin{align}
\begin{split}
(\frac{d}{d\tau}+\frac{n-2}{2}+\frac{k(k+2)}{6})\alpha_{n,k,l}=&\mathcal{O}(e^{-\tau}), \ n=1,2,3,\\
(\frac{d}{d\tau}+\frac{k(k+2)}{6})\alpha_{0,k,l}=&\mathcal{O}(e^{-\frac{2}{3}\tau}).
\end{split}
\end{align} When $(n,k)=(3,0),\ (2,1)$, they can be rewritten into
\begin{align}
\alpha_{n,k,l}(\tau)=& e^{-\frac{1}{2}\tau}\Big[\alpha_{n,k,l}(0)+\int_{0}^{\tau}e^{\frac{1}{2}\sigma}\mathcal{O}(e^{-\sigma})\ d\sigma\Big]\nonumber\\
=&d_{n,k,l} e^{-\frac{1}{2}\tau}-e^{-\frac{1}{2}\tau}\int_{\tau}^{\infty} e^{\frac{1}{2}\sigma} \mathcal{O}(e^{-\sigma})\ d\sigma\nonumber\\
=&d_{n,k,l} e^{-\frac{1}{2}\tau}+\mathcal{O}(e^{-\tau})\label{eq:leadPart}
\end{align} where $d_{n,k,l}$ are constants defined as $d_{n,k,l}:=\alpha_{n,k,l}(0)+\int_{0}^{\infty}e^{\frac{1}{2}\sigma}\mathcal{O}(e^{-\sigma})\ d\sigma.$ 

For the others, if $\frac{n-2}{2}+\frac{k(k+2)}{6}\leq 0$, then since $\lim_{\tau\rightarrow \infty}\alpha_{n,k,l}(\tau)=0$, 
\begin{align}
    \alpha_{n,k,l}(\tau)=-\int_{\tau}^{\infty}e^{-(\frac{n-2}{2}+\frac{k(k+2)}{6})(\tau-\sigma)}\ \mathcal{O}(e^{-\sigma})\ d\sigma=\mathcal{O}(e^{-\tau});
\end{align} if $\frac{n-2}{2}+\frac{k(k+2)}{6}>0$, then we compute directly.
Consequently, besides \eqref{eq:leadPart}, we obtain
\begin{align}
\begin{split}
   e^{\frac{5}{6}\tau}|\alpha_{1,k,l}(\tau)|+e^{\frac{\tau}{3}}|\alpha_{0,k,l}(\tau)|\lesssim & 1 \\
   |\alpha_{2,k,l}(\tau)|\lesssim & e^{-\tau}, \ \text{if}\ k\not=1.
   \end{split}
\end{align}

 Now we remove the part $\sum_{l=1}^4 d_{2,1,l}e^{-\frac{1}{2}\tau}H_2 \omega_l,$ which is the main part on the direction $H_2 \omega_l,$ by reparametrize the rescaled MCF as $\Psi_{Q_2,v}$, where $Q_2\in \mathbb{R}^4$ takes the form, for some constants $a_l$, 
\begin{align}\label{eq:2repara}
Q_2(y,\tau)=H_{2}(y) e^{-\frac{1}{2}\tau}\Big(a_{1},a_{2},a_{3},a_{4}
\Big)^{T},
\end{align} and $v$ takes the form
\begin{align}
v(y,\omega,\tau)=\sqrt{6+\tilde\alpha_{3,0,1}(\tau)H_3(y)+\sum_{n=0}^2 \sum_{k=0}^{N}\sum_l \tilde\alpha_{n,k,l}(\tau)H_n(y) f_{k,l}(\omega)+\tilde{\eta}(y,\omega,\tau)}
\end{align}  $N$ is a large integer, and after reasoning as in \eqref{eq:leadPart} we prove, for some $d_3\in \mathbb{R}$,
\begin{align}
\tilde\alpha_{3,0,1}(\tau)=d_3 e^{-\frac{1}{2}\tau}+\mathcal{O}(e^{-\tau}).\label{eq:d3First}
\end{align} The others satisfy the estimates,
\begin{align}
    e^{\tau}|\tilde\alpha_{2,k,l}(\tau)|+e^{\frac{5}{6}\tau}|\tilde\alpha_{1,k,l}(\tau)|+e^{\frac{\tau}{3}}|\tilde\alpha_{0,k,l}(\tau)|\lesssim 1,\label{eq:trialAlpha}
\end{align} and for any $\delta_0>0,$
\begin{align}
\Big\|1_{\leq 10\tau^{\frac{1}{2}+\frac{1}{20}}}\tilde\eta(\cdot,\tau)\Big\|_{\mathcal{G}}\leq & C_{\delta_0} e^{-(1-\delta_0)\tau}.\label{eq:esttildeta}
\end{align}

Next we prove $d_3=0$, in order not to contradict the condition (\ref{eq:secFunForm1}). Suppose that $d_3\not=0,$ then the cases for $d_3>0$ and $d_3<0$ can be considered identically. Without loss of generality we assume $d_3>0$. Then our techniques, specifically sufficiently sharp estimates on certain weighted $L^{\infty}-$norms of $\tilde\eta$, allow us to study the following region, for any small positive constant $\epsilon_0,$ 
\begin{align}\label{eq:d3Region}
\Big\{\ y\ \Big|\ - \Big( \frac{6-\epsilon_0}{d_3}\Big)^{\frac{1}{3}}(1+\epsilon_0^2) \leq ye^{-\frac{1}{6}\tau} \leq - \Big( \frac{6-\epsilon_0}{d_3}\Big)^{\frac{1}{3}}(1-\epsilon_0^2)\Big\}
\end{align} and prove that,
\begin{align}
\begin{split}
v(y,\omega,\tau)=\sqrt{6+d_3 e^{-\frac{1}{2}\tau }y^3}\Big(1+\mathcal{O}&(e^{-\sqrt{\tau}})\Big)=\sqrt{\epsilon_0}\Big(1+o(\epsilon_0)\Big),\\
\sum_{j+|i|=1,2}\Big|v^{j-1}\partial_{y}^j\nabla_{E}^i v(y,\omega,\tau)\Big|\leq & e^{-\sqrt{\tau}}.
\end{split}
\end{align} Since, in the considered region, $v$ can be interpreted as the radius of the cylinder, these estimates imply that the second fundamental form $\tilde{A}(y,\omega,\tau)$ of the rescaled MCF is not uniformly bounded, 
\begin{align}\label{eq:rescFunForm}
\sup_{y,\omega}|\tilde{A}(y,\omega,\tau)|\rightarrow \infty\ \text{as}\ \epsilon_0\rightarrow 0.
\end{align}
This contradicts to the condition (\ref{eq:secFunForm1}) for MCF. Consequently $d_3$ must be zero!

If $d_3=0$ we continue to study the direction $H_4$. Initially we study the old parametrization $\Psi_{Q_2,v}$. After going through the process between \eqref{eq:trialVdecom} and \eqref{eq:esttildeta}, we parametrize the rescaled MCF as $\Psi_{Q_3,v}$
where $Q_3$ is of the form
\begin{align}
    Q_{3}(y,\tau)=\sum_{n=2}^{3}H_{n}(y) e^{-\frac{n-1}{2}\tau}\Big(a_{n,1},a_{n,2},a_{n,3},a_{n,4}
\Big)^{T}
\end{align} for some constants $a_{n,k}$, and compared to \eqref{eq:2repara}, $a_{2,l}=a_{l}$. The function $v$ is of the form
\begin{align}
    v(y,\omega,\tau)=\sqrt{6+ \beta_{4,0,1}(\tau)H_4(y)+\sum_{n=0}^3\sum_{k=0}^{N}\sum_l \beta_{n,k,l}(\tau) H_n(y) f_{k,l}(\omega)+\xi(y,\omega,\tau)}
\end{align}
where, $N$ is a large integer, and for some $d_4\in \mathbb{R}$,
\begin{align}
    \beta_{4,0,1}(\tau)=d_4 e^{-\tau}+\mathcal{O}(e^{-\frac{3}{2}\tau}),
\end{align} 
for the other $\beta_{n,k,l}$, 
\begin{align}
e^{\tau}|\beta_{2,k,l}(\tau)|+ e^{\frac{3}{2}\tau}|\beta_{3,k,l}(\tau)|+
e^{\frac{5}{6}\tau} |\beta_{1,k,l}|+ e^{\frac{\tau}{3}}|\beta_{0,k,l}|\lesssim 1,
\end{align} and for the remainder $\xi,$
\begin{align}
\|e^{-\frac{1}{8}y^2}\xi(\cdot,\tau)\|_{2}\leq & e^{-(\frac{3}{2}-\frac{1}{4})\tau}.
\end{align}

Now we discuss the possible signs of the constant $d_4$.

By the same strategy of ruling out the possibility $d_3\not=0$ we rule out the possibility that $d_4<0.$ 

If $d_4>0$, then similar to the nondegenerate case, we can build a profile similar to that in \eqref{eq:profile}, and from there we study MCF.
Specifically, we use our technical advantage to study the region 
\begin{align}
|y|\leq Z_{4}(\tau):=e^{\frac{1}{4}\tau}e^{10\sqrt{\tau}},
\end{align} and prove that $v$ satisfies the estimates
\begin{align}\label{eq:largeV}
\begin{split}
    v(y,\omega,\tau)=&\sqrt{6+d_4 e^{-\tau} y^4}\Big(1+o(e^{-\sqrt{\tau}})\Big),\\
    \sum_{j+|l|=1,2}&v^{j-1}\Big|\partial_{y}^{j}\nabla_{E}^{l}v(y,\omega,\tau)\Big|\leq  e^{-\sqrt{\tau}}.
\end{split}
\end{align} 
From here we study the MCF, very similar to the nondegenerate case. For the details, see Section \ref{sec:step1Third} below.

If $d_4=0$ we continue to study the direction $H_{5}$. The treatment is similar to $H_3$ since the corresponding constant $d_5$ must be zero. After that we study $H_6$, and so on.

By the ideas presented above, to exhaust all the possibilities, it is better to use induction. This is how we will formulate Theorem \ref{THM:sec} below.

Next we discuss the possibility that $d_m=0$ for all $m\geq 3$. We will formulate a conjecture and want to convince the readers that it is reasonable.

The main argument is that, by the induction process in Theorem \ref{THM:sec}, there exists some $Q_{m-1}$ of the form, for some constants $a_{k,l},$
\begin{align}
Q_{m-1}(y,\tau)=\displaystyle\sum_{k=2}^{m-1}e^{-\frac{k-1}{2}\tau}H_{k}(y)\Big(a_{k,1},\ a_{k,2},\ a_{k,3},\ a_{k,4}\Big)^{T},
\end{align}
and a function $v_{m}$, such that if the rescaled MCF is parametrized as $\Psi_{Q_{m-1},v_m}$,
then $v_{m}$ takes the following form, recall that $z$ of MCF and $y$ of rescaled MCF are related by $z=e^{-\frac{1}{2}\tau} y,$
\begin{align}
\begin{split}
v_m(y,\omega,\tau)=&\sqrt{6+\beta_{m}(e^{-\frac{1}{2}\tau}y,\omega)+Re_{m}(y,\omega,\tau)+\eta_m(y,\omega,\tau)}
\end{split}
\end{align} where $\beta_{m}$ is considered the main part, for some constants $\beta_{n,k,l},$ independent of $m$, and integer $N$,
\begin{align}
    \beta_{m}(e^{-\frac{1}{2}\tau} y ,\omega)=\sum_{n=2}^{m}\sum_{k\leq N}\sum_l\beta_{n,k,l}e^{-\frac{n}{2}\tau} y^n f_{k,l}(\omega)
\end{align} 
and $Re_{m}$ is of the form, for some uniformly bounded functions $\alpha_{n,k,l},$
\begin{align}
Re_{m}(y,\omega,\tau)=e^{-\frac{1}{3}\tau}\sum_{n=0}^{m}\sum_{k=0}^{N}\sum_l \alpha_{n,k,l}(\tau) e^{-\frac{n}{2}\tau} y^n  f_{k,l}(\omega)
\end{align}
and $\eta_m$ satisfies the estimate,
\begin{align}
\|e^{-\frac{1}{8}y^2}\eta_m(\cdot,\tau)\|_{2}\lesssim e^{-\frac{m-1}{2}\tau}.
\end{align}

It is important to point out that $\beta_m$ and $Re_m$ are uniformly bound functions of variable $z$ when $|z|\leq \epsilon$. And $\beta_m$ is generated, directly or indirectly, by $Q_{m-1}$: if $Q_{m-1}\equiv 0$, then $\beta_{m}\equiv 0;$ and moreover the governing equation $v$ reads, for some constants $c_1$ and $c_2,$
\begin{align}
\partial_{\tau}v=-Lv+c_1 v^{-1} |\partial_{y}Q_{m}|^2+c_2 v^{-1}(\omega\cdot \partial_{y}Q_{m})^2+\text{OtherTerms}.\label{eq:qmdnkl}
\end{align} Through this, the parts $|\partial_{y}Q_{m}|^2$ and $(\omega\cdot \partial_{y}Q_{m})^2$ produce $\beta_m$. These two parts are from the terms $J_7$ and $J_8$, defined in \eqref{def:J7J8}, in the equation for $v^2$ below.

Also for the vector $Q_{m-1}$, $e^{-\frac{1}{2}\tau}Q_{m-1}(y,\tau)$ is uniformly bounded in the variable $z$ when $|z|\leq \epsilon .$

We conjecture that, for some $\epsilon>0,$ in the region $|z|\leq e^{-\frac{1}{2}\tau}|y|\leq \epsilon$, the following limits exist: for some smooth functions $Q_{\infty}$ and $v_{\infty}$,
\begin{align}\label{eq:PiInf}
\begin{split}
Q_{\infty}(y,\tau)=&\lim_{m\rightarrow \infty}Q_{m-1}(y,\tau),\\
v_{\infty}(y,\omega,\tau)=& \lim_{m\rightarrow \infty}v_{m-1}(y,\omega,\tau).
\end{split}
\end{align}
If the limits hold, then they imply that a whole neighborhood of $z=0$ of MCF will become singular at the time $T$.
 
It seems the conjecture (\ref{eq:PiInf}) is natural since one obtains $Q_{\infty}$ by removing $\sum_{n=2}^{\infty}\sum_{l=1}^4\alpha_{n,1,l} H_n \omega_l$ from $v$ in \eqref{eq:ideas}, hence probably it is of finite energy in some sense, and then the main parts of $v_{\infty}$ is generated by $Q_{\infty}$, as discussed above.
Moreover our conjecture is consistent with the known results for nonlinear heat equations, see Theorem \ref{THM:val} below.

Now we discuss our techniques. Our main technical advantage is to control the solution by the norms $\|(1+| y|)^{-m}\cdot\|_{\infty}, \ m\geq 0$. Compare to the previously widely used norm $\|e^{-\frac{1}{8}|y|^2}\cdot \|_2$, we can retrieve information of $v$ for large $|y|$. We can use these norms because we apply propagator estimates, see Lemma \ref{LM:frequencyWise} below. It is also important to make normal form transformations to remove the directions $H_n\omega_l$, as discussed above.

In what follows we compare our results to the known ones.

We need ideas of Herrero and Vel\'{a}zquez in \cite{HerrVel1993, herrero1992, Vel1993}, where they studied Type-I blowup of nonlinear heat equations (NLH)
\begin{align}
\partial_{t}u=\Delta u+|u|^{p}u.
\end{align}The case $p=2$ is the most relevant, since a transformation $u\rightarrow u^{-1}$ will make the equation very similar to that of spherically symmetric MCF. 

Now we discuss their results for one dimensional NLH with $p=2$. Suppose the solution blows up at time $T$ and $z=0$, and the blowup is Type I, then define a new function $v$ by 
\begin{align}
u(z,t)=(T-t)^{-\frac{1}{2}}v(y,\tau),\ \text{with} \ y:=(T-t)^{-\frac{1}{2}}z,\ \text{and} \ \tau:=-\ln(T-t).
\end{align} Here $v=\frac{1}{\sqrt{2}}$ is a static solution. Define a function $w$ by $v=\frac{1}{\sqrt{2+w}}$ and derive an equation for it
\begin{align}
\partial_{\tau}w=-(L_0-1) w+NonLinearity(w)\label{eq:linearize}
\end{align} with the linear operator $L_0$ defined in \eqref{def:linearL}.
Decompose $v$ according to the spectrum,
\begin{align}\label{eq:decomNLH}
v(y,\tau)=\Big(2+\sum_{n=0}^{\infty} \alpha_{n}(\tau)H_{n}(y)\Big)^{-\frac{1}{2}}.
\end{align} 
The following result was proved in \cite{HerrVel1993, herrero1992}: see Theorems A and 1 in \cite{herrero1992},
\begin{theorem}\label{THM:val}
The solution \eqref{eq:decomNLH} must is one of the following three possibilities:
\begin{itemize}
\item[(1)] $\alpha_2(\tau)=\tau^{-1}(1+o(1))$ as $\tau\rightarrow\infty$, and for any other $k$, $|\alpha_k(\tau)|=\mathcal{O}(\tau^{-2});$
\item[(2)] there exists an integer $m\geq 2$ such that $H_{2m}$ dominates, specifically for some $d_{2m}>0$, 
\begin{align}
\alpha_{2m}(\tau)=d_{2m} e^{-(m-1)\tau}(1+o(1)),
\end{align} and
\begin{align}
\alpha_{n}(\tau)=o(e^{-(m-1)\tau})\ \text{if}\ n\not=2m,
\end{align}
\item[(3)] $v=\frac{1}{\sqrt{2}}$, i.e. $c_{n}\equiv 0$ for all $n\geq 0.$
\end{itemize}
\end{theorem}
We will use their ideas, for example, in looking for the main part of a direction in \eqref{eq:leadPart}. We expect that the same results will hold for rotationally symmetric MCF.

However MCF differs significantly from NLH. Specifically, Theorem \ref{THM:val} says that if there is no dominating direction, then $v$ must be homogeneous in $y$. This is not true for MCF, for example, if the initial hypersurface is $\mathbb{S}^{3}_{q}\otimes \mathbb{S}_{\sqrt{2}}$ for some $q\gg 1$, where $q$ and $\sqrt{2}$ are radius of $\mathbb{S}^3$ and $\mathbb{S}$, then the blowup takes place everywhere at the same time.
Consequently for MCF it is important to find a proper ``curved space'' and parametrize the hypersurfaces on it. 

For MCF, Choi, Haslhofer and Hershkovits proved the mean convexity conjecture for hypersurfaces in $\mathbb{R}^3$ in \cite{choi1810ancient}, and Choi, Haslhofer, Hershkovits and White proved the mean convexity when the limit cylinder of rescaled MCF is $\mathbb{R}^{l}\times \mathbb{S}^n$, $l=0,1$, in \cite{choi2019ancient}. For the related work, see also \cite{HuiskenSur2009, BrHui2016, MR3602529, MR3662439}. For the other related works, see \cite{Hamilton1997, AltAngGiga1995, sesum2008, AnDaSE15, AnDaSE18}.
Compared to these works, in the present paper we obtained a detailed description for the various cases, besides mean convexity, we can prove that the singularity is isolated. Another highlight is that our techniques are applicable for the cases where the limit cylinder is $\mathbb{R}^{k}\times \mathbb{S}^{l}$, $k\geq 2$, which is to be discussed below. However our result is less complete in the sense that we do not have a complete picture, instead, for some cases we only have a conjecture.

In our earlier paper \cite{GZ2018}, we considered only the nondegenerate case when the limit cylinder is $\mathbb{R}^{3}\times \mathbb{S}$. The corresponding degenerate cases will be more involved than the present problem because we need to consider more possibilities. For example, the scalar function $b$ in \eqref{eq:beqn} becomes a $3\times 3$-matrix $B(\tau)$ in \cite{GZ2018} and we proved that it takes the form
\begin{align}\label{eq:matrixB}
    B(\tau)=\tau^{-1}\left[
    \begin{array}{lll}
    c_1&0&0\\
    0&c_2&0\\
    0&0&c_3
    \end{array}
    \right]+\mathcal{O}(\tau^{-2}),\ c_{k}=0\ \text{or}\ 1.
\end{align} In \cite{GZ2018} we only considered the case $c_1=c_2=c_3=1$. The degenerate cases includes many new cases, for example $c_1=1$ and $c_2=c_3=0$, which require interpretation and are not problems in the present consideration. 

We will address the degenerate cases of the regimes where the singularity is modeled on $\mathbb{R}^{k}\times \mathbb{S}^{l},\ k\geq 2,$ in subsequent papers.

The paper is organized as follows. The main theorems are stated in Section \ref{sec:MTHM}. In Theorem \ref{THM:lIS1} we discuss the nondegenerate case, and in Theorem \ref{THM:sec} we study the degenerate cases, and we formulate the results by induction for the reasons presented earlier. Theorem \ref{THM:lIS1} will be proved in Section \ref{sec:MainTHM1}. Theorem \ref{THM:sec} will be proved in subsequent sections.

In the present paper we use the following conventions. $A\lesssim B$ signify that there exists a universal constant $C>0$ such that $A\leq CB$, and we define $\langle y\rangle^k, \ k\in \mathbb{R},$ as  $$\langle y\rangle^{k}:=(1+|y|^2)^{\frac{k}{2}}.$$ The inner product $\langle \cdot,\cdot \rangle_{\mathcal{G}}$ is defined as,
for any functions $f$ and $g$
\begin{align}
\langle f,\ g\rangle_{\mathcal{G}}:=\int_{0}^{2\pi}\int_{\mathbb{R}^3} e^{-\frac{1}{4}|y|^2} f(y,\omega) \overline{g}(y,\omega)\ d^3y dS(\omega),\label{eq:GInner}
\end{align} and we use the notation $f\perp_{\mathcal{G}} g$ to signify that $\langle f,\ g\rangle_{\mathcal{G}}=0,$ and from this we define the norm $\|\cdot\|_{\mathcal{G}}$.

\section{Main Theorem}\label{sec:MTHM}
We are interested in the cases where the initial hypersurface $\Sigma_0$ satisfies 
\begin{align}\label{eq:generic5}
\lambda(\Sigma_0)<\infty
\end{align} where the functional $\lambda$ was defined in \eqref{eq:generic}.  Suppose the blowup time is $T$ and blowup point is the origin, 
then Colding and Minicozzi proved in \cite{ColdingMiniUniqueness} that the rescaled MCF $\frac{1}{\sqrt{T-t}}X_{t}$ will converge to a unique cylinder, i.e. up to a rotation, for some $k=1,\cdots,n,$
\begin{align}
\frac{1}{\sqrt{T-t}}X_{t}\rightarrow \mathbb{S}^{k}_{\sqrt{2k}}\times \mathbb{R}^{n-k},\ \text{as}\ t\rightarrow T.
\end{align}

In the present paper we consider the case $k=3$ and $n-k=1$, i.e. the limit cylinder is $\mathbb{S}^{3}_{\sqrt{6}}\times \mathbb{R}$.
\cite{ColdingMiniUniqueness} implies that there exists some positive function $u: \mathbb{R}\times \mathbb{S}^3\times [0,
\ T)\rightarrow \mathbb{R}^{+}$ such that, in a (possibly shrinking) neighborhood of the origin, the MCF and the rescaled MCF can be parametrized as, 
\begin{align}\label{eq:origPara10}
\Psi(z,\omega, t)= \left[
\begin{array}{ccc}
z\\
u(z, \omega, t)\omega
\end{array}
\right]:=\sqrt{T-t} \left[
\begin{array}{ccc}
y\\
v(y, \omega,\tau)\omega
\end{array}
\right],
\end{align}
where $\omega\in \mathbb{S}^3,$ and $y$ and $\tau$ are new spatial and time variables defined as
\begin{align}
y:=\frac{z}{\sqrt{T-t}} &\ \text{and}\ \tau:=|ln(T-t)|,
\end{align}
and the function $v$ is defined as\begin{align}
v(y,\omega,\tau)&:=\frac{1}{\sqrt{T-t}}u(z,\omega,t).
\end{align}
The results in \cite{ColdingMiniUniqueness} imply that
\begin{align}\label{eq:fixedYconv}
\text{for any fixed} \ y, \ v(y,\omega,\tau)\rightarrow \sqrt{6} \ \text{as} \ \tau\rightarrow \infty.
\end{align}

As discussed earlier, to identify the dominant direction we need some normal form transformation to parametrize the MCF. Specifically, let $\Pi_{N}$ be a vector-valued function of the form
\begin{align}\label{eq:generalP}
\Pi_{N}(z,t)= \sum_{n=2}^{N}H_{n}(\frac{z}{\sqrt{T-t}}) (T-t)^{\frac{n}{2}}\Big(a_{n,1},a_{n,2},a_{n,3},a_{n,4}
\Big)^{T}
\end{align}where $N\geq 2$ is a natural number, $H_n$ is the $n-$the Hermite polynomial, and $a_{n,k}\in \mathbb{R}$ are constants, then in a possibly shrinking neighborhood of the origin, we can parametrize the MCF as, 
\begin{align}\label{eq:parametrization}
\Phi_{\Pi_{N},u} (z,\omega,t)
=\left[
\begin{array}{ccc}
z\\
\Pi_{N}(z,t)
\end{array}
\right]+u(z,\omega,t)\left[
\begin{array}{ccc}
-\partial_{z}\Pi_{N}(z,t)\cdot\omega\\
\omega
\end{array}
\right],
\end{align} where $u$ is a positive function, and $\omega\in \mathbb{S}^3$. 

The corresponding rescaled MCF, denoted by $\Psi_{Q_{N},v}$,
\begin{align}\label{def:PsiQv}
\Psi_{Q_{N},v}(y,\omega,\tau):=\frac{1}{\sqrt{T-t}}\Phi_{\Pi_{N},u},
\end{align} takes the form
\begin{align}\label{eq:norPara}
\Psi_{Q_{N},v}(y,\omega,\tau)&=\left[
\begin{array}{ccc}
y\\
Q_{N}(y,\tau)
\end{array}
\right]+v(y,\omega,\tau)\left[
\begin{array}{ccc}
-\partial_{y}Q_{N}(y,\tau)\cdot\omega\\
\omega
\end{array}
\right].
\end{align} Here $y,$ $\tau$ and $v$ are defined in terms of $x,$ $t$ and $u$ by the identities,
\begin{align}\label{eq:ytau}
\begin{split}
y:=\frac{z}{\sqrt{T-t}},\ & \ \text{and}\ \tau:=|ln(T-t)|,\\
v(y,\omega,\tau):=& \frac{1}{\sqrt{T-t}}u(z,\omega,t),
\end{split}
\end{align}
and the vector-valued function $Q_{N}$ is defined in terms of $\Pi_{N}$ in \eqref{eq:generalP} by the identity
\begin{align}\label{eq:QPIdentity}
Q_{N}(y,\tau):=\frac{1}{\sqrt{T-t}}\Pi_{N}(z,t)= \sum_{n=2}^{N}H_{n}(y) e^{-\frac{n-1}{2}\tau}\Big(a_{n,1},a_{n,2},a_{n,3},a_{n,4}
\Big)^{T}.
\end{align}

Now we are ready to state the first result.
\begin{theorem}\label{THM:lIS1}
Suppose the initial hypersurface satisfies the condition \eqref{eq:generic5}, and the limit cylinder of the rescaled MCF is $\mathbb{R}\times\mathbb{S}^3_{\sqrt{6}}$. Then one and only one of the following two possibilities must hold.

\begin{itemize}
\item[(1)] For the first possibility, the rescaled MCF is parametrized as $\Psi_{0,v}$ in the region 
\begin{align}
 \Big\{y\ \Big|\ |y|\leq 10\tau^{\frac{1}{2}+\frac{1}{20}}\Big\},
\end{align} 
for some the function $v$ of the form 
\begin{align}
v(y,\omega,\tau)=\sqrt{6+\sum_{n=0}^2\alpha_n(\tau)H_n(y)+\sum_{k=0,1}\sum_{l=1}^4\alpha_{k,l}(\tau)H_{k}(y)\omega_l+\eta(y,\ \omega,\tau)},\label{eq:firstCase}
\end{align} where the functions $a_n$ and $a_{k,l}$ satisfy the estimates
\begin{align}\label{eq:firstB}
   \begin{split}
    \Big|\alpha_2(\tau)-3\tau^{-1}\Big|+\sum_{n=0,1}|\alpha_n(\tau)|+\sum_{k=0,1}\sum_{l=1}^4|\alpha_{k,l}(\tau)|\lesssim & \tau^{-2},
 \end{split}
\end{align} and the remainder $\eta$ satisfies the estimates,
\begin{align}
\Big\|\eta(\cdot,\tau) 1_{|y|\leq 10\tau^{\frac{1}{2}+\frac{1}{20}}}\Big\|_{\mathcal{G}}\lesssim & \tau^{-2},
\end{align}and
          \begin{align}
\Big|v(y,\omega,\tau)-\sqrt{6+3\tau^{-1}y^2}\Big|+\sum_{j+|i|=1,2}v^{j-1}\Big|\partial_{y}^{j}\nabla_{E}^i v(y,\omega,\tau)\Big|\leq &  \tau^{-\frac{3}{10}}.
\end{align}

For the corresponding MCF, a small neighborhood of singularity is mean convex and the singularity is isolated. Moreover for any small $\epsilon>0$, in the space-and-time region $$0\leq T-t\leq \epsilon\ \text{and}\ |z|\leq \epsilon,$$ $u$ satisfies the estimate
\begin{align}\label{eq:epsiMT}
    u(z,\omega,t)>0 \ \text{if}\ (z,t)\not=(0,T),
\end{align} for any positive integer $M,$ there exists some constant $\delta_{M}(\epsilon),$ with $\lim_{\epsilon\rightarrow 0^{+}}\delta_{M}(\epsilon)=0,$ s.t.
\begin{align}\label{eq:deriUMT}
    \sum_{1\leq j+|i|\leq M}\Big|u^{j-1}\partial_{z}^{j}\nabla_{E}^{i}u(z,\omega,t)\Big|\leq \delta_{M}(\epsilon).
\end{align}

\item[(2)] For the second possibility, the rescaled MCF is parametrized by $\Psi_{Q_2,v}$ in the region
\begin{align}
\Big\{ y\ \Big| \   |y|\leq 8\tau^{\frac{1}{2}+\frac{1}{20}}\Big\},
\end{align} $Q_2$ is a vector-valued polynomial defined as, for some constants $a_k,$
\begin{align}
Q_2(y,\tau)=e^{-\frac{1}{2}\tau}H_{2}(y)\Big(a_{1},\ a_{2},\ a_{3},\ a_{4}\Big)^{T},
\end{align} for any sufficiently large $N,$ the function $v$ takes the form
\begin{align}
    v(y,\omega,\tau)=\sqrt{6+\sum_{n=0}^3\sum_{k=0}^{N}\sum_l \alpha_{n,k,l}(\tau) H_n(y)f_{k,l}(\omega)+\eta(y,\omega,\tau)}
\end{align} where, for any $k$ and $l$, for any $j\geq 1$, and for some constants $d_{2,k,l}$,
\begin{align}
   e^{\frac{1}{3}\tau} |\alpha_{0, k,l}(\tau)|+e^{\frac{5}{6}\tau}|\alpha_{1,k,l}(\tau)|+e^{\frac{4}{3}\tau}\Big|\alpha_{2,k,l}(\tau)-d_{2,k,l} e^{-\tau} \Big|+(1+\tau)^{-1}e^{\tau}|\alpha_{3,j,l}(\tau)|\lesssim 1,
\end{align}
and for some constant $d_3\in \mathbb{R}$,
\begin{align}
|\alpha_{3,0,1}(\tau)-d_3 e^{-\frac{1}{2}\tau}|\leq C_{\delta_0}e^{-(\frac{3}{2}-\delta_0)\tau},
\end{align}
the remainder $\eta$ satisfies the following estimates, 
\begin{align}
\sum_{i+|j|\leq 2}\Big\|1_{\leq 5 \tau^{\frac{1}{2}+\frac{1}{20}}} \partial_{y}^{i}\nabla_{E}^j\eta \Big\|_{\mathcal{G}}\leq  &  C_{\delta_0}e^{-(1-\delta_0)\tau},\label{eq:mGnormeta}\\
\sum_{i+|j|\leq 2}\Big|1_{\leq 5 \tau^{\frac{1}{2}+\frac{1}{20}}} \partial_{y}^{i}\nabla_{E}^j\eta \Big| \lesssim &  \tau^{-\frac{3}{10}}.
\end{align} Here $\delta_0$ is any positive constant, and $C_{\delta_0}$ is a constant depending on $\delta_0.$
\end{itemize}

\end{theorem}
Here $1_{\leq D}$, for any $D>0,$ is the Heaviside function taking value $1$ when $|y|\leq D$ and $0$ otherwise; $\nabla$ denotes covariant differentiation on $\mathbb{S}^3$, and $(E_1,E_2,E_3)$ is an orthonormal basis of $T_{\omega}\mathbb{S}^3.$ Recall the $\mathcal{G}-$norm defined in \eqref{eq:GInner}.

The proof will be presented in Section \ref{sec:MainTHM1}.

Next we continue to study the second possibility in Theorem \ref{THM:lIS1}.
As explained in Introduction, we will use induction, and need the condition that the blowup is Type I. Specifically let $A(p,t)$ be the second fundamental form of MCF at the point $p$ and time $t,$ we need the following condition, 
\begin{align}
\text{for some fixed constant}\ \beta>0,\  \sup_{t\leq T,\ p}\sqrt{T-t}|A(p,t)|\leq \beta.\label{eq:secFun}
\end{align} 
Now we state our second result.
\begin{theorem}\label{THM:sec}
We assume all the conditions in Theorem \ref{THM:lIS1}, and the condition \eqref{eq:secFun}.

In the first step of induction, we suppose that for some integer $m\geq 3$, there exists a vector-valued polynomial of the form, for some constants $a_{k,l}\in \mathbb{R},$
\begin{align}\label{eq:qmminus1}
Q_{m-1}(y,\tau)=\displaystyle\sum_{k=2}^{m-1}e^{-\frac{k-1}{2}\tau}H_{k}(y)\Big(a_{k,1},\ a_{k,2},\ a_{k,3},\ a_{k,4}\Big)^{T},
\end{align}
and a function $v$ s.t. the rescaled MCF is parametrized as $\Psi_{Q_{m-1},v}$, defined in \eqref{eq:parametrization},
in the region, 
\begin{align}\label{eq:defwm}
\Big\{y\ \Big|\ |y|\leq  8\tau^{\frac{1}{2}+\frac{1}{20}}\Big\}
\end{align} 
and $v$ takes the following form, for any sufficiently large $N\in \mathbb{N}$,
\begin{align}
v(y,\omega,\tau)=\sqrt{6+\sum_{n=0}^{m}\sum_{k=0}^{N}\sum_{l} \gamma_{n,k,l}(\tau)H_{n}(y)f_{k,l}(\omega)+\xi(y,\ \omega,\tau)},\label{eq:newdec}
\end{align} 
and for any $2\leq n\leq m-1$, for any $k$ and $l$, and for any $j\geq 1$, there exist constants $d_{n,k,l}$ s.t.
\begin{align}
\begin{split}\label{eq:n01Rapid}
    e^{\frac{1}{3}\tau}|\gamma_{0,k,l}(\tau)|+e^{\frac{5}{6}\tau}|\gamma_{1,k,l}(\tau)|+e^{(\frac{n}{2}+\frac{1}{3})\tau}\Big|\gamma_{n,k,l}(\tau)-& d_{n,k,l} e^{-\frac{n}{2}\tau}\Big|\\
    & + (1+\tau)^{-1}e^{\frac{m}{2}\tau}|\gamma_{m,j,l}(\tau)|\lesssim 1,
\end{split}
\end{align}
for some real constant $d_m$,
\begin{align}\label{eq:dm}
|\gamma_{m,0,1}(\tau)- d_m e^{-\frac{m-2}{2}\tau}|&\lesssim  e^{-\frac{m-1}{2}\tau},
\end{align}
the remainder $\xi$ satisfies the following estimates, for any $\delta_0>0$ there exists a $C_{\delta_0}>0$ such that
\begin{align}
\sum_{k+|l|\leq 2}\Big\|  1_{\leq 8 \tau^{\frac{1}{2}+\frac{1}{20}}}\partial_{y}^k \nabla_{E}^{l}\xi(\cdot,\tau)\Big\|_{\mathcal{G}}\leq & C_{\delta_0} e^{-\frac{m-1-\delta_0}{2}\tau},\label{eq:GLEtaM2}\\
\sum_{k+|l|\leq 2}\Big|  1_{\leq 8 \tau^{\frac{1}{2}+\frac{1}{20}}}\partial_{y}^k \nabla_{E}^{l}\xi(\cdot,\tau)\Big|_{\infty}\lesssim & \tau^{-\frac{3}{10}}.\label{eq:GLEtaM}
\end{align} 

If the results \eqref{eq:qmminus1}-\eqref{eq:GLEtaM} above hold, then we will prove the following results.

\begin{itemize}
\item[(A)]By the condition \eqref{eq:secFun}, $d_m=0$ if $m$ is odd; and $d_m\geq 0$ if $m$ is even.
\item[(B)] The results here hold for the following two cases: (1) the constant $d_m$ is positive and $m\geq 4$ is even; (2) $d_m=0$ and $m\geq 3$ (even or odd). 

The parametrization $\Psi_{Q_{m-1},v}$, with $Q_{m-1}$ defined in \eqref{eq:qmminus1}, works in a much larger region,
\begin{align}
|y|\leq   e^{\frac{m-2}{2m}\tau+8 \sqrt{\tau}}\label{eq:region}
\end{align}
and $v$ satisfies the estimates
\begin{align}\label{eq:pointwisedm}
\Big| v(y,\omega,\tau)-\sqrt{6+d_m e^{-\frac{m-2}{2}\tau} H_m(y)}\Big|+\sum_{j+|i|=1,2}\Big|v^{j-1}\partial_{y}^j\nabla_{E}^i v(y,\omega,\tau)\Big|\ll e^{-5\sqrt{\tau}}.
\end{align}
Here the constant $d_m$ in \eqref{eq:pointwisedm} is the same to that in \eqref{eq:dm}. 

\item[(C)]
If in \eqref{eq:dm} $d_m$ is positive and $m$ is even, then the following results hold for the MCF.

A fixed space-and-time neighborhood of singularity of MCF is mean convex and the singularity is isolated. Moreover there exists a constant $\epsilon>0$ s.t. in the space-and-time region,
\begin{align}\label{eq:regionMC}
    \epsilon\geq T-t\geq 0,\ \text{and}\ |z|\leq \epsilon,
\end{align}
the MCF takes the form
\begin{align}
    \left(
    \begin{array}{c}
         z\\
         \tilde{Q}_{m}(z,t) 
    \end{array}
    \right)+\left(
    \begin{array}{c}
         -\omega\cdot\partial_{z}\tilde{Q}_m(z,t)\\
         \omega
    \end{array}
    \right)u(x,\omega,t)
\end{align} where $\tilde{Q}_{m}(z,t)=\sqrt{T-t}Q_{m}(y,\tau)\in \mathbb{R}^4$ takes the form
\begin{align}
    \tilde{Q}_{m}(z,t)=\tilde{Q}_{m-1,1}(z)+(T-t)\tilde{Q}_{m-3,2}(z)+(T-t)^2\tilde{Q}_{m-5,3}(z)+\cdots
\end{align} and each of $\tilde{Q}_{k,l}\in \mathbb{R}^4$ is a vector-valued polynomial of degree $\leq k$, and $u$ satisfies the estimates:
\begin{align}
    u(z,\omega,t)>0 \ \text{when}\ (z,t)\not=(0,T)
\end{align}
and for any integer $M$, there exists some constant $\delta_{M}(\epsilon)$ satisfying $\lim_{\epsilon\rightarrow 0}\delta_{M}(\epsilon)=0,$ such that in the region \eqref{eq:regionMC},
\begin{align}
    \Big|u^{k-1}\partial_{z}^k\nabla_{E}^{l}u(z,\omega,t)\Big|\leq \delta_{M}(\epsilon),\ k+|l|=1,\cdots, M.\label{eq:MCSmDe}
\end{align}

\item[(D)] If $d_m=0,$ then we prepare for the next step of induction.

Specifically there exists a vector-valued polynomial, for some real constants $a_{k,l}$, 
\begin{align}
    Q_{m}(y,\tau)=\displaystyle\sum_{k=2}^{m}e^{-\frac{k-1}{2}\tau}H_{k}(y)\Big(a_{k,1},\ a_{k,2},\ a_{k,3},\ a_{k,4}\Big)^{T}
\end{align} a function $v>0$ such that the rescaled MCF can be parametrized as $\Psi_{Q_{m},v}$ in the region
\begin{align}
 \Big\{ y\ \Big|\   |y|\leq 8 \tau^{\frac{1}{2}+\frac{1}{20}} \Big\}.
\end{align} Here the constants $a_{k,l}$, $k\leq m-1$, are the same to those in \eqref{eq:qmminus1}.
For the function $v$, there exists an integer $N_{m+1}$ such that for any $N\geq N_{m+1}$, $v$ takes the form,
\begin{align}
v(y,\omega,\tau)=\sqrt{6+\sum_{n=0}^{m+1}\sum_{k=0}^{N} \sum_{l}\gamma_{n,k,l}(\tau) H_n(y) f_{k,l}(\omega)+\xi(y,\omega,\tau)},
\end{align} 
where, for any $2\leq n\leq m$, $k$ and $l$, and for any $j\geq 1$, there exist constants $d_{n,k,l}$ such that
\begin{align}
\begin{split}
    e^{\frac{1}{3}\tau}|\gamma_{0,k,l}(\tau)|+e^{\frac{5}{6}\tau}|\gamma_{1,k,l}(\tau)|+e^{(\frac{n}{2}+\frac{1}{3})\tau}\Big|\gamma_{n,k,l}(\tau)-& d_{n,k,l} e^{-\frac{n}{2}\tau}\Big|\\
    & + (1+\tau)^{-1}e^{\frac{m}{2}\tau}|\gamma_{m+1,j,l}(\tau)|\lesssim 1,
\end{split}
\end{align}
the focus is on $\gamma_{m+1,0,1}$, for some constant $d_{m+1}\in \mathbb{R}$,
\begin{align}
    \Big|\gamma_{m+1,0,1}(\tau)-d_{m+1}e^{-\frac{m-1}{2}\tau}\Big|\lesssim e^{-\frac{m}{2}\tau},
\end{align} 
and $\xi$ satisfies the estimates, for any $\delta_0>0$ there exists a $C_{\delta_0}>0$ such that
\begin{align}
    \sum_{k+|j|\leq 2}\Big\| 1_{\leq 8 \tau^{\frac{1}{2}+\frac{1}{20}}} \partial_{y}^{k}\nabla_{E}^{l}\xi(\cdot,\tau)\Big\|_{\mathcal{G}}\leq C_{\delta_0}e^{-\frac{m-\delta_0}{2}\tau},\\
    \sum_{k+|j|\leq 2}| 1_{\leq 8 \tau^{\frac{1}{2}+\frac{1}{20}}}\partial_{y}^{k}\nabla_{E}^{l}\xi(\cdot,\tau)|\leq  e^{-\sqrt{\tau}}.
\end{align} 
\end{itemize}
\end{theorem}
The four parts of the Theorem will be proved in Sections \ref{sec:step1Third}, \ref{sec:BTHMSec}, \ref{sec:ATHMsec} and \ref{sec:ProofTHMsecC} respectively. We will prove Parts B and C in detail, since technically these are the most involved parts.

If $d_m=0$ for all $m\geq 3,$ then we will formulate a conjecture, and it implies that a fixed neighborhood of the origin will blowup at time $T$. 

We start with retrieving some useful information from Part C of Theorem \ref{THM:sec}. If $d_2=d_4=\cdots=d_{m+1}=0$, then in the region
\begin{align}
  \Big\{\ y\ \Big|\  |y|\leq e^{-\frac{m-2}{2m}\tau}\Big\}
\end{align}
we parameterize the rescaled MCF as
\begin{align}
    \Psi_{Q_m, v_m}=
\left[
\begin{array}{ccc}
y\\
Q_m(y,\tau)
\end{array}
\right]+v_m(y,\omega,\tau)\left[
\begin{array}{ccc}
-\partial_{y}Q_m(y,\tau)\cdot\omega\\
\omega
\end{array}
\right],
\end{align} for some $Q_m$ of the form, for some constants $a_{n,l},$
\begin{align}\label{eq:conQm}
    Q_{m}(y,\tau)=e^{\frac{1}{2}\tau}\displaystyle\sum_{n=2}^{m}e^{-\frac{n}{2}\tau}H_{k}(y)\Big(a_{n,1},\ a_{n,2},\ a_{n,3},\ a_{n,4}\Big)^{T}
\end{align} and some function $v_m$ of the form 
\begin{align}\label{eq:conVm}
    v_{m}(y,\omega,\tau)=\sqrt{6+\sum_{n=0}^m\sum_{k=0}^{N_m}\sum_l\alpha_{n,k,l}(\tau)e^{-\frac{n}{2}\tau} H_n(y)f_{k,l}(\omega)+\eta_{m}(y,\omega,\tau)},
\end{align} where $N_m$ is a large integer, and the functions satisfy the following estimates, for some $d_{n,k,l}\in \mathbb{R}$,
$$\alpha_{n,k,l}(\tau) =d_{n,k,l} +\mathcal{O}(e^{-\frac{1}{3}\tau});$$
moreover 
\begin{align*}
|\alpha_{0,k,l}(\tau)|+|\alpha_{1,k,l}(\tau)|\lesssim e^{-\frac{1}{3}\tau},
\end{align*} and the remainder satisfies the estimate
\begin{align}
\|\eta_{m}(\cdot,\tau)\|_{\mathcal{G}}\lesssim e^{-\frac{m-1}{2}\tau}.
\end{align}

As discussed in the introduction, the important information is that $e^{-\frac{n}{2}\tau}H_n(y)$, in \eqref{eq:conQm} and \eqref{eq:conVm}, are uniformly small in a region $|y|\leq \epsilon_{m} e^{\frac{1}{2}\tau}$, provided that $\epsilon_m$ is small enough. For the corresponding MCF, this region corresponds to $|z|=e^{-\frac{1}{2}\tau} |y|\leq \epsilon_{m}$. Also we pointed out, around \eqref{eq:qmdnkl}, that $\sum_{n=2}^{m}\sum_{k,l}e^{-\frac{n}{2}\tau} d_{n,k,l} y^n f_{k,l}  $ is generated by $\partial_{y}Q_{m}.$

Now we are ready to state our conjecture,
\begin{conjecture}
If $d_m=0$ for all $m\geq 3$ then there exists a $\epsilon_{\infty}>0$ such that when $|y|\leq \epsilon_{\infty}e^{\frac{\tau}{2}}$, the limits $\displaystyle\lim_{m\rightarrow \infty}e^{-\frac{1}{2}\tau}Q_m$ and $\displaystyle\lim_{m\rightarrow \infty}v_{m}$ exists and the limits are smooth.
\end{conjecture}

If the conjecture is true, then the rescaled MCF take the form
$
    \Psi_{Q_{\infty},v_{\infty}},
$ and
the corresponding MCF is $\Phi_{\Pi_{\infty},u_{\infty}}$ defined as
\begin{align}
    \Phi_{\Pi_{\infty}, u_{\infty}}:=\sqrt{T-t} \Psi_{Q_{\infty},v_{\infty}}.
\end{align} It will blowup in the set $|z|\leq \epsilon_{\infty}$ at the time $T$.

%%%%%%%%%%%%%%%%%%%%%%%%%%%%%%%%%%%%%%%%%%%%%%%%%%%%%%

\section{Proof of Theorem \ref{THM:lIS1}}\label{sec:MainTHM1}
We start with deriving some estimates for $v$ from \cite{ColdingMiniUniqueness}, in Lemma \ref{LM:ColdMini} below.
To measure the size of the controlled neighborhood of $y=0$, we define a function $ R_1:\ \mathbb{R}^{+}\rightarrow \mathbb{R}^{+}$ as
\begin{align}
 R_{1}(\tau):=\sqrt{\frac{28}{5}}\sqrt{\ln\tau}.\label{eq:defR1T}
\end{align}

The result is derived from \cite{ColdingMiniUniqueness}. The derivation is the same to
that in our previous paper \cite{GZ2017}.
\begin{lemma}\label{LM:ColdMini}
There exists a constant $\tau_0$ such that if $\tau\geq \tau_0$ and $|y|\leq R_1(\tau)$, $v$ satisfies the estimates
\begin{align}
\Big|v-\sqrt{2}\Big|,\ \Big|\partial_y v\Big|,\ \Big|\nabla_{E} v\Big|\leq \tau^{-\frac{18}{25}} e^{\frac{1}{8}|y|^2},\label{eq:cm1}
\end{align} and there exists a constant $C$ such that if $k\in (\mathbb{Z}^{+})^3$ and $l\in \mathbb{Z}^{+} $ satisfy $2\leq |k|+l\leq 10,$ then
\begin{align}
|\partial_{y}^{l}\nabla_{E}^{k}v|\leq C,\label{eq:cm2}
\end{align}
And there exists a constant $\beta>0$ such that
\begin{align}
\Big\|\Big(v(\cdot,\tau)-\sqrt{6}\Big)1_{|y|\leq R_1}\Big\|_{\infty}+ \sum_{|k|+l\leq 4}\Big\|1_{|y|\leq R_1}\nabla_{E}^{k}\partial_{y}^{l}v(\cdot,\tau)\Big\|_{\infty}\leq \tau^{-\beta}.\label{eq:IniWeighted}
\end{align}

\end{lemma}
Here $1_{|y|\leq R_1}$ is the standard Heaviside function taking value $1$ when $|y|\leq R_1(\tau),$ and $0$ otherwise.

Now we expand the considered region and find a finer description, identically to what we did in our previous papers \cite{GZ2017,GZ2018}. The result is following:
%%%%%%%%%%%%%%%%%%%%%%%%%%%%%%%%%%%%%

\begin{proposition}\label{prop:mostgeneric}

Suppose the conditions before \eqref{eq:origPara10} hold. We parametrize the rescaled MCF as $\Psi_{0,v}$, which is defined in \eqref{eq:norPara}, for some function $v>0$ in the region 
\begin{align}
\Big\{y\ \Big|\ |y| \leq 10 \tau^{\frac{1}{2}+\frac{1}{20}}\Big\}\label{eq:IniRegion}
\end{align} 
$v$ is of the form
\begin{align}
v(y,\omega,\tau)=\sqrt{6+\sum_{n=0}^2\alpha_n(\tau) H_n(y)+\sum_{k=0,1}\sum_{l=1}^{4}\alpha_{k,l}(\tau)H_{k}(y)\omega_l +\eta(y,\ \omega,\tau)}.
\end{align} 
Then one and only one of the following two possibilities must hold.
\begin{itemize}
\item[(1)] For the first possibility, the following estimates hold,
\begin{align}
          \begin{split}
    \sum_{n=0,1}|\alpha_{n}(\tau)|+|\alpha_2(\tau)-3\tau^{-1}|&+\sum_{k=0,1}\sum_{l=1}^{4}|\alpha_{n,l}(\tau)|\\
    &+\sum_{j+|i|\leq 2}\Big\| 1_{\leq 10 \tau^{\frac{1}{2}+\frac{1}{20}}}\partial_{y}^{j}\nabla_{E}^{i}\eta(\cdot,\tau) \Big\|_{\mathcal{G}}\lesssim \tau^{-2},
           \end{split}
\end{align} 
and
\begin{align}
        \sum_{k+|l|\leq 2} | \partial_{y}^k\nabla_{E}^{l}
         \eta(\cdot,\tau) |\lesssim& \tau^{-\frac{3}{10}},\label{eq:pointWiseR}
\end{align}
Here the singularity is isolated and a small neighborhood of singularity of MCF is mean convex, and the estimates in \eqref{eq:epsiMT} and \eqref{eq:deriUMT} hold.

\item[(2)] For the second possibility the function $\alpha_2$ decays more rapidly,
\begin{align}\label{eq:secab}
   \begin{split}
    \sum_{n=0}^2|\alpha_{n}(\tau)|+\sum_{k=0,1}\sum_{l=1}^{4}|\alpha_{k,l}(\tau)|+\sum_{j+|i|\leq 2}\Big\| 1_{\leq 10 \tau^{\frac{1}{2}+\frac{1}{20}}}\partial_{y}^{j}\nabla_{E}^{i}\eta(\cdot,\tau) \Big\|_{\mathcal{G}}\lesssim \tau^{-2},
    \end{split}
\end{align}
and the estimates in \eqref{eq:pointWiseR} still hold.
\end{itemize}
\end{proposition}
Here we choose to skip the proof of this proposition by the following reasons:
\begin{itemize}
    \item the proof will be very similar to those in \cite{GZ2017, GZ2018};
    \item more importantly, all the needed techniques will be used in the present paper, for example, in Section \ref{sec:BTHMSec} below, where we will prove better estimates in a much larger region.
\end{itemize}

Next we prove the second part of Theorem \ref{THM:lIS1}. 
We will prove the desired results in two steps. 
In the first step we prepare for the normal form transformation by studying the direction $H_2\omega_l$, or the functions $\alpha_{2,1,l}$, $l=1,2,3,4,$ in Proposition \ref{prop:step2} below. In the second step, which is Proposition \ref{prop:afterNorForm}, we preform a normal form transformation to remove the main part of $\alpha_{2,1,l}H_2\omega_l.$

In the next result we are interested in the region 
\begin{align}
\Big\{y\ \Big| \ |y|\leq (1+\epsilon)R(\tau) \Big\}
\end{align} where $R$ is a function defined as
\begin{align}
R(\tau):= 8\tau^{\frac{1}{2}+\frac{1}{20}},\label{def:R}
\end{align} and $\epsilon>0$ is a fixed small constant to appear in \eqref{eq:defChi3} below. Here the considered region is smaller than 
that in \eqref{eq:IniRegion}, the reason is that, we need to change the coordinate in the proof, thus have to choose a smaller region to accommodate this.

To limit our consideration to the desired set, we need a cutoff function $\chi_{R}$, defined as
\begin{align}
    \chi_{R}(y)=\chi(\frac{y}{R}).\label{eq:chiRy}
\end{align}
Here $\chi$ is an even cutoff function in $C^{29,1}$ such that for some small $\epsilon>0,$
\begin{align}\label{eq:defChi3}
\chi(z)=\chi(|z|)=\ \Big[
\begin{array}{lll}
1,\ \text{if}\ |z|\leq 1,\\
0,\ \text{if}\ |z|\geq 1+\epsilon.
\end{array}
\end{align}

%%%%%%%%%%%%%%%%%%%%%%%%%%%%%%%%%%%%%%%%%%%%

The main result is following:
\begin{proposition}\label{prop:step2}
For the second case of Proposition \ref{prop:exponential}, there exist functions $\alpha_{n,k,l}$ such that
\begin{align}
v(y,\omega,\tau)=\sqrt{6+\sum_{n=0}^3 \sum_{k=0}^{N} \sum_{l} \alpha_{n,k,l}(\tau)H_n(y)f_{k,l}(\omega)+\eta(y,\omega,\tau)}
\end{align} where $N$ is a large integer, and $\chi_{R}\eta$ satisfies the orthogonality conditions
\begin{align}
\chi_{R}\eta\perp_{\mathcal{G}} H_{n} f_{k,l},\ n=0,1,2,3; \ k=0,\cdots, N,
\end{align}the focus is on $\alpha_{2,1,l}:$ for some real constants $d_{2,l}$,
\begin{align}
\alpha_{2,1,l}(\tau)=&d_{2,l}e^{-\frac{1}{2}\tau}+\mathcal{O}(e^{-\tau}),\ l=1,2,3,4,\label{eq:d21Order}
\end{align}
and for any $k$ and $l$, and for any $n\not=1$ and $j\not=0$, 
\begin{align}
\begin{split}\label{eq:d3d}
e^{\frac{1}{3}\tau}|\alpha_{0,k,l}(\tau)|+e^{\frac{5}{6}\tau}|\alpha_{1,k,l}(\tau)|+e^{\tau}|\alpha_{2,n,l}(\tau)|+e^{\frac{1}{2}\tau}|\alpha_{3,0,1}(\tau)|+
    (1+\tau)^{-1} e^{\tau}|\alpha_{3,j,l}| \lesssim & 1,
\end{split}
\end{align}and for any $\delta_0>0$ there exists some constant $C_{\delta_0}$such that
\begin{align}
 \sum_{k+|l|\leq 2}\Big\|\partial_{y}^k\nabla_{E}^l\chi_{R}\eta(\cdot,\tau) \Big\|_{\mathcal{G}}\leq & C_{\delta_0} e^{-(1-\delta_0)\tau},\label{eq:RetaSharp}\\
\Big|\partial_{y}^k\nabla_{E}^l\chi_{R}\eta(\cdot,\tau)\Big|\lesssim \tau^{-\frac{3}{10}} .\label{eq:etaPoin34}
\end{align}

\end{proposition}
The proposition will be proved in Section \ref{sec:PropStep2}.

Next we make a normal form transformation to remove the part $\sum_{l}d_{2,l} e^{-\frac{\tau}{2}} H_2 \omega_l$ and then study the new rescaled MCF.
The result is following: recall the definition of $\Psi_{Q_{N},v}$ defined from \eqref{eq:norPara},
\begin{proposition}\label{prop:afterNorForm}
There exist real constants $a_{k}$, $k=1,2,3,4,$ such that if we parametrize the rescaled MCF as $\Psi_{Q_2,v}$,
with $Q_2$ defined as
\begin{align}
Q_2(y,\tau)=e^{-\frac{1}{2}\tau} H_2(y) \Big(a_1,a_2,a_3,a_4\Big), \label{eq:Q2normal}
\end{align} 
then $v$ can be decomposed into the form
\begin{align}
v(y,\omega,\tau)=\sqrt{6+\sum_{n=0}^{3}\sum_{k=0}^{N} \sum_{l}\gamma_{n,k,l}(\tau)H_{n}(y)f_{k,l}(\omega)  +\xi(y,\ \omega,\tau)},
\end{align} where $N$ is a large integer, and $\chi_{R}\xi$ satisfies the following orthogonality conditions 
\begin{align}
\chi_{R}\xi\perp_{\mathcal{G}} H_n f_{k,l},\ n=0,1,2,3; \ k=0,\cdots,N,\label{eq:xiortho}
\end{align} the focus is on $\alpha_{3,0,1}$, for some constant $d_3\in \mathbb{R}$,
\begin{align}
     \Big|\gamma_{3,0,1}(\tau)- d_3 e^{-\frac{1}{2}\tau} \Big| \lesssim e^{-\tau}
\end{align}
and for any $k$ and $l$, for any $j\geq 1$, there exist constants $d_{2,k,l}$ such that
\begin{align}
    e^{\frac{1}{3}\tau}|\gamma_{0,k,l}(\tau)|+ e^{\frac{5}{6}\tau}|\gamma_{1,k,l}(\tau)|+e^{\frac{4}{3}\tau}\Big|\gamma_{2,k,l}(\tau)-d_{2,k,l}(\tau)e^{-\tau}\Big|   +(1+\tau)^{-1} e^{\tau}\Big|\gamma_{3,j,l}(\tau)\Big| \lesssim 1, \label{eq:tiAlpha012}
    \end{align}
and the remainder $\xi$ enjoys the same estimates to $\eta$ in \eqref{eq:RetaSharp} and \eqref{eq:etaPoin34}: for any $\delta_0>0,$
\begin{align}
 \sum_{k+|l|\leq 2}\Big\|\partial_{y}^k\nabla_{E}^l\chi_{R}\xi(\cdot,\tau) \Big\|_{\mathcal{G}}\leq & C_{\delta_0} e^{-(1-\delta_0)\tau},\label{eq:TetaSharp}\\
\Big|\partial_{y}^k\nabla_{E}^l\chi_{R}\xi(\cdot,\tau)\Big|\lesssim \tau^{-\frac{3}{10}} .\label{eq:TetaPoin34}
\end{align}
\end{proposition}
In what follows we discuss the proof.

To remove the part $\sum_{l=1}^4 d_{2,l} e^{-\frac{1}{2}\tau} H_2 \omega_l$ we parametrize the rescaled MCF as 
$$\Psi_{Q_2, v}=
\left[
\begin{array}{ccc}
y\\
Q_2(y,\tau)
\end{array}
\right]+v(y,\omega,\tau)\left[
\begin{array}{ccc}
-\partial_{y}Q_2(y,\tau)\cdot\omega\\
\omega
\end{array}
\right]$$ for some vector-valued function $Q_{2}$ of the form, for some real constants $a_k,\ k=1,2,3,4,$
\begin{align}
Q_2(y,\tau)=e^{-\frac{1}{2}\tau}H_{2}(y)\Big(a_{1},\ a_{2},\ a_{3},\ a_{4}\Big)^{T}.
\end{align}
The existence of such a vector-valued function $Q_2$ is standard, for example, by a graphic illustration. This idea was used in \cite{GaKn2014,ColdingMiniUniqueness} to find better coordinates.

To prove the desired estimates we need to derive a governing equation for $v$. Since in the rest of the paper we often need a governing equation for $v$ with parametrization $\Psi_{Q_N,v},$ here we consider the problem in a slightly more general setting. 

Suppose that the parametrization for the rescaled MCF is
\begin{align}\label{eq:psiqnv}
    \Psi_{Q_N, v}=
\left[
\begin{array}{ccc}
y\\
Q_N(y,\tau)
\end{array}
\right]+v(y,\omega,\tau)\left[
\begin{array}{ccc}
-\partial_{y}Q_N(y,\tau)\cdot\omega\\
\omega
\end{array}
\right],
\end{align} where $Q_{N}\in \mathbb{R}^4$ is a vector-valued function of the form, for some real constants $a_{n,k},$
\begin{align}
Q_N(y,\tau)=\sum_{n=2}^{N} e^{-\frac{n-1}{2}\tau}H_{n}(y)\Big(a_{n,1},\ a_{n,2},\ a_{n,3},\ a_{n,4}\Big)^{T}.
\end{align} 
The corresponding MCF, $\Phi_{\Pi_{N}}(u),$ takes the form
\begin{align}\label{eq:qnrescaled}
\sqrt{T-t}\Psi_{Q_{N},v}=\Phi_{\Pi_{N}, u}=\left[
\begin{array}{ccc}
z\\
\Pi_N(z,t)
\end{array}
\right]+u(z,\omega,t)\left[
\begin{array}{ccc}
-\partial_{z}\Pi_N(z,t)\cdot\omega\\
\omega
\end{array}
\right],
\end{align} $u,$ $z$, $\tau$ and $\Pi_{N}$ are defined in terms of $v,$ $y$, $\tau$ and $Q_N$ by the identities in \eqref{eq:ytau} and \eqref{eq:QPIdentity}.

We will derive in Appendix \ref{sec:derivation} below that the function $u$ satisfies the equation
\begin{align}
\partial_{t}u=\partial_{z}^2 u +u^{-2}\Delta_{\mathbb{S}^3}u+N(u)+V_{\Pi_N}(u)
\end{align}where $N(u)$ and $V_{\Pi_{N}}(u)$ are defined in (\ref{eq:effectivUeqn}).

From this we derive a governing equation for $v$ as
\begin{align}
\partial_{\tau}v=\partial_{y}^2v -\frac{1}{2}y\partial_{y}v+v^{-2}\Delta_{\mathbb{S}^3}v+\frac{1}{2}v+N(v)+W_{Q_N}(v),
\end{align} where $N(v)$ is defined in \eqref{eq:effeV}, and $W_{Q_N}(v)$ is defined as
\begin{align}
W_{Q_N}(v):=\sqrt{T-t} V_{\Pi_{N}}(u).
\end{align}
In the present paper what is more useful is a governing equation for $\tilde{v}:=v^2,$
\begin{align}
\partial_{\tau}\tilde{v}=v^{-2}\Delta_{\mathbb{S}^2}\tilde{v}+\partial_{y}^2 \tilde{v}-\frac{1}{2}y\partial_{y}\tilde{v} +\tilde{v}-6-\frac{1}{2}v^{-4} |\nabla_{\omega}^{\perp}v^2 |^2-\frac{1}{2}v^{-2}|\partial_{y}v^2 |^2+2v N(v)
+2vW_{Q_N}(v).\label{eq:tilv}
\end{align} Here $\nabla_{\omega}^{\perp}$ stands for $P_{\perp \omega}\nabla_{\omega}$, and the operator $P_{\perp\omega}:\mathbb{R}^4\rightarrow \mathbb{R}^4$ is the orthogonal projection defined as, for any vector $A\in \mathbb{R}^4$,
\begin{align}
    P_{\perp\omega}A=A-(\omega\cdot A)\omega.
\end{align}

Even though the nonlinearities $v^{-4} |\nabla_{\omega}^{\perp}v^2|^2+v^{-2}|\partial_{y}v^2|^2$, $vN(v)$ and $vW_{Q_N}(v)$ contain many terms, only a few will play important roles and will be the focus of our treatment. 
Up to a constant factor, the nonlinearities contain the following terms:
\begin{itemize}
\item from $vN(v)$, 
\begin{align}\label{def:Kterm}
\begin{split}
K_1(v):= v(\partial_{y}v)^2 \partial_{y}^2 v;\\
K_2(v):=v^{-2} |\nabla_{\omega}^{\perp}v|^2;
\end{split}
\end{align}
\item from $v^{-4} |\nabla_{\omega}^{\perp}v^2|^2+v^{-2}|\partial_{y}v^2|^2$, which also contains $K_2(v)$,
\begin{align}\label{def:Iterm}
I_1(v):=&v^{-1} |\partial_{y} v^2|^2,
\end{align}
\item from $vW_{Q_{N}}(v)$, see Remark \ref{rem:source} below,
\begin{align}
\begin{split}\label{def:J}
J_1(Q_N, v):=&(\partial_{y}v) (\partial_{y}Q_{N}\cdot \omega),\\
J_2(Q_{N}, v):=& v^{-1} (\partial_{y}Q_{N}\cdot \omega) (\nabla_{\omega}^{\perp} v\cdot \partial_{y}Q_{N}) ,\\
J_3(Q_{N},v):=&v^{-1} (\partial_{y}Q_{N})\cdot 
\nabla_{\omega}^{\perp}\Big((\partial_{y}Q_{N})\cdot \nabla_{\omega}^{\perp} v\Big),\\
J_4(Q_{N},v):=& v(\partial_{y}^2 v) |\partial_{y}Q_{N}\cdot \omega|^2,\\
J_5(Q_{N},v):=& v^2 \partial_{y}^2 v (\partial_{y}^2Q_{N}\cdot \omega), \\
J_6(Q_{N}, v):=&\partial_{y}Q_{N}\cdot \partial_{y}\nabla_{\omega}^{\perp}v, 
\end{split}
\end{align} the terms $J_7$ and $J_8$ are independent of $v$ and defined as
\begin{align}
\begin{split}\label{def:J7J8}
    J_7(Q_{N}):=&|P_{\perp\omega}\partial_{y} Q_{N}|^2,\\
J_8(Q_{N}):=&(\omega\cdot \partial_{y}Q_N)^2,
\end{split}
\end{align}
$J_9$ and $J_{10}$ are different since they are not nonlinear in terms of the derivatives of $v$ and $\partial_{y}Q_N.$
\begin{align}
          \begin{split}\label{def:J9J10}
    J_9(Q_{N},v):=&v (\omega\cdot \partial_{y}^2Q_{N});\\
    J_{10}(Q_{N},v):=& v \sum_{n=2}^N e^{-\frac{n-1}{2}\tau}\sum_{k=1}^{\floor*{\frac{n}{2}}}k h_{n-2k,n} y^{n-2k-2} \Big(\omega\cdot (a_{n,1},a_{n,2},a_{n,3},a_{n,4})\Big).
           \end{split}
\end{align}$J_{10}$ is from the term $\omega\cdot \partial_{t}\Pi_{N}$ in Remark \ref{rem:source} below, the constants $h_{n-2k,n}$ are the coefficients of 
the Hermite polynomial $H_{n}$, see (\ref{eq:hermit}).
\end{itemize}

These chosen terms are ``simple" in the sense that their number of factors is less than some other terms. But this means that it is harder to control them, since each factor is small, the less number of factors implies less room to maneuver. For example, by the techniques we will develop, it is easier to control $\tilde{K}_j-K_j$ and $\tilde{J}_k-J_k$ even though they have more factors, since $D$ and $(\partial_{y}v)^2+v^{-2}|\nabla_{\omega}^{\perp}v|^2$ are small.
Here
$\tilde{K}_j(v)$, $\tilde{J}_k(v)$  and $D$ are defined as:
\begin{align}
\begin{split}\label{def:TilKJ}
    \tilde{K}_{j}(v):=&\frac{K_j(v)}{1+(\partial_{y}v)^2+v^{-2}|\nabla_{\omega}^{\perp}v|^2},\ j=1,2;\\
\tilde{J}_{k}(v):=&\frac{J_k(v)}{1+D},\ k=1,2,\cdots,8,\\
    D:=&|P_{\perp \omega}\partial_{y}Q_{N}|^2-(\partial_{y}Q_{N}\cdot\omega) \partial_{y}v-(\partial_{y}^2 Q_{N}\cdot \omega)v+v^{-1} (\partial_{y}Q_{N}\cdot \omega) (\nabla_{\omega}^{\perp} v\cdot \partial_{y}Q_{N}).
    \end{split}
\end{align}

%%%%%%%%%%%%%%%%%%%%%%%%%%%%%%%%%%%%%%%%%%%%%%%%%%%%%%%%%%%%%%%%%%%%%%%%%%%%%%%%%%%%%%%%%%%%%%%%%%%%%%%%%%%%%%%%%%%%

Now we continue to prove Proposition \ref{prop:afterNorForm}. 
\subsection{Proof of Proposition \ref{prop:afterNorForm}}\label{subsec:afterNorm3}
We start with some preliminary estimates: recall the estimates for $\alpha_{n,k,l}$ in Proposition \ref{prop:step2},
\begin{lemma}
\begin{align}
\Big|\alpha_{n,k,l}(\tau)-\gamma_{n,k,l}(\tau)\Big|+\sum_{k+|l|\leq 2}\Big\|\partial_{y}^k\nabla_{E}^{l}\chi_{R}\xi(\cdot,\tau)\Big\|_{\mathcal{G}} \lesssim e^{-\frac{1}{2}\tau}\label{eq:3prelim}
\end{align}
\end{lemma}
The proof is easy
since we consider the reparametrization as a perturbation of order $e^{-\frac{1}{2}\tau}$.

These estimates and \eqref{eq:d21Order}-\eqref{eq:etaPoin34} imply the desired pointwise estimates for $\chi_{R}\xi$ in \eqref{eq:TetaPoin34}.

What is left is to improve the decay estimates for $\gamma_{n,k,l}$ and the $\mathcal{G}-$norm of $\chi_{R}\xi$ by studying their governing equations.

To simplify notations we define a function $\gamma$ as
\begin{align*}
\gamma(y,\omega,\tau):=&6+\sum_{n=0}^{3}\sum_{k=0}^N\sum_l \gamma_{n,k,l}(\tau) H_{k}(y)f_{k,l}(\omega).
\end{align*} 
From (\ref{eq:tilv}) we derive
\begin{align}\label{eq:EffTilXi}
\partial_{\tau}\chi_{R}\xi=-L\chi_{R}\xi+\chi_{R}\Big(F(Q_2,\gamma)+SN(Q_2,\gamma,\xi)\Big)+\mu_{R}(\xi),
\end{align}
where the linear operator $L$ is defined in \eqref{def:L}, and $\mu_{R}(\xi)$ is defined in the same way as \eqref{eq:effectChiREta},
the function $F$ is independent of $\xi$ and is defined as 
\begin{align*}
F(Q_2,\gamma):=&F_1+F_2,
\end{align*} and the functions $F_1$ and $F_2$ are defined as, 
\begin{align*}
F_1:=&-\sum_{n=0}^3\sum_{k=0}^{2}\sum_{l=1}^4 \Big(\frac{d}{d\tau}+\frac{n-2}{2}+\frac{k(k+2)}{6}\Big)\gamma_{n,k,l} H_{n}f_{k,l}(\omega),\\
F_2:=&\frac{6-\gamma}{6\gamma }\Delta_{\mathbb{S}^3}\gamma-\frac{1}{2\gamma^2} |\nabla_{E}\gamma|^2-\frac{1}{2\gamma}|\partial_{y}\gamma|^2+2\sqrt{\gamma} N(\sqrt{\gamma})+2\sqrt{\gamma}W_{Q_2}(\sqrt{\gamma}),
\end{align*}
$SN(Q_2,\gamma,\xi)$ contains terms nonlinear in terms of $\eta$, or linear in terms of $\eta$ but is ``small",
\begin{align*}
SN:=&\frac{1}{2\gamma^2}|\nabla_{E}\gamma|^2-\frac{1}{2v^4} |\nabla_{E}v^2|^2+\frac{1}{2q}|\partial_{y}q|^2-\frac{1}{2}v^{-2}|\partial_{y}v^2 |^2+2v N(v)-2\sqrt{\gamma} N(\sqrt{\gamma})\\
&+\frac{6-\gamma}{6\gamma}\Delta_{\mathbb{S}^3}\xi-\frac{\xi}{v^2 \gamma }\Delta_{\mathbb{S}^2}v^2+2v W_{Q_2}(v)-2\sqrt{\gamma}W_{Q_2}(\sqrt{\gamma}).
\end{align*}

Now we estimate various functions. The proof is similar to that of Proposition \ref{prop:step2} in subsection \ref{sec:PropStep2}. Thus we only sketch the proof.

To estimate $\chi_{R}\xi$ and its derivatives in the $\mathcal{G}-$norm, we define
\begin{align*}
    \phi_{j,m}(\tau):=\Big\langle (-\Delta_{\mathbb{S}^3}+1)^{m}\partial_{y}^j\chi_{R}\xi(\cdot,\tau),\partial_{y}^j\chi_{R}\xi(\cdot,\tau)\Big\rangle_{\mathcal{G}}
\end{align*} 
and then apply the same methods used in \eqref{eq:geqn} and \eqref{eq:gnm0}
to derive
\begin{align}
   \Big( \frac{1}{2}\frac{d}{d\tau}+1+\mathcal{O}(\tau^{-\frac{1}{4}})\Big)\sum_{j+\frac{1}{2}m\leq 2}\phi_{n,m}
   &\leq \sum_{j+\frac{1}{2}m\leq 2}\Big\langle (-\Delta_{\mathbb{S}^3}+1)^{m}\partial_{y}^j \chi_{R}\xi\  \partial_{y}^j \chi_{R}F_2\Big\rangle_{\mathcal{G}}+o(e^{-\frac{1}{5}R^2}).
\end{align}
The orthogonality condition enjoyed by $\chi_{R}\xi$ in (\ref{eq:xiortho}) cancels contribution from the slowly decaying terms of $F_2$. This and \eqref{eq:3prelim} imply that, for any constant $\delta_0>0$ there exists a $C_{\delta_0}$ s.t.
\begin{align}
    \sum_{j+\frac{1}{2}m\leq 2}\phi_{j,m}\leq C_{\delta_0} e^{-(1-\delta_0)\tau}.
\end{align}

To prepare for estimating $\gamma_{n,k,l}$, from the orthogonality conditions (\ref{eq:xiortho}) we derive
\begin{align}
\Big(\frac{d}{d\tau}+\frac{n-2}{2}+\frac{k(k+2)}{6}\Big)\gamma_{n,k,l}=\frac{1}{\|H_nf_{k,l}\|_{\mathcal{G}}^2}\Big\langle  F_2+SN+\mu_{R}(\xi),\ H_{n}f_{k,l}\Big\rangle_{\mathcal{G}}=:NL_{n,k,l}.
\end{align}

%%%%%%%%%%%%%%%%%%%%%%%%%%%%%%%%%%%%%%%%%%%%%%%%%%%%%%%%%%%%%%%%%%%%%%%%%%%%%%%%%%%%%%%%%%%%

%%%%%%%%%%%%%%%%%%%%%%%%%%%%%%%%%%%%%%%%%%%%%%%%%%%%%%%%%%%%%%%%%%%%%

\eqref{eq:3prelim} implies that, since all the terms in $F_2+SN$ are nonlinear in terms of $\gamma_{n,k,l}$, $\xi$ and $\partial_{y}Q_2$ except $J_9$ and $J_{10}$ defined in (\ref{def:J9J10}), and since $J_9$ and $J_{10}$ take explicit forms, 
\begin{align}\label{eq:snl3}
\begin{split}
    |NL_{n,k,l}(\tau)|\lesssim & e^{-\tau}, \ n=1,2,3,\\
    |NL_{0,k,l}(\tau)|\lesssim & e^{-\frac{1}{2}\tau}.
\end{split}
\end{align} The terms $J_9$ and $J_{10}$ make the second estimate sharp.

Similar to proving \eqref{eq:02Sharp} and \eqref{eq:matFin}, since the linear decay rates of $\gamma_{0,2,l}$ and $\gamma_{1,2,l}$, which are $e^{-\frac{1}{3}\tau}$ and $e^{-\frac{5}{6}\tau}$, are the lowest among $\gamma_{0,k,l}$ and $\gamma_{1,k,l}$, $k\geq 2,$ it is easy to obtain sharp estimates,
\begin{align}\label{eq:slowest}
    |\gamma_{0,k,l}(\tau)|e^{\frac{\tau}{3}}+|\gamma_{1,k,l}(\tau)|e^{\frac{5}{6}\tau}\lesssim 1,\ \text{for any}\ k\geq 2.
\end{align} For the pairs $(n,k)=(0,k),\ (1,k)$ with $k=0,1,$ since $\frac{n-2}{2}+\frac{k(k+2)}{6}\leq 0$ and $\lim_{\tau\rightarrow \infty}\gamma_{n,k,l}(\tau)=0$,
$$\gamma_{n,k,l}(\tau)=-\int_{\tau}^{\infty} e^{-(\frac{n-2}{2}+\frac{k(k+2)}{6})(\tau-\sigma) } NL_{n,k,l}(\sigma)\ d\sigma.$$ From here we obtain the desired estimates.

What is left is to consider $\gamma_{n,k,l}$ with $n\geq 2.$

Compared to the estimates for $\alpha_{2,1,l}$ in \eqref{eq:d21Order}, those for $\gamma_{2,1,l},$ $l=1,2,3,4,$ are improved,
\begin{align}
\gamma_{2,1,l}(\tau)=& e^{-\frac{1}{2}\tau} \gamma_{2,1,l}(0)+\int_{0}^{\tau} e^{-\frac{1}{2}(\tau-\sigma)} NL_{2,l,l}(\sigma)\ d\sigma\nonumber\\
=&-e^{-\frac{1}{2}\tau}\int_{\tau}^{\infty} e^{\frac{1}{2}\sigma} NL_{2,l,1}(\sigma)\ d\sigma
=\mathcal{O}(e^{-\tau})\label{eq:21l}
\end{align}
where we use that  $
    \tilde{d}_{2,l}=\gamma_{2,1,l}(0)+\int_{0}^{\tau} e^{\frac{1}{2}\sigma} NL_{2,l,1}(\sigma)\ d\sigma=0,
$ since we made normal form transformation to make them decay faster than $e^{-\frac{1}{2}\tau}$.

\eqref{eq:snl3} implies the following sharp estimates: when $(n,k)=(2,0)$,
\begin{align}\label{eq:2021}
    \gamma_{2,0,1}(\tau)=&-\int_{\tau}^{\infty} NL_{2,0,1}(\sigma)\ d\sigma=\mathcal{O}(e^{-\tau});
\end{align} for the pairs $(n,k)=(2,k)$ with $k\geq 2$, we have $\frac{n-2}{2}+\frac{k(k+2)}{6}=\frac{k(k+2)}{6}\geq \frac{4}{3}$, and thus
\begin{align}\label{eq:2kg2}
    \gamma_{2,k,l}(\tau)=&e^{-\frac{k(k+2)}{6}\tau} \gamma_{2,k,l}(0)+\int_{0}^{\tau}e^{-\frac{k(k+2)}{6}(\tau-\sigma)} NL_{2,k,l}(\sigma)\ d\sigma=\mathcal{O}(e^{-\tau}).
\end{align}

Next we consider $(n,k)$ with $n=3.$

When $(n,k)=(3,0)$, 
\begin{align}
  \gamma_{3,0,1}(\tau)= & e^{-\frac{1}{2}\tau} \gamma_{3,0,1}(0)+\int_{0}^{\tau} e^{-\frac{1}{2}(\tau-\sigma)} NL_{3,0,1}(\sigma)\ d\sigma\nonumber\\
=&d_{3} e^{-\frac{1}{2}\tau} -e^{-\frac{1}{2}\tau}\int_{\tau}^{\infty} e^{\frac{1}{2}\sigma} NL_{3,0,1}(\sigma)\ d\sigma
=d_{3} e^{-\frac{1}{2}\tau} +\mathcal{O}(e^{-\tau}),
\end{align} where $d_3$ a constant defined as
\begin{align*}
    d_3:=\gamma_{3,0,1}(0)+\int_{0}^{\infty} e^{\frac{1}{2}\sigma} NL_{3,0,1}(\sigma)\ d\sigma.
\end{align*}

When $n=3$ and $k\geq 1$, we have $\frac{k(k+2)}{6}+\frac{1}{2}\geq 1.$
This, together with \eqref{eq:snl3}, implies that
\begin{align}
    \gamma_{3,k,l}(\tau)=&e^{-(\frac{k(k+2)}{6}+\frac{1}{2})\tau}\gamma_{3,k,l}(0)+\int_{0}^{\tau} e^{-(\frac{k(k+2)}{6}+\frac{1}{2})(\tau-\sigma)} NL_{3,k,l}(\sigma)\ d\sigma
    =\mathcal{O}\Big(e^{-\tau}(1+\tau)\Big).\label{eq:3kl}
\end{align} 
Here the estimates for $\gamma_{3,1,l}$ in \eqref{eq:3kl} are sharp since $F_2$ contains $\frac{3}{2} \frac{\alpha_{3,0,1}y^2}{\sqrt{q}}(\partial_{y}Q_2\cdot \omega)$ contributed by $J_1(\sqrt{q})=(\partial_{y}\sqrt{q}) (\partial_{y}Q_{2}\cdot \omega)$. $J_1(v)$ is defined in \eqref{def:J}.

Next we complete the proof by providing finer descriptions for some of the functions $\gamma_{n,k,l}$.

The estimates for $\gamma_{2,0,1}$ and $\gamma_{2,2,l}$, in \eqref{eq:2021} and \eqref{eq:2kg2}, can be refined. $J_{7}$ and $J_8$, defined in \eqref{def:J7J8}, are parts of $F_2(\sqrt{q})$ and make, for some constants $d_{2,0,1}$ and $d_{2,2,l}$,
\begin{align}
\begin{split}
    \gamma_{2,0,1}(\tau)=&d_{2,0,1}e^{-\tau}+\mathcal{O}(e^{-\frac{4}{3}\tau}),\\
    \gamma_{2,2,l}(\tau)=&d_{2,2,l}e^{-\tau}+\mathcal{O}(e^{-\frac{4}{3}\tau}).
\end{split}
\end{align}
The estimate for $\gamma_{2,1,l}$ in \eqref{eq:21l} can be improved by feeding the estimates above into $NL_{2,1,l},$
\begin{align}
    \gamma_{2,1,l}(\tau)=-e^{-\frac{1}{2}\tau}\int_{\tau}^{\infty} e^{\frac{1}{2}\sigma} NL_{2,l,1}(\sigma)\ d\sigma=\mathcal{O}(e^{-\frac{4}{3}\tau}).
\end{align}
Similarly, for any $k\geq 3$, study \eqref{eq:2kg2} to obtain that
\begin{align}
    |\gamma_{2,k,l}|\lesssim e^{-\frac{4}{3}\tau}.
\end{align}

Thus we proved all the desired estimates.

%%%%%%%%%%%%%%%%%%%%%%%%%%%%%%%%%%%%%%%%%%%%%%%%%%%%%%%%%%%%%%%%%%%%%%%%%%%%%%%%%%%%%%%%%%

\section{Proof of the Part C in Theorem \ref{THM:sec} by assuming Part B}\label{sec:step1Third}
Recall that we parametrize the rescaled MCF as
\begin{align}\label{eq:S3}
\Psi_{Q_{m-1},v}&=\left[
\begin{array}{ccc}
y\\
Q_{m-1}(y,\tau)
\end{array}
\right]+v(y,\omega,\tau)\left[
\begin{array}{ccc}
-\partial_{y}Q_{m-1}(y,\tau)\cdot\omega\\
\omega
\end{array}
\right],
\end{align} 
and the vector-valued function $Q_{m-1}\in \mathbb{R}^4$ takes the form, for some constants $c_{n,k},$
\begin{align*}
Q_{m-1}(y,\tau)= \sum_{n=2}^{m-1}H_{n}(y) e^{-\frac{n-1}{2}\tau}\Big(a_{n,1},a_{n,2},a_{n,3},a_{n,4}
\Big)^{T}.
\end{align*}
The corresponding MCF, denoted as $\Phi_{\Pi_{m-1},u}$, takes the form
\begin{align}\label{eq:mcfMean}
    \Phi_{\Pi_{m-1},u}:=\left[
\begin{array}{ccc}
z\\
\Pi_{m-1}(z,t)
\end{array}
\right]+u(z,\omega,t)\left[
\begin{array}{ccc}
-\partial_{z}\Pi_{m-1}(z,t)\cdot\omega\\
\omega
\end{array}
\right]
\end{align} by the identity
\begin{align}
    \frac{1}{\sqrt{T-t}}\Phi_{\Pi_{m-1},u}=\Psi_{Q_{m-1},u}\label{eq:reMCFToRes}
\end{align} where $u$, $z$ and $\Pi_{m-1}$ are naturally defined in terms of $v$, $y$ and $Q_{m-1},$ and $\tau:=-\ln(T-t).$

In our paper \cite{GZ2018}, we considered the nondegenerate cases with parametrization
\begin{align}\label{eq:S1}
\Psi_{0,v}(y,\theta,\tau)&=\left[
\begin{array}{ccc}
y\\
0\\
0
\end{array}
\right]+v(y,\theta,\tau)\left[
\begin{array}{ccc}
0\\
\sin(\theta)\\
\cos(\theta)
\end{array}
\right],\ y\in \mathbb{R}^3,\ \theta\in [0,\ 2\pi).
\end{align}

What makes \eqref{eq:S3} different from \eqref{eq:S1} is the vector $Q_{m-1}$. However it is not an obstacle: in the considered region the curve $\left[
\begin{array}{ccc}
y\\
Q_{m-1}(y)
\end{array}
\right]$ varies very slowly, for any sufficiently large $\tau$,
\begin{align}
|\partial_{y}Q_{m_1}(y)|\ll e^{-(\frac{1}{2}-\frac{m-2}{2m})\tau} e^{20\sqrt{\tau}} \ll 1,\label{eq:curve}
\end{align} hence locally it is almost a straight line.

Consequently, the proof is similar to that in \cite{GZ2017}. However, here we will present a self-contained proof. 

Now we consider MCF in (\ref{eq:mcfMean}). For each fixed small $z_1\not=0$, we will study a small neighborhood of it in two temporal regions, defined as 
\begin{align}
\Big\{\ t(\tau)\ \Big| \ |y|=e^{\frac{1}{2}\tau}|z_1|\geq e^{\frac{m-2}{2m}\tau(t)} e^{5\sqrt{\tau(t)}} \Big\} \ \text{and}\ \Big\{\ t(\tau)\ \Big| \ |y|=e^{\frac{1}{2}\tau}|z_1| < e^{\frac{m-2}{2m}\tau(t)} e^{5\sqrt{\tau(t)}} \Big\}.
\end{align}Recall that in the rescaled MCF, $z$ and $t$ are rescaled into $y=\frac{z}{\sqrt{T-t}}=e^{\frac{1}{2}\tau}z$ and $\tau=-\ln (T-t)$. The treatment in the second region is easier since the estimates \eqref{eq:pointwisedm} apply. In the first region we need new techniques, since the ones we used for the rescaled MCF become less and less powerful as $|y|$ increases.

%%%%%%%%%%%%%%%%%%%%%%%%%%%%%%%%%%%%%%%%%%%%%%%%%%%%%%%%%%%%%%%%%%%%%%%%%%%%%%%%%%

We start with studying the first region.

%%%%%%%%%%%%%%%%%%%%%%%%%%%%%%%%%%%%%%%%%

\subsection{Analysis in the temporal region \texorpdfstring{$\Big\{\ t(\tau)\ \Big| \ e^{\frac{1}{2}\tau(t)}|z_1|\geq e^{\frac{m-2}{2m}\tau(t)} e^{5\sqrt{\tau(t)}}\Big\}$}{} }
We start with defining the time $\tau_1\gg 1$ to be the unique time such that 
\begin{align}
    |y|=e^{\frac{1}{2}\tau_1}|z_1|= e^{-\frac{m-2}{2m}\tau_1} e^{5\sqrt{\tau_1}}.
\end{align}

We need to retrieve some useful information from the rescaled MCF.
We consider the region
\begin{align}
e^{\frac{m-2}{2m}\tau_1} e^{5\sqrt{\tau_1}}-e^{\tau_1^{\frac{3}{4}}}\leq |y|\leq  e^{\frac{m-2}{2m}\tau_1} e^{5\sqrt{\tau_1}}+e^{\tau_1^{\frac{3}{4}}}.\label{eq:chosen}
\end{align} This is inside the region we chose in (\ref{eq:region}), and thus the estimate (\ref{eq:pointwisedm}) applies.
It implies that
\begin{align}\label{eq:LargeRegion}
v(y,\omega,\tau_1)=\lambda_1\Big(1+o(1)\Big)
\end{align} where $\lambda_1=\lambda(\tau_1)$ is a large constant defined as
\begin{align}
    \lambda_1:=\sqrt{d_m} e^{\frac{5m\sqrt{\tau_1}}{2}}\gg 1,
\end{align} recall that the Hermite polynomials $H_{n}$ take the forms in \eqref{eq:hermit}.
In contrast to \eqref{eq:LargeRegion}, when $y=0,$
\begin{align}
v(0,\omega,\tau_1)=\sqrt{6}\Big(1+o(1)\Big).\label{eq:vYzero}
\end{align}

\eqref{eq:LargeRegion} and \eqref{eq:vYzero} are crucial for our purpose. To see this, define $t_1$ to be the unique time $t_1:=t(\tau_1)$,
through the identity \eqref{eq:reMCFToRes}, they imply that in the corresponding region of MCF,
\begin{align}\label{eq:regionTau1}
e^{-\frac{1}{2}\tau_1}\Big( e^{-\frac{m-2}{m}\tau_1} e^{5\sqrt{\tau_1}}-e^{\tau_1^{\frac{3}{4}}}\Big)  \leq   |z|\leq e^{-\frac{1}{2}\tau_1}\Big(e^{-\frac{m-2}{m}\tau} e^{5\sqrt{\tau_1}}+e^{\tau_1^{\frac{3}{4}}}\Big),
\end{align} the following estimates hold for $u$,
\begin{align}
    u(z,\omega,t_1)=\lambda_1 \sqrt{T-t_1}\Big(1+o(1)\Big);
\end{align} at $z=0$, where a singularity will form at $t=T,$ \eqref{eq:vYzero} implies
\begin{align}
    u(0,\lambda,t_1)=\sqrt{6}\sqrt{T-t_1}\Big(1+o(1)\Big).
\end{align}
These, together with the smoothness estimates provided through the identity
\begin{align}
v^{-j+1}\partial_{y}^j \nabla_{E}^{i} v(y,\omega,\tau)=u^{-j+1}\partial_{z}^{j}\nabla_{E}^{i} u(z,\omega,t)\label{eq:ThrouIdentity}
\end{align} and the estimates in \eqref{eq:pointwisedm},
indicate that when a singularity is formed at $z=0$, the MCF stays smooth in the region \eqref{eq:regionTau1}.

To make these ideas rigorous and to make the proof transparent, we define a new MCF by rescaling the old one $\Phi_{\Pi_{m-1},u}$ by a factor $\lambda_1\sqrt{T-t_1}$, in order to achieve (\ref{eq:p1}) below in the interested region.
Specifically the new MCF takes the form, for some function $p,$
\begin{align}\label{eq:flow2}
\begin{split}
\left[
\begin{array}{ccc}
x\\
\tilde{Q}(x,s)
\end{array}
\right]+p(x,\omega,s)\left[
\begin{array}{ccc}
-\partial_{x}\tilde{Q}(x,s)\cdot\omega\\
\omega
\end{array}
\right]=&\frac{\sqrt{1-\lambda_1^2 s}}{\lambda_1 }\Psi_{Q_{m-1},v}(y,\omega,\tau)\\
=&\frac{1}{\lambda_1 \sqrt{T-t_1}}\Phi_{\Pi_{m-1},u}(z,\omega,t)
\end{split}
\end{align}
where 
$p(x,\theta,s)$ is defined in terms of $v(y,\omega,\tau)$, and hence of $u(z,\omega,t)$ through the identity \eqref{eq:ytau},
\begin{align}
\begin{split}\label{eq:restart}
p(x,\omega,s):
=&\frac{1}{\lambda_1 \sqrt{T-t_1}} u(\lambda_1 \sqrt{T-t_1} x, \omega,  \lambda_1^2 (T-t_1) s+t_1)\\
=&\frac{\sqrt{1-\lambda_1^2 s}}{\lambda_1}\ v\big(\frac{\lambda_1 x}{\sqrt{1-\lambda_1^2 s}},\omega, -\ln(1-\lambda_1^2 s)+\tau_1\big),
\end{split}
\end{align} and $\tilde{Q}\in \mathbb{R}^4$ is a vector-valued function defined as 
\begin{align}
    \tilde{Q}(x,s):=\frac{1}{\lambda_1} Q(y,\tau)=\frac{1}{\lambda_1}Q(\frac{\lambda_1}{\sqrt{1-\lambda_1^2 s}}x, -\ln (1-\lambda_1^2 s)+\tau_1)
\end{align} and $x$ and $s$ are spatial and time variables defined in terms of $z$ and $t$,
\begin{align}
s:=\frac{t-t_1}{\lambda_1^2 (T-t_1)}=\frac{1-e^{-(\tau-\tau_1)}}{\lambda_1^2},\ \text{and}\ x:=\frac{1}{\lambda_1 \sqrt{T-t_1}}z=\frac{\sqrt{1-\lambda_1^2 s}}{\lambda_1 } y.
\end{align} 
The blowup time of the new MCF is $\lambda_1^{-2}\ll 1$ since, as $s\rightarrow \lambda_1^{-2}$, $\tau=-ln(1-\lambda_1^{2}s)+\tau_1\rightarrow \infty.$ 

Now we make preparations for studying $p(x,\omega,s)$ when $s\in [0,\lambda_1^{-2}]$ and $x$ is in the large region
\begin{align}
\lambda_1^{-1}\Big(e^{-\frac{m-2}{m}\tau_1} e^{5\sqrt{\tau_1}}-e^{\tau_1^{\frac{3}{4}}}\Big)\leq |x|\leq  \lambda_1^{-1}\Big(e^{-\frac{m-2}{m}\tau} e^{5\sqrt{\tau_1}}+e^{\tau_1^{\frac{3}{4}}}\Big).\label{eq:xregion}
\end{align} 

The main tools are the standard techniques of local smooth extension and interpolation between estimates of derivatives, see e.g. \cite{EckerBook, ColdingMiniUniqueness}. 

To prepare for the applications, we observe that, the scaling in \eqref{eq:invar} implies the identities 
\begin{align}\label{eq:ptov}
p^{-j+1}\partial_{x}^j \nabla_{E}^i p(x,\omega,s)=v^{-j+1}\partial_{y}^j \nabla_{E}^i v(y,\omega,\tau).
\end{align} Moreover, the estimates for $v$ in \eqref{eq:pointwisedm} imply that
\begin{align}\label{eq:invar}
  \sum_{j+|i|=1,2} v^{-j+1}\Big|\partial_{y}^j\nabla_{E}^{i} v(y,\omega,\tau)\Big|\ll  e^{-5\sqrt{\tau_1}} \ll 1,\ \text{when}\ |y|\leq e^{\frac{m-2}{2m}\tau+8\sqrt{\tau}}.
\end{align}These provide the corresponding ones for $p$
when $s\in [-1,0]$ and $x$ is in the region \eqref{eq:xregion}. Especially at the time $s=0$ and in the region \eqref{eq:xregion}
we have that
\begin{align}
p(x,\omega,0)=1+o(1)\label{eq:p1}.
\end{align} And when $s\in [-1,0]$ and $x$ is in the same region, there exists a constant $C>0$ such that
\begin{align}
\begin{split}
p(x,\omega,s)\in& [\frac{9}{10},\ \sqrt{10}],\\
\Big|\partial_{x}^k \nabla_{E}^l p(x,\omega,s)\Big|\leq & C e^{-\sqrt{\tau_1}}.
\end{split}
\end{align}

By \eqref{eq:curve}, in the interested region, the curve $\left[
\begin{array}{ccc}
x\\
\tilde{Q}(x,s)
\end{array}
\right]$ is almost a straight line.

These estimates make the tools of local smooth extension and interpolation between estimates of derivatives applicable. They imply that, in the small time interval $$s\in [0,\lambda_1^{-2}],$$ and in the region 
\begin{align}
\lambda_1^{-1}\Big(e^{-\frac{m-2}{m}\tau_1} e^{5\sqrt{\tau_1}}-\frac{9}{10} e^{\tau_1^{\frac{3}{4}}}\Big)\leq |x|\leq  \lambda_1^{-1}\Big(e^{-\frac{m-2}{m}\tau} e^{5\sqrt{\tau_1}}+\frac{9}{10} e^{\tau_1^{\frac{3}{4}}}\Big),
\end{align} 
the following estimates hold, for any integer $N$, there exists some $\delta_{N}(\tau_1)>0$ satisfying $\lim_{\tau_1\rightarrow \infty}\delta_{N}(\tau_1)=0,$ 
\begin{align}
\begin{split}
p(x,\omega,s)\in \Big[\frac{9}{10},\ \frac{11}{10}\Big],\\
\sum_{1\leq k+|l|\leq N}\Big|p^{-k+1}\partial_{x}^k \nabla_{E}^l p(x,\omega,s)\Big|\leq & \delta_{N}(\tau_1).
\end{split}
\end{align}
From here it is easy to prove that the considered part remains smooth, and is mean convex and the singularity is isolated. Recall that $s=\lambda_1\ll 1$ is the blowup time, and the singular point is $z=0.$

\subsection{Analysis in the temporal region \texorpdfstring{$\Big\{\ t(\tau)\ \Big| \ e^{\frac{1}{2}\tau(t)}|z_1| < e^{\frac{m-2}{2m}\tau(t)} e^{5\sqrt{\tau(t)}} \Big\}$}{}}
Compared to the previous subsection, the analysis here is significantly easier since, when $j+|i|\leq 2,$ the estimates in \eqref{eq:pointwisedm} provide bounds for $u^{j-1} \partial_{z}^{j}\nabla_{E}^{i}u$ through the following identity: for any nonnegative integer $j$ and $i\in (\mathbb{N}\cup\{0\})^3,$
\begin{align}
v^{-j+1}\partial_{y}^{j}\nabla_{E}^{i}v(y,\omega,\tau)=u^{-j+1} \partial_{z}^{j}\nabla_{E}^{i}u(z,\omega,t).\label{eq:nmIdentity}
\end{align}
When $j+|i|>2$, we control $\partial_{y}^{j}\nabla_{E}^{i}v$ by defining a new MCF similar to that in \eqref{eq:flow2}, applying the standard technique of smoothness estimates, and then using the identities \eqref{eq:flow2} and (\ref{eq:nmIdentity}). We choose to skip this part.

%%%%%%%%%%%%%%%%%%%%%%%%%%%%%%%%

\section{Reformulation of the part B in Theorem \ref{THM:sec}}\label{sec:BTHMSec}
Here we are interested in the region
\begin{align}
    \Big\{y\ \Big| \ |y|\leq (1+\epsilon)Z_m(\tau)\Big\}, \ \text{and}\ \tau\geq T_1,
\end{align}
where $T_1\gg 1$ is a constant, and $Z_{m}$ is a scalar function defined as
\begin{align}\label{def:Zm}
    Z_m(\tau):=e^{\frac{m-2}{2m}(\tau-T_1)}e^{10\sqrt{\tau-T_1}}+5\tau^{\frac{1}{2}+\frac{1}{20}},
\end{align} the constant $\epsilon$ will appear in \eqref{def:chi} below.
The reason for choosing such a region is to make the results in \eqref{eq:qmminus1}-\eqref{eq:GLEtaM} applicable when $\tau=T_1$, which will be used to prove Lemma \ref{LM:startBoot} below.

To restrict our attention to such a region, we define a cutoff function $\chi_{Z}$ as
\begin{align}
    \chi_{Z}(y):=\chi(\frac{y}{Z_m}).
\end{align} $\chi$ is a cutoff function satisfying the following conditions: $\chi(x)=\chi(|x|)$ is non-increasing in $|x|$; and
\begin{align}\label{def:chi}
    \chi(x)=\left[
    \begin{array}{cc}
      1   &  \text{when}\ |x|\leq 1;\\
       0  & \text{when}\ |x|\geq 1+\epsilon;
    \end{array}
    \right.
\end{align}
it is a $C^{2m+20}_1$ function, and when $0<\epsilon-|x|\ll 1$, there exist constants $C_j$, $j=0,\cdots,m+2,$ s.t.
\begin{align}
\begin{split}\label{eq:chiSmooth}
    \frac{d^j}{dx^j}\chi (x)=& C_j(1-|x|)^{2m+20-j}\Big(1+o(1)\Big).
\end{split}
\end{align} This will be needed in \eqref{eq:TwoCut} below to control the functions $\frac{z\chi^{'}}{\chi}\frac{d^k}{dz^k}\chi,\ 1\leq k\leq m+1.$

For the strategy of the proof, we will prove the desired results by bootstrap. This is necessary. Our goal is to prove \eqref{eq:vdecomMsup}-\eqref{eq:smallM} below for any $\tau\geq T_1\gg 1.$ To achieve this we have technical difficulties: before proving these results in an interval $\tau\in [T_1,\ T_2]$, we need the existence of the rescaled MCF and some primitive estimates; on the other hand only after proving sufficiently good estimates, for example \eqref{eq:vdecomMsup}-\eqref{eq:smallM} in an interval $[T_1,T_2]$,  one can prove the existence and primitive estimates in a slightly larger interval $[T_1,T_2+\kappa]$ for some $\kappa>0.$ 

Thus in order to prove the desired result we need to bootstrap. 

To initiate the bootstrap we use the following results:
\begin{lemma}\label{LM:startBoot}
For any small $\epsilon_1>0$ there exists a time $T_0$ such that, when $\tau\geq T_0$ and $|y|\leq 7 \tau^{\frac{1}{2}+\frac{1}{20}}$, 
\begin{align}\label{eq:initiaBoot}
    \Big|v-\sqrt{6+d_m e^{-\frac{m-2}{2}\tau}H_m}\Big|+ \sum_{j+|l|=1}^{2m+4}\Big|\partial_{y}^{j}\nabla_{E}^{l}v\Big|\leq \epsilon_1.
\end{align}
\end{lemma}
This is implied by the technique of smoothness estimates and the smallness estimates in \eqref{eq:qmminus1}-\eqref{eq:GLEtaM}. The proof is tedious, but easy, and hence is skipped.

We are ready to state our results.  Recall that we are considering the following cases: (1) $d_m>0$ and $m\geq 4$ is even, (2) $d_m=0$ and $m\geq 3$ is even or odd.

In the first step of bootstrap, we prove the following results.
\begin{proposition}\label{prop:Mconvex}
There exists a small constant $\delta$ such that if in the space-and-time region
\begin{align}
\tau\in [T_1,\tau_1],\ y\in \Big\{y\ \Big| \ |y|\leq (1+\epsilon) Z_{m}(\tau)\Big\},\label{eq:induRegionSup}
\end{align} $v$ satisfies the estimates
\begin{align}\label{eq:conditionZm}
\Big|\frac{v-\sqrt{6+d_m e^{-\frac{m-2}{2}\tau}H_m}}{v}\Big|+ \sum_{m+|l|=1}^{2m+4}v^{m-1}\Big|\partial_{y}^{m}\nabla_{E}^{l}v\Big|\leq \delta,
\end{align}
then \eqref{eq:vdecomMsup}-\eqref{eq:smallM} below hold in the same space-and-time region.

For any sufficiently large integer $N$, there exist functions $\alpha_{j,k,l}$ such that 
\begin{align}
    v(y,\omega,\tau)=\sqrt{6+\sum_{j=0}^{m} \sum_{k=0}^{N}\sum_{l}\alpha_{j,k,l}(\tau) H_{j}(y)f_{k,l}(\omega)+\eta(y,\omega,\tau)},\label{eq:vdecomMsup}
\end{align} and $\chi_{Z}\eta$ satisfies the orthogonality conditions
\begin{align}\label{eq:ZEtaOrtho}
    \chi_{Z}\eta\perp_{\mathcal{G}},\ H_{n} f_{k,l}, \ 0\leq n\leq m, \ \text{and}\ k\leq N,
\end{align}
$\alpha_{j,k,l}$ and the $\mathcal{G}-$norm of $\chi_{Z}\eta$ enjoy the same estimates to the corresponding parts in \eqref{eq:dm}-\eqref{eq:mGnormeta}.

Certain weighted $L^{\infty}$-norm of $\chi_{Z}\eta$ satisfies the inequality,
\begin{align}\label{eq:etaMsup}
\mathcal{M}(\tau)\lesssim  \kappa \Big(1+\mathcal{M}(\tau)\Big)+\mathcal{M}^2(\tau),
\end{align} where $\mathcal{M}$ is to be defined in (\ref{eq:major2}) below, $\kappa$ is a small constant satisfying that $\kappa\rightarrow 0$ as $\delta+T_1^{-1}\rightarrow 0.$ This, together with the condition $\mathcal{M}(T_1)\lesssim \delta\ll 1$ implied by Lemma \ref{LM:startBoot} and a standard application of Granwall's inequality, implies that
\begin{align}
    \mathcal{M}(\tau)\ll 1,\ \text{for any}\ \tau\in [T_1, \tau_1].\label{eq:smallM}
\end{align}

\end{proposition}
The proposition will be proved in subsection \ref{subsec:propMconvex}.

Now we define the function $\mathcal{M}$ used in (\ref{eq:etaMsup}). Before that we define the following functions, 
\begin{align}\label{def:majorants}
\mathcal{M}_{n,j}(\tau):=\sup_{T_1\leq s\leq \tau}\tilde{Z}_m^{n}(s) Y_{T_2}(s) \Big\|\langle y\rangle^{-n}\|\big(-\Delta_{\mathbb{S}^3}+1\big)^5 \partial_{y}^{j}\chi_{Z} \eta(\cdot,s)\|_{\mathbb{S}^3}\Big\|_{\infty},
\end{align} where the pair $(n,j)$ belongs to the following set
\begin{align*}
    \Upsilon_m:=&\Big\{(m+1,0),\ (m+\frac{1}{2},1),\ (m, 2),\ (m-\frac{1}{2},3),\ \cdots,\ (0,2m+2)\Big\}\\
    =&\Big\{ \Big(m+1-\frac{j}{2},j\Big)\ \Big|\ j=0,1,\cdots, 2m+2 \Big\},
\end{align*} and $\tilde{Z}_m$ is a function defined as 
\begin{align}\label{def:TiZm}
\tilde{Z}_{m}(\tau):=e^{\frac{m-2}{2m}(\tau-T_1)}e^{10\sqrt{\tau-T_1}};
\end{align}
and $Y_{T_2}$ is a function defined as
\begin{align}\label{def:YT2}
Y_{T_2}(\tau):=e^{20m\sqrt{\max\big\{s-T_2,\ 0\big\}}},    
\end{align} $T_2>T_1$ is the unique time defined by the identity
$\tilde{Z}_{m}(T_2)=T_2^{\frac{1}{2}+\frac{1}{20}}$. We will need $T_2$ and $Y_{T_2}$ to overcome some minor difficulties in (\ref{eq:pureTech2}) and (\ref{eq:pureTech}) below. The important information is that, $$Y_{T_2}(\tau)=e^{20m\sqrt{\tau}}\Big(1+o(1)\Big),\ \text{as} \ \tau\rightarrow \infty.$$

We define the function $\mathcal{M}$ as
\begin{align}\label{eq:major2}
\mathcal{M}(\tau):=\sum_{(n,j)\in \Upsilon_{m}}\mathcal{M}_{n,j}(\tau).
\end{align}
The reasons for choosing such norms will be explained after Proposition \ref{prop:ZMajorRefor}, when we are ready.

The next result is the second and last step of bootstrap:
\begin{proposition}\label{prop:localSmo}
If Lemma \ref{LM:startBoot} holds and $v$ satisfies the estimates \eqref{eq:vdecomMsup}-\eqref{eq:smallM} in the region
$$\tau\in [T_1,\tau_1],\ y\in \Big\{y\ \Big| \ |y|\leq (1+\epsilon) Z_{m}(\tau)\Big\},$$
then, for any $\delta>0$, provided that $T_1$ is large enough, there exists a small $\zeta>0$ such that in the space-and-time region
\begin{align}
\tau\in [T_1,\tau_1+\zeta],\ \text{and}\  y\in \Big\{y\ \Big| \ |y|\leq (1+\epsilon) Z_{m}(\tau)\Big\},
\end{align}  the following estimates hold,
\begin{align}
\Big|\frac{v-\sqrt{6+d_m e^{-\frac{m-2}{2}\tau}H_m}}{v}\Big|+ \sum_{j+|i|=1}^{2m+4}v^{j-1}\Big|\partial_{y}^{j}\nabla_{E}^{i}v\Big|\leq \delta.
\end{align}
\end{proposition}
\begin{proof}
We consider two temporal regions: $\{\tau \ | \ \tilde{Z}_{m}(\tau)\leq \tau^{\frac{1}{2}+\frac{1}{20}}\}$ and $\{\tau\ |\ \tilde{Z}_{m}(\tau)>\tau^{\frac{1}{2}+\frac{1}{20}}\}.$

In the first region $Z_{m}\leq 6\tau^{\frac{1}{2}+\frac{1}{20}},$ Lemma \ref{LM:startBoot} implies the desired results. 

In the second region we need \eqref{eq:vdecomMsup}-\eqref{eq:smallM}. Since $\mathcal{M}\leq 1$, the definition of $\mathcal{M}$ imply that
\begin{align}
    \sum_{j+|i|\leq 2}\Big|\partial_{y}^j\nabla_{E}^{i}\chi_{Z}\eta(\cdot,\tau)\Big|\leq e^{-10m\sqrt{\tau-T_1}}.
\end{align} This and the estimates on the other functions imply that, when $|y|\leq Z_{m}(\tau),$
\begin{align}
    \Big|v-\sqrt{6+d_m e^{-\frac{m-2}{2}\tau}H_m}\Big|+ \sum_{j+|i|=1,2}\Big|v^{j-1}\partial_{y}^{j}\nabla_{E}^{i}v\Big|\lesssim e^{-\sqrt{\tau-T_1}}.
\end{align}
From here we obtain the desired results by comparing MCF and the rescaled MCF, applying the standard techniques of local smooth extension and interpolation between estimates between derivatives. We choose to skip the detail since we considered the same problem in our previous paper \cite{GZ2018}, and these techniques were applied in \cite{ColdingMiniUniqueness}. Moreover in Section \ref{sec:step1Third} we used these techniques to prove a more difficult result.
\end{proof}

Assuming Proposition \ref{prop:Mconvex} we are ready to prove Part II of Item [B] in Theorem \ref{THM:sec}.

%%%%%%%%%%%%%%%%%%%%%%%%%%%%%%%%%%%%%%%%%%%%%%%%%%%

\subsection{Proof of Part B in Theorem \ref{THM:sec}}
\begin{proof}
As we discussed before, we need to bootstrap. To initiate it we need (\ref{eq:initiaBoot}). By the definition of $Z_m(\tau)$ in \eqref{def:Zm}, if we choose a sufficiently large $T_1,$ then (\ref{eq:initiaBoot}) holds in the space and time region
\begin{align}
\tau\in [T_1,T_1+20],\ \text{and}\ |y|\leq (1+\epsilon)Z_m(\tau).
\end{align} Consequently the condition \eqref{eq:conditionZm} is satisfied, hence the results \eqref{eq:vdecomMsup}-\eqref{eq:etaMsup} in Proposition \ref{prop:Mconvex} hold in the same space and time region.

For the next step of bootstrap, suppose that \eqref{eq:vdecomMsup}-\eqref{eq:etaMsup} in Proposition \ref{prop:Mconvex} hold in a space-and-time region, for some $T\geq T_1+20$,
\begin{align}
\tau\in [T_1,T],\ \text{and}\ |y|\leq (1+\epsilon)Z_m(\tau).
\end{align} Thus Proposition \ref{prop:localSmo} becomes applicable. This, in turn, implies that, for some $\zeta>0$, in the region,
\begin{align}
\tau\in [T_1,T+\zeta],\ \text{and}\ |y|\leq (1+\epsilon)Z_m(\tau),
\end{align} (\ref{eq:conditionZm}) holds. Consequently \eqref{eq:vdecomMsup}-\eqref{eq:etaMsup} actually hold in this larger region.

By continuity we prove that the desired results \eqref{eq:vdecomMsup}-\eqref{eq:etaMsup} hold in the region
\begin{align}
\tau\in [T_1,\infty),\ \text{and}\ |y|\leq (1+\epsilon)Z_m(\tau).
\end{align}

Those results imply the desired Part B in Theorem \ref{THM:sec}.
\end{proof}

%%%%%%%%%%%%%%%%%%%%%%%%%%%%%%%%%%%%%%%%%%%%%%%%%%%%%%%%%%%%%%%

\subsection{Proof of Proposition \ref{prop:Mconvex}}\label{subsec:propMconvex}
We start with some preliminary estimates for $\alpha_{n,k,l}$ and the $\mathcal{G}$-norm of $\eta$ by 
comparing them to $\gamma_{n,k,l}$ and $\xi$ in \eqref{eq:newdec}. Recall that $1_{\leq R}$ is the Heaviside function.
\begin{lemma}\label{LM:fixedPoint} The functions $\eta$ and $\alpha_{n,k,l}$ satisfy the following estimates,
\begin{align}
\Big\|\chi_{Z}\eta(\cdot,\tau)-1_{\leq R}\xi(\cdot,\tau)
\Big\|_{\mathcal{G}}+ &
    \Big|\gamma_{n,k,l}(\tau)-\alpha_{n,k,l}(\tau)\Big|\leq C_{\delta_0} e^{-\frac{m-1-\delta_0}{2}\tau}.\label{eq:DiffAlBe}
    \end{align}
\end{lemma}
\begin{proof}

For \eqref{eq:DiffAlBe}, the existence of the functions $\alpha_{n,k,l}$ to make $\chi_{Z}\eta$ satisfy the orthogonality conditions (\ref{eq:ZEtaOrtho}) is equivalent to find $\alpha_{n,k,l}$ such that, for any $i\leq m$ and $j\leq N,$
\begin{align}
G_{i,j,h}(\alpha):=\Big\langle \chi_{Z} \Big(v^2-6-\sum_{n,k,l} \alpha_{n,k,l} H_n f_{k,l}\Big),\ H_{i}f_{j,h} \Big\rangle_{\mathcal{G}}=0.
\end{align}

The main tool is the Fixed Point Theorem. To make it applicable we form a $N\times 1$ vector $G(\alpha)$ from $G_{i,j,h}(\alpha),$ with $N$ being the total number of these functions.  Then it is easy to prove that a $N\times N$ matrix formed by $\partial_{\beta_{n,k,l}}G(\alpha)$ is uniformly invertible.

Moreover, the estimates in \eqref{eq:newdec}-\eqref{eq:GLEtaM2} imply that
\begin{align*}
\tilde{G}_{i,j,h}(\gamma):=\Big\langle 1_{\leq R} \Big(v^2-6-\sum_{n,k,l}\gamma_{n,k,l} H_n f_{k,l}\Big),\ H_{i}f_{j,h} \Big\rangle_{\mathcal{G}},
\end{align*} 
satisfy the estimates
\begin{align}
    |\tilde{G}_{i,j,h}(\gamma)|\leq C_{\delta_0} e^{-\frac{m-1-\delta_0}{2}\tau},
\end{align}
thus 
\begin{align}
|G_{i,j,h}(\alpha)|=\Big|\Big\langle (\chi_{Z}-1_{\leq R})\Big(v^2-6-\sum_{n,k,l} \alpha_{n,k,l} H_n f_{k,l}\Big), H_{i}f_{j,h}  \Big\rangle_{\mathcal{G}} \Big|\leq C_{\delta_0} e^{-\frac{m-1-\delta_0}{2}\tau}.
\end{align}Now we apply the fixed point theorem to obtain the desired result 
\begin{align}
 \sum_{n,k,l}  \Big|\beta_{n,k,l}(\tau)-\alpha_{n,k,l}(\tau)\Big|\leq C_{\delta_0} e^{-\frac{m-1-\delta_0}{2}\tau},
\end{align}
which is the second part of \eqref{eq:DiffAlBe}. Moreover this implies the first part. 

\end{proof}

The estimates provided by Lemma \ref{LM:fixedPoint} are not satisfactory, but only need a slight improvement. For that we need to study their governing equations.
Define a function $q$ as
\begin{align*}
    q:=&6+\sum_{n=0}^m\sum_{k=0}^{N}\sum_{l=1} \alpha_{n,k,l}(\tau) H_n(y)f_{k,l}(\omega).
\end{align*}
We take the main part of $q$, and denote it by $q_{M},$
\begin{align*}
q_M(y,\tau):=\left[
\begin{array}{cc}
  6+\alpha_{m,0,1}(\tau)H_m(y) & \ \text{when} \ d_m>0 \ \text{and}\ m\ \text{is even},\\
  6 &\ \text{when}\ d_m=0.
\end{array}
\right.
\end{align*}
Similar to deriving \eqref{eq:EffTilXi}, we derive a governing equation for $\chi_{Z}\eta$,  
\begin{align}\label{eq:mSupeta}
\partial_{\tau}\chi_{Z}\eta=-L\chi_{Z}\eta+\chi_{Z}\Big(G+SN\Big)+\mu_{Z}(\eta)
\end{align}where the linear operator $L$ is defined as
\begin{align*}
L:=-\partial_{y}^2+\frac{1}{2}y\partial_{y}-1-\frac{1}{q_M}\Delta_{\mathbb{S}^3},
\end{align*}
and the function $G=G_1+G_2$ is defined as
\begin{align}
\begin{split}
    G_1:=&-\sum_{n,k,l} (\frac{d}{d\tau}+\frac{n-2}{2}+\frac{k(k+2)}{6})\alpha_{n,k,l} H_n f_{k,l},\\
    G_2:=&\frac{6-q}{6q}\Delta_{\mathbb{S}^3}q-\frac{1}{2q^2} |\nabla_{\omega}^{\perp}q|^2-\frac{1}{2q}|\partial_{y}q|^2+2\sqrt{q} N_1(\sqrt{q})+2\sqrt{q}W_{Q_{m-1}}(\sqrt{q}),
    \end{split}
\end{align} 
and the term $SN$ collects the terms nonlinear and ``small" linear in terms of $\eta,$
\begin{align}
\begin{split}
    SN:=&\frac{1}{2q^2}|\nabla_{\omega}^{\perp}q|^2-\frac{1}{2}v^{-4} |\nabla_{\omega}^{\perp}v^2 |^2+\frac{1}{2q}|\partial_{y}q|^2-\frac{1}{2}v^{-2}|\partial_{y}v^2|^2+2v N_1(v)-2\sqrt{q} N_1(\sqrt{q})\\
&+2vW_{Q_{m-1}}(v)-2\sqrt{q}W_{Q_{m-1}}(\sqrt{q})+\frac{q_M-q}{q_{M}q}\Delta_{\mathbb{S}^3}\eta-\frac{\eta}{v^2 q}\Delta_{\mathbb{S}^2}v^2.
\end{split}
\end{align} 
and the term $\mu_{Z}(\eta)$ is defined as
\begin{align}\label{def:muMsup}
\mu_{Z}(\eta):=\frac{1}{2}(y\partial_{y}\chi_{Z})\eta+(\partial_{\tau}\chi_{Z})\eta-\big(\partial_{y}^2\chi_{Z}\big)\eta-2\partial_{y}\chi_{Z}\cdot  \partial_{y}\eta.
\end{align}
From the orthogonality conditions \eqref{eq:ZEtaOrtho} we derive governing equations for $\alpha_{n,k,l},$
\begin{align}\label{eq:BetaNKLeqn}
\begin{split}
\Big(\frac{d}{d\tau}+\frac{n-2}{2}+\frac{k(k+2)}{6}\Big)\alpha_{n,k,l}=&\frac{1}{\|H_n f_{k,l}\|_{\mathcal{G}}^2}\Big\langle \chi_{Z}(G_2+SN)+\mu_{Z}(\eta),\ H_n f_{k,l}\Big\rangle_{\mathcal{G}}\\
=&\frac{1}{\|H_n f_{k,l}\|_{\mathcal{G}}^2} NL_{n,k,l},
\end{split}
\end{align} where the functions $NL_{n,k,l}$ are naturally defined.

To improve the preliminary estimates provided by Lemma \ref{LM:fixedPoint}, we need some sharp estimates for $NL_{n,k,l}$. This will be achieved by feeding the preliminary ones into $NL_{n,k,l}.$ The result is following:
\begin{lemma}
The functions $NL_{n,k,l}$ satisfy the following estimates: for some constants $\tilde\alpha_{n,k,l}$ ,
\begin{align}
\begin{split}\label{eq:estBetaEff}
|NL_{m,k,l}|\lesssim &e^{-\frac{m-1}{2}\tau};\\
 NL_{n,k,l}= &\tilde\alpha_{n,k,l} e^{-\frac{n}{2}\tau}+\mathcal{O}(e^{-(\frac{n}{2}+\frac{1}{3})\tau}),\ \text{when}\ 2\leq n\leq m-1,\\
 e^{\frac{1}{2}\tau}|NL_{0,k,l}|&+e^{\tau}|NL_{1,k,l}|\lesssim  1.
 \end{split}
\end{align}
\end{lemma}
\eqref{eq:estBetaEff} will be proved in Section \ref{sec:source} below.

From \eqref{eq:estBetaEff} and \eqref{eq:BetaNKLeqn} we derive the desired estimates for $\alpha_{n,k,l}.$ This is similar to the treatment in subsection \ref{subsec:afterNorm3}, we skip the detail here. The first estimate in \eqref{eq:DiffAlBe} and \eqref{eq:mGnormeta} imply the desired estimates for the $\mathcal{G}-$norm of $\chi_{Z}\eta$.

What is left is to prove \eqref{eq:etaMsup}. This is the most involved part and it will be reformulated into Proposition \ref{prop:estSourceNonL} below.

We start with presenting the difficulties and ideas in overcoming them. We will have to decompose $\chi_{Z}\eta$ into finitely many pieces and treat them by different techniques. 

To see this, decompose $\eta$ according to the spectrum of $-\Delta_{\mathbb{S}^3}$, recall that $f_{k,l},$ $k\geq 0$, are eigenvectors of $-\Delta_{\mathbb{S}^3}$ with eigenvalues $k(k+2)$,
\begin{align}
    \eta(y,\omega,\tau)=\sum_{k=0}^{\infty}\sum_{l}\eta_{k,l}(y,\tau)f_{k,l}(\omega),
\end{align} where the functions $\eta_{k,l}$ are defined as
\begin{align}
\eta_{k,l}(y,\tau):= \frac{1}{\langle f_{k,l}, f_{k,l}\rangle_{\mathbb{S}^3}}\Big\langle \eta(y,\cdot,\tau),\ f_{k,l}\Big\rangle_{\mathbb{S}^3}.
\end{align} Here and in the rest of the paper the inner product $\langle \cdot,\cdot\rangle_{\mathbb{S}^3}$ denotes the standard inner product for $L^2(\mathbb{S}^3)$ space.
Thus $\chi_{Z}\eta_{k,l}$ satisfies the equation
\begin{align}
\begin{split}\label{eq:ZEtaKL}
    \partial_{\tau}\chi_{Z}\eta_{k,l}=&-L_{k}\chi_{Z}\eta_{k,l}
    +\frac{1}{\langle f_{k,l},\ f_{k,l}\rangle_{\mathbb{S}^3}} \chi_{Z}\Big\langle G+SN(\eta),\ f_{k,l}\Big\rangle_{\mathbb{S}^3}+\mu_{Z}(\eta_{k,l}),
    \end{split}
\end{align}
where the linear operator $L_k$ is defined as
\begin{align}
    L_{k}:=-\partial_{y}^2+\frac{1}{2}y\partial_{y}-1+\frac{k(k+2)}{q_{M}}.
\end{align}

Now we present the difficulty. Since $|q_{M}-6|\ll 1$ in the region $\frac{|y|}{e^{\frac{(m-2)\tau}{m}}}\ll 1$, at least intuitively,
\begin{align}
    L_{k}\approx L_0-1+\frac{k(k+2)}{6}.
\end{align} By the spectral analysis after \eqref{def:linearL}, the eigenvalues of $L_0:=-\partial_{y}^2+\frac{1}{2}y\partial_{y}$ are $\frac{j}{2}, \ j=0,1,2,\cdots$. In order to prove $\chi_{Z}\eta_{k,l}$ decays faster than $e^{-\frac{m-2}{2}\tau}$, it is important to have some orthogonality conditions when $k$ is not large.
When $\frac{k(k+2)}{6}$ is sufficiently large, we will prove the function decays rapidly after some transformation and applying the maximum principle, see Section \ref{sec:MaxTiEta} below.

To make this rigorous, we decompose $\eta$ according to the spectrum of $-\Delta_{\mathbb{S}^3}$:
\begin{align}
\eta(y,\omega,\tau)=\sum_{k=0}^{N}\sum_{l}\eta_{k,l}(y,\tau) f_{k,l}(\omega)+\tilde\eta(y,\omega,\tau)
\end{align} where we require $N$ to be large so that 
\begin{align}
\text{when}\ k>N,\ \frac{k(k+2)}{6}-1\geq \frac{m}{2} .\label{eq:largeNzm}
\end{align}This makes $\tilde\eta\perp_{\mathbb{S}^3} f_{k,l}$ for any $k\leq N$, or
\begin{align}
    P_{\omega,N} \tilde\eta= \tilde\eta;
\end{align} and by the orthogonality conditions imposed on $\chi_{Z}\eta$ in \eqref{eq:ZEtaOrtho},
\begin{align}
    \chi_{Z}\eta_{k,l}\perp_{\mathcal{G}} H_k,\ k=0,\cdots,m.
\end{align}
Here $P_{\omega, N}$ is an orthogonal projection defined as,
\begin{align}
    P_{\omega, N}g=\sum_{k>N}\sum_l g_{k,l} f_{k,l},\ \text{for any function}\ g(\omega)=\sum_{k=0}^{\infty}\sum_l g_{k,l} f_{k,l}(\omega).\label{def:OmePro}
\end{align}

Now we are ready to reformulate \eqref{eq:etaMsup}, before that we define a constant $\kappa$ as
\begin{align}
    \kappa:=\delta+e^{- \sqrt{T_1}}.\label{def:kappa}
\end{align}
\begin{proposition}\label{prop:ZMajorRefor} For any pair $(n,j)\in \Upsilon_m$ with $j\leq m+1$, we have that
\begin{align}
\begin{split}\label{eq:EtaKLlowN}
    Y_{T_2}\tilde{Z}_m^n\Big(\Big\|\langle y\rangle^{-n}\partial_{y}^j\chi_{Z}\eta_{k,l} \Big\|_{\infty}+
    \Big\| \langle y\rangle^{-n} \|\partial_{y}^j(-\Delta_{\mathbb{S}^3}+1)^5\chi_{Z}\tilde\eta\|_{L^2(\mathbb{S}^3)} \Big\|_{\infty}\Big)
    \lesssim \kappa \Big(1+\mathcal{M}\Big)+\mathcal{M}^2,
\end{split}
\end{align}
for any pair $(n,j)\in \Upsilon_m$ with $j\geq m+2$,
\begin{align}\label{eq:EtaKLhighN}
       Y_{T_2}\tilde{Z}_m^n \Big\| \langle y\rangle^{-n} \|\partial_{y}^j (-\Delta_{\mathbb{S}^3}+1)^5\chi_{Z}\eta\|_{L^2(\mathbb{S}^3)} \Big\|_{\infty}\lesssim &\kappa \Big(1+\mathcal{M}\Big)+\mathcal{M}^2.
\end{align}
\end{proposition}
The proposition will be proved in Sections \ref{sec:lowDeta}, \ref{sec:MaxTiEta} and \ref{sec:jGreat} below. These and the definition of $\mathcal{M}$ obviously imply the desired (\ref{eq:etaMsup}).

Next we explain the reasons for choosing the norms in (\ref{def:majorants}).

Starting from the easiest, for the reason of choosing $(-\Delta_{\mathbb{S}^3}+1)^5$ in the norms, we will need it, together with the embedding results Lemma \ref{LM:embedding}
below, to control some nonlinearity, see \eqref{eq:prelD11}.
\begin{lemma}\label{LM:embedding}
There exists a constant $C_{em}$ such that for any function $h:\ \mathbb{S}^3\rightarrow \mathbb{R}$
\begin{align}
\|h\|_{\infty}\leq C_{em} \|(-\Delta_{S^3}+1)^{3}h\|_{L^2(\mathbb{S}^3)}.
\end{align} 
\end{lemma}
The proof of this lemma is standard, hence we choose to skip the details here.

Next we present the reasons for choosing the norms in \eqref{def:majorants}.

Recall that our goal is to prove \eqref{eq:pointwisedm} in the main Theorem \ref{THM:sec}. We need some relatively sharp estimate for $\chi_{Z}\eta$. Our key tool is propagator estimates, see Lemma \ref{LM:frequencyWise} below. This forces us to prove that $\|\langle y\rangle^{-m-1} \chi_{Z}\eta (\cdot,\tau)\|_{\infty}$ decays slightly faster than $e^{-\frac{(m-2)(m+1)}{2m}\tau} $, as shown in \eqref{def:majorants}.

An adverse fact is that, some difficult terms are in the governing equation of $\chi_{Z}\eta$. Even though we want to ``close the estimates" as fast as possible, we have to take many steps if $m$ is large.
Among the many terms in the governing equations for $\chi_{Z}\eta$, here we only discuss two terms
\begin{align}
\begin{split}\label{eq:samD1D2}
    D_1:=&e^{-\frac{1}{2}\tau}\sqrt{q }\partial_{y}^2\eta (A_2\cdot \omega),\\
    D_2:=&e^{-\tau}y^2 |A_2\cdot \omega|^2\partial_{y}^2\eta.
\end{split}
\end{align}
They are parts of 
$
J_4:= \sqrt{q}|\partial_yQ_{m-1}\cdot \omega|^2\partial_{y}^2 \eta$ and $
J_5:=e^{-\frac{1}{2}\tau} q\partial_{y}^2 \eta (\partial_{y}^2 Q_{m-1}\cdot \omega)
$ in \eqref{def:J}:
we take
$$e^{-\frac{\tau}{2}}y^2 \Big(a_{2,1},\ a_{2,2},\ a_{2,3},\ a_{2,4}\Big)=\frac{1}{2} e^{-\frac{\tau}{2}}y^2 A_2$$ 
from $Q_{m-1},$ where the vector $A_2\in \mathbb{R}^4$ is naturally defined.
It is easier to treat the other parts since, in the considered region, there exists some $\epsilon_0>0$ such that $e^{-\frac{\tau}{2}}|y|\leq e^{-\epsilon_0\tau}\ll 1$ when $\tau\geq T_1\gg 1.$

In some sense $D_1$ and $D_2$ force us to choose the norms listed in \eqref{def:majorants}. 
We need to prove that $\|\langle y\rangle^{-m-1}\chi_{Z}D_{k}\|_{\infty},$ $k=1,2,$  decay, at least, as fast as $\|\langle y\rangle^{-m-1} \chi_{Z}\eta (\cdot,\tau)\|_{\infty}.$ 
For $D_2$, observe that in the considered region the factor $e^{-\tau}y^2$ decays very slowly, but $\langle y\rangle^{-1} e^{-\tau} y^2$ decays faster than $e^{-\frac{1}{2}\tau}$; for $D_1$ the factor $e^{-\frac{1}{2}\tau}\sqrt{q }$ decays slightly slower than $e^{-\frac{\tau}{2}},$ and the decay rate of $\langle y\rangle^{-1}e^{\frac{1}{2}\tau}\sqrt{q}\leq e^{-\frac{1}{2}}$ only improves slightly. Another constraint is that we need propagator estimate in Lemma \ref{LM:frequencyWise}.
Based on these reasons, we need to prove that
$\|\langle y\rangle^{-m}\partial_{y}^2\chi_{Z}\eta(\cdot,\tau)\|_{\infty}$ decays slightly faster than $e^{-\frac{m-2}{2}\tau}$, as shown in (\ref{def:majorants}).

By the same reason, to prove the desired decay rate for $\langle y\rangle^{-m}|\partial_{y}^2\chi_{Z}\eta|$, we need to prove $\langle y\rangle^{-m+1}|\partial_{y}^4\chi_{Z}\eta|$ decays at a rate about $\tilde{Z}_{m}^{-m+1} Y_{T_2}^{-1}$.

Fortunately we do not need to estimate $\partial_{y}^k\chi_{Z}\eta$ for all $k\in \mathbb{N}$. The reason is that, the wanted decay rates for $\partial_{y}^k \chi_{Z}\eta$ become lower as $k$ increases, and when $j=2m+3$ and $j=2m+4$ we only need $|\partial_{y}^{2m+3} \chi_{Z}\eta|+|\partial_{y}^{2m+4} \chi_{Z}\eta|\leq 1$, and they are provided by \eqref{eq:conditionZm}.

Next we discuss a crucial technical tool in the next subsection.

%%%%%%%%%%%%%%%%%%%%%%%%%%%%%%%%%%%%%%%%%%%%%%%%%%%%%%%%%%%%%%%%%

\subsection{The propagator estimates}
A crucially important tool is propagator estimate, since the propagator generates decay estimates. We are interested propagator generated by a linear operator $\mathcal{L}_V$, defined as
\begin{align}
    \mathcal{L}_{V}:=e^{-\frac{1}{8}y^2}\Big(-\partial_{y}^2+\frac{1}{2}y\partial_{y}-1+V\Big)e^{\frac{1}{8}y^2}=&-\partial_{y}^2 +\frac{1}{16}y^2-\frac{1}{4}-1+V\nonumber\\
    =& \mathcal{L}_0-1+V
\end{align} where $\mathcal{L}_0$ is naturally defined, and $V\geq 0$ is a bounded multiplier satisfying, for some small $\epsilon_0>0,$
\begin{align}
   \sup_{y}\Big\{ \Big|\partial_{y}V(y,\tau)\Big|+ \langle y\rangle^{-5}\Big|V(y,\tau)-V(0,\tau)\Big|\Big\}\leq \epsilon_0 (1+\tau)^{-\frac{2}{5}}. \label{eq:smallV}
\end{align}

Before stating the results, we define $P_n$ to be the orthogonal projection onto the subspace orthogonal to the eigenvectors $e^{-\frac{1}{8}y^2}H_k$, $k=0,1,\cdots,n$; and define $U_n(\tau,\sigma)$ to be the propagator generated by $-P_{n}\mathcal{L}_{V}P_{n}$, from times $\sigma$ to $\tau$, and $U(\tau,\sigma)$ be the one generated by $-\mathcal{L}_{V}$
\begin{lemma}\label{LM:frequencyWise}
For any fixed positive constants $\delta_0$, $k$ and nonnegative integer $n$, there exists a constant $C_{n,k,\delta_0}$ such that for any function $g,$
\begin{align}
    \Big\| \langle y\rangle^{-n-1-k} e^{\frac{1}{8}y^2}U_n(\tau,\sigma)P_n g\Big\|_{\infty}\leq &C_{n,k,\delta_0}e^{-\frac{n-1+\delta_0}{2}(\tau-\sigma)} \Big\| \langle y\rangle^{-n-1-k} e^{\frac{1}{8}y^2} g\Big\|_{\infty},\label{eq:propagatorMPro}
 \end{align}and for any $k\geq 0$, there exists some constant $C_k$ such that,
    \begin{align}
        \Big\| \langle y\rangle^{-k} e^{\frac{1}{8}y^2}U(\tau,\sigma)g\Big\|_{\infty}\leq & C_{k} e^{\tau} \Big\| \langle y\rangle^{-k} e^{\frac{1}{8}y^2} g\Big\|_{\infty}.\label{eq:propagatorNoPro}
\end{align}
\end{lemma}
The lemma will proved in Appendix \ref{sec:propagator}.

Now we present some ideas. The spectrum of $\mathcal{L}_0$ is $\Big\{\frac{k}{2}\ \Big|\ k=0,1,2,3,\cdots,\Big\}$ with corresponding eigenvectors $e^{-\frac{1}{8}y^2}H_k$. When $V$ is nonnegative and slowly varying, one expects the propagator generated by $-\mathcal{L}_V$, in certain subspace, decays rapidly. This is the intuition behind the results.

The present problem is very similar to what was considered in \cite{BrKu, DGSW} for the blowup problem of one-dimensional nonlinear heat equation; \cite{GS2008, GaKnSi, GaKn2014} for MCF; and \cite{MultiDHeat} for multidimensional nonlinear heat equations, where the propagator is generated by $-\tilde{L}_W$ of the form
\begin{align}
\tilde{L}_{W}:=-\partial_{y}^2 +\frac{1}{16}y^2-\frac{1}{4}-1+W.
\end{align} Here $W$ is bounded and positive, and satisfies a slightly different estimate, 
\begin{align}
  \sup_{y}\Big\{ \Big|\partial_{y}W(y,\tau)\Big|+ \langle y\rangle^{-5}\Big|W(y,\tau)\Big|\Big\}\leq \epsilon_0 R(\tau) .\label{eq:smallW}
\end{align} for some $\epsilon_0\ll 1$ and $R(\tau)\rightarrow 0$ as $\tau\rightarrow \infty.$
The differences between \eqref{eq:smallW} and \eqref{eq:smallV} do not make our proof any harder, especially after observing the identity in \eqref{eq:linearEvo} below. 

%%%%%%%%%%%%%%%%%%%%%%%%%%%%%%%%%%%%%%%%%%%%%%%%%

%%%%%%%%%%%%%%%%%%%%%%%%%%%%%%%%%%%%%%%%%%%%%%%%%%

%%%%%%%%%%%%%%%%%%%%%%%%%%%%

%%%%%%%%%%%%%%%%%%%%%%%%%%%%%%%%%%%%%%%%%%%%%%%%%%%%%%%%%%%%%%%%%

%%%%%%%%%%%%%%%%%%%%%%%%%%%%%%%%%

\section{Proof of the estimates for \texorpdfstring{$\chi_{Z}\eta_{k,l}$}{} in (\ref{eq:EtaKLlowN})}\label{sec:lowDeta}
From the governing equation for $ \chi_{Z}\eta_{k,l}$ in \eqref{eq:ZEtaKL} we derive
\begin{align}
\begin{split}\label{eq:nZEtaKL2}
    \partial_{\tau}\partial_{y}^{j}\chi_{Z}\eta_{k,l}&= -\Big(L_{k}+\frac{j}{2}\Big)\partial_{y}^{j}\chi_{Z}\eta_{k,l}+ \partial_y^j \chi_{Z}F_{k,l}+\widetilde{SN}_{k,l,j}+\partial_{y}^j \mu_{Z}(\eta_{k,l}),
\end{split}
\end{align} where $\frac{j}{2}$ in the linear operator is from a commutation relation: for any $h\in \mathbb{N}$ and function $g$,
\begin{align}
    \partial^h_{y}(\frac{1}{2}y\partial_{y}g)=(\frac{1}{2}y\partial_{y}+\frac{h}{2}) \partial_{y}^h g,\label{eq:commRel}
\end{align} and the functions $F_{k,l}$ and $\widetilde{SN}_{k,l}$ are defined as,
\begin{align*}
    F_{k,l}:=&\frac{1}{\langle f_{k,l},\ f_{k,l}\rangle_{\mathbb{S}^3}}\Big\langle F,\ f_{k,l}\Big\rangle_{\mathbb{S}^3}\\
    \widetilde{SN}_{k,l,j}:=&\frac{1}{\langle f_{k,l},\ f_{k,l}\rangle_{\mathbb{S}^3}}\partial_{y}^j \chi_{Z}\Big\langle SN,\ f_{k,l}\Big\rangle_{\mathbb{S}^3}+\frac{1}{q_M}\partial_{y}^{j}(\chi_{Z}\eta_{k,l})-\partial_{y}^j\Big(\frac{1}{q_M}\chi_{Z}\eta_{k,l}\Big).
\end{align*}

We need to transform the equation since $\partial_{y}^j\mu_{Z}(\eta_{k,l})$ contains difficult terms. It has two parts:
\begin{align}
\partial_y^j\mu_{Z}(\eta_{k,l})=&\partial_{y}^j \Big[(\partial_{\tau}\chi_{Z})\eta_{k,l}+\frac{1}{2}(y\partial_{y}\chi_{Z})\eta_{k,l}-(\partial_{y}^2\chi_{Z})\eta_{k,l}-2(\partial_{y}\chi_{Z})(  \partial_{y}\eta_{k,l})\Big]\nonumber\\
=&\partial_y^j\Big[\big|\frac{1}{2}-Z_{m}^{-1}\frac{d}{d\tau}Z_{m}\big|\ ( y\partial_{y}\chi_{Z})\eta_{k,l}-(\partial_{y}^2\chi_{Z})\eta_{k,l}-2(\partial_{y}\chi_{Z}) ( \partial_{y}\eta_{k,l}) \Big]\nonumber\\
=&-\Big|\frac{1}{2}-Z_{m}^{-1}\frac{d}{d\tau}Z_{m}\Big|\ \Big| y\partial_{y}\chi_{Z}\Big|\partial_{y}^j \eta_{k,l}+\Gamma_{j,1}(\eta_{k,l})\label{eq:diffCut}
\end{align} where the term $\Gamma_{j,1}(\eta_{k,l})$ is considered small and is defined as
\begin{align*}
&\Gamma_{j,1}(\eta_{k,l})\\
:=&\Big|\frac{1}{2}-Z_{m}^{-1}\frac{d}{d\tau}Z_{m}\Big| \Bigg\{\Big[\partial_{y}^j\big( ( y\partial_{y}\chi_{Z})\eta_{k,l}\big)- ( y\partial_{y}\chi_{Z}) \partial_{y}^j\eta_{k,l}\Big]-\partial_y^j\Big[(\partial_{y}^2\chi_{Z})\eta_{k,l}+2(\partial_{y}\chi_{Z}) ( \partial_{y}\eta_{k,l}) \Big]\Bigg\},
\end{align*} and in the second step we used that $\frac{1}{2}-Z_{m}^{-1}\frac{d}{d\tau}Z_{m}$ is positive by the definition of $Z_{m}$, and $y\partial_{y}\chi_{Z}\leq 0$ by requiring that $\chi(z)=\chi(|z|)$ is a decreasing function (see \eqref{eq:defChi3}).

$\Gamma_{j,1}(\eta_{k,l})$ is indeed small:
the definition $\chi_{Z}(|y|)=\chi(\frac{|y|}{Z_{m}})$ implies that for any $h\in \mathbb{N},$ $$\partial_{y}^{h}\chi_{Z}=Z_{m}^{-h} \big(\frac{d^h}{d x^n }\chi(x)\big)\Big|_{x=\frac{y}{Z_m}}.$$
This, together with
$q^{-1}\sum_{k=0}^2|\partial_{y}^k\eta|\ll 1$ implied by \eqref{eq:conditionZm} and that $\chi'(\frac{y}{Z_{m}})$ and $\chi^{''}(\frac{y}{Z_{m}})$ are supported by the set $|y|\geq Z_{m}$, implies that, for any $i\geq 0$ there exists a constant $C_{\chi,i}$, such that 
\begin{align}
    \langle y\rangle^{-i}|\Gamma_{j,1}(\eta_{k,l})|\leq C_{\chi, i} Z_{m}^{-i-1} q \ll Z_{m}^{-i-\frac{1}{2}} .\label{eq:smallGamma}
\end{align} 

It is more involved to control the first term in \eqref{eq:diffCut}. We observe that the $L^{\infty}$-norm of $y\partial_{y}\chi_{Z}$ is fixed, thus we can not treat it as a small term, since the definition $\chi_{Z}(y)=\chi(\frac{|y|}{Z_{m}})$ implies
\begin{align}
\sup_{y}\Big|y\partial_{y}\chi_{Z}(y)\Big|=\sup_{x}\Big|\chi^{'}(x)\Big|;
\end{align} and moreover, for example when $j=0,$ the map $\chi_{Z}w\rightarrow (y\cdot \partial_{y}\chi_{Z})w=\frac{y\cdot \partial_{y}\chi_{Z}}{\chi_{Z}}\chi_{Z}w$ is unbounded since $|\frac{y\cdot \partial_{y}\chi_{Z}}{\chi_{Z}}|\rightarrow \infty$ as $|y|\rightarrow (1+\epsilon)Z_{m}$.

To overcome this difficulty, we observe that it is non-positive. This is favorable, since at least intuitively,  a non-positive multiplier on the right hand side of \eqref{eq:nZEtaKL2} should help $\chi_{Z}\eta_{k,l}$ decay faster.
Our strategy is to absorb ``most" of it into the linear operator. For that purpose we define a new non-negative smooth cutoff function $\tilde\chi_{Z}(y)$ such that
\begin{align}
\tilde\chi_{Z}(y)=\left[
\begin{array}{ll}
1 ,\ \text{if}\ |y|\leq Z_{m}(1+\epsilon- Z_{m}^{-\frac{1}{4}}),\\
0,\ \text{if}\ |y|\geq Z_{m}(1+\epsilon-2 Z_{m}^{-\frac{1}{4}})
\end{array}
\right.
\end{align} and require it satisfies the estimate
\begin{align}
|\partial_{y}^{k}\tilde{\chi}_{Y}(y)|\lesssim Z_{m}^{-\frac{3}{4}|k|},\ |k|=1.
\end{align} Such a function is easy to construct, hence we skip the details. 
Decompose the function $(\frac{1}{2}-Z_{m}^{-1}\frac{d}{d\tau} Z_{m}) (y\partial_{y}  \chi_{Z}) \partial_{y}^n\eta_{k,l} $ into two parts and then compute directly to obtain
\begin{align}
\begin{split}\label{eq:unboundtwo}
&(\frac{1}{2}-Z_{m}^{-1}\frac{d}{d\tau} Z_{m})(y\partial_{y} \chi_{Z}) \partial_{y}^j \eta_{k,l}\\
=& (\frac{1}{2}-Z_{m}^{-1}\frac{d}{d\tau} Z_{m})\frac{(y\partial_{y} \chi_{Z}) \tilde\chi_{Z}}{\chi_{Z}}\chi_{Z} \partial_{y}^j\eta_{k,l}+(\frac{1}{2}-Z_{m}^{-1}\frac{d}{d\tau} Z_{m}) (y\partial_{y} \chi_{Z}) (1-\tilde\chi_{Z})\partial_{y}^j\eta_{k,l}\\
=&(\frac{1}{2}-Z_{m}^{-1}\frac{d}{d\tau} Z_{m})\frac{(y\partial_{y} \chi_{Z}) \tilde\chi_{Z}}{\chi_{Z}} \partial_{y}^j\chi_{Z}\eta_{k,l}+\Gamma_{j,2}(\eta_{k,l})
\end{split}
\end{align} where $\Gamma_{j,2}(\eta_{k,l})$ is defined as 
\begin{align*}
    \Gamma_{j,2}:=&(\frac{1}{2}-Z_{m}^{-1}\frac{d}{d\tau} Z_{m})\frac{(y\partial_{y} \chi_{Z}) \tilde\chi_{Z}}{\chi_{Z}} \Big[\chi_{Z}\partial_{y}^j\eta_{k,l}-\partial_{y}^j\chi_{Z}\eta_{k,l}\Big]\\
    &+(\frac{1}{2}-Z_{m}^{-1}\frac{d}{d\tau}Z_{m}) (y\partial_{y} \chi_{Z}) (1-\tilde\chi_{Z})\partial_{y}^j\eta_{k,l}.
\end{align*}

The following two observations will be used in later development: 
\begin{itemize}
\item[(A)]
The first part in \eqref{eq:unboundtwo} is a bounded multiplication operator since, for some $c(\epsilon)>0,$
\begin{align}
\Big|\frac{y\partial_{y} \chi_{Z}\  \tilde\chi_{Z}}{\chi_{Z}}\Big|\leq c(\epsilon) Z_m^{\frac{1}{4}}.\label{eq:NewPoten}
\end{align}If $Z_m$ is sufficiently large, then the properties of $\chi$ in \eqref{eq:chiSmooth} imply that
\begin{align}
\Big|\partial_{y}^{k}\Big[\frac{y\partial_{y} \chi_{Z}\  \tilde\chi_{Z}}{\chi_{Z}}\Big]\Big|\leq  Z_m^{-\frac{1}{4}},\ |k|=1,2.\label{eq:deriveSmooth}
\end{align}

\item[(B)]
To control $\Gamma_{j,2}$, we use \eqref{eq:chiSmooth} and techniques used in proving \eqref{eq:smallGamma} to find that, for any $h\geq 0,$
\begin{align}
\langle y\rangle^{-h}|\Gamma_{j,2}|\ll  Z_m^{-h-\frac{1}{2}}.\label{eq:TwoCut}
\end{align}
\end{itemize}

Thus \eqref{eq:nZEtaKL2} takes a new form
\begin{align}\label{eq:jetakl}
    \partial_{\tau}e^{-\frac{1}{8}y^2}\partial_{y}^j\chi_{Z}\eta_{k,l}=-(H_k+\frac{j}{2}) e^{-\frac{1}{8}y^2}\chi_{Z}\eta_{k,l} +e^{-\frac{1}{8}y^2}\Big[\partial_{y}^j\chi_{Z}F_{k,l}+ \widetilde{SN}_{k,l,j}\Big]
    +e^{-\frac{1}{8}y^2}\Gamma_{j}(\eta_{k,l}),
\end{align} where $H_k$ is a linear operator defined as
\begin{align*}
    H_k:=e^{-\frac{1}{8}y^2} L_k e^{\frac{1}{8}y^2}+\Big|\frac{1}{2}-Z_{m}^{-1}\frac{d}{d\tau}Z_{m}\Big|\Big|\frac{\tilde\chi_{Z}\ y\cdot \partial_{y} \chi_{Z}  }{\chi_{Z}}\Big|,
\end{align*} and $\Gamma_j(\eta_{k,l})$ is defined as
\begin{align*}
    \Gamma_{j}(\eta_{k,l}):=\Gamma_{j,1}(\eta_{k,l})+\Gamma_{j,2}(\eta_{k,l}).
\end{align*}

From \eqref{eq:ZEtaOrtho} we derive
\begin{align}
e^{-\frac{1}{8}y^2}\partial_{y}^j\chi_{Z}\eta_{k,l}\perp e^{-\frac{1}{8}y^2} H_i, \ i=0,\cdots, m-j, 
\end{align} equivalently, 
\begin{align}
P_{m-j} e^{-\frac{1}{8}y^2}\chi_{Z}\eta_{k,l}=e^{-\frac{1}{8}y^2}\chi_{Z}\eta_{k,l}.
\end{align}Recall that $1-P_n$ is the orthogonal projection onto the subspace spanned by $e^{-\frac{1}{8}y^2}H_i,\ i=0,\cdots,n.$

Rewrite \eqref{eq:jetakl} by applying the operator $P_{m-j}$, and then Duhamel's principle to obtain,
\begin{align}
\begin{split}\label{eq:durh}
e^{-\frac{1}{8}y^2}\partial_{y}^j&\chi_{Z}\eta_{k,l}(\cdot,\tau)=U_j(\tau,T_1) e^{-\frac{j}{2}(\tau-T_1)} e^{-\frac{1}{8}y^2}\chi_{Z}\eta_{k,l}(\cdot,T_1)\\
&+\int_{T_1}^{\tau} U_{j}(\tau,\sigma)e^{-\frac{j}{2}(\tau-\sigma)} P_{m-j} e^{-\frac{1}{8}y^2}
\Big[\partial_{y}^j\chi_{Z}F_{k,l}+ \widetilde{SN}_{k,l,j}\Big]+\Gamma_{j}(\eta_{k,l})\Big](\sigma)\ d\sigma,
\end{split}
\end{align} where $U_j(\tau,\sigma)$ is the propagator generated by $-P_{m-j} H_k P_{m-j}$. 

It is crucial that the propagator generates decay estimate. Lemma \ref{LM:frequencyWise} implies that, if $(n,j)\in \Upsilon_m$, then for any $\delta_0>0$ there exists a constant $C_{\delta_0}>0$ such that for any function $f$ and $\tau\geq \sigma,$
\begin{align}
e^{-\frac{j}{2}(\tau-\sigma)}\Big\|\langle y\rangle^{-n} e^{\frac{1}{8}y^2} U_j(\tau,\sigma)P_{m-j} f\Big\|_{\infty}\leq C_{\delta_0} e^{-(\frac{m-1}{2}-\delta_0)(\tau-\sigma)} \Big\|\langle y\rangle^{-n} e^{\frac{1}{8}y^2}f\Big\|_{\infty}.
\end{align}

For the terms on the right hand side we provide the following estimates: recall the definitions of $\tilde{Z}_m$ and $\kappa$ from (\ref{def:TiZm}) and (\ref{def:kappa}),
\begin{proposition}\label{prop:estSourceNonL} For any pair $(n,j)\in \Upsilon_m$, we have that
\begin{align}
\Big\|\langle y\rangle^{-n}P_{m-j} e^{-\frac{1}{8}y^2} \partial_{y}^j  \chi_{Z} F_{k,l}(\cdot,\tau) \Big\|_{\infty} \lesssim & e^{-40m\sqrt{\tau}} \tilde{Z}_{m}^{-n}(\tau) ,\label{eq:fkl}\\
\Big\|\langle y\rangle^{-n}\widetilde{SN}_{k,l,j}(\cdot,\tau)\Big\|_{\infty}\lesssim &  \tilde{Z}_{m}^{-n} Y_{T_2}^{-1}\  \Big(\kappa (1+\mathcal{M})+\mathcal{M}^2 \Big),\label{eq:snkl}\\
\Big\|\langle y\rangle^{-n}\Gamma_j(\eta_{k,l})(\cdot,\tau)\Big\|_{\infty}\lesssim & \kappa Z_{m}^{-n-\frac{1}{2}} .\label{eq:GamKL}
\end{align}
\end{proposition}
These first two estimates will be proved in Sections \ref{sec:source} and \ref{sec:smallLinear} respectively. The third one is implied by the definition of $\Gamma_{j}(\eta_{k,l})$ in \eqref{eq:jetakl}, and the estimates \eqref{eq:smallGamma} and \eqref{eq:TwoCut}.

Thus, \eqref{eq:durh} becomes
\begin{align*}
\Big\|\langle y\rangle^{-n}\partial_{y}^{l}\chi_{Z}&\eta_{k,l}(\cdot,\tau)
\Big\|_{\infty}\leq  C_{\delta_0}\Big\{ e^{-\frac{m-1-\delta_0}{2}(\tau-T_1)} \delta\\
&+\delta \int_{T_1}^{\tau}e^{-\frac{m-1-\delta_0}{2}(\tau-\sigma)} \Big[\tilde{Z}_{m}^{-n}(\sigma) Y_{T_2}^{-1}(\sigma) \Big]\ d\sigma\  \Big(\kappa+\kappa\mathcal{M}(\tau)+\mathcal{M}^2(\tau)\Big)\Big\}
\end{align*} Here we choose $\delta_0>0$ such that $\frac{m-1-\delta_0}{2}> \frac{(m+1)(m-2)}{2m}$, which is equivalent to
$$0<\delta_0<\frac{2}{m},$$
so that the function $e^{-\frac{m-1-\delta_0}{2}\tau}$ decays faster than 
$\tilde{Z}_{m}^{-m-1} Y_{T_2}^{-1}$.

By this and that $n\leq m+1$, we obtain the desired result, for some $C>0,$
\begin{align}
\Big\|\langle y\rangle^{-n}\partial_{y}^{l}\chi_{Z}\eta_{k,l}(\cdot,\tau)
\Big\|_{\infty}\leq C\delta  \tilde{Z}_{m}^{-n}(\tau) Y_{T_2}^{-1}  \Big(\kappa +\kappa \mathcal{M}(\tau)+\mathcal{M}^2(\tau)\Big).
\end{align}

%%%%%%%%%%%%%%%%%%%%%%%%%%%%%%%%%%%%%%%%%%%%%%%%%%%%%%%%%%%%%%%%%%%%%%%%%%%%%%%%%%%%

\section{Proof of the part for \texorpdfstring{$\chi_{Z}\tilde\eta$}{} in (\ref{eq:EtaKLlowN})}\label{sec:MaxTiEta}
The main tool is the maximum principle.
To prepare for its application we need a governing equation for the function
\begin{align}
g_j(y,\tau):=(M+y^2)^{-n}\Big\langle\xi_j(y,\cdot,\tau),\ \xi_j(y,\cdot,\tau)\Big\rangle_{\mathbb{S}^3},
\end{align} 
where $K$ is a positive constant, and $\xi_{j}$ is a function, defined as
\begin{align*}
K:=&20m^2,\\
\xi_{j}(y,\omega,\tau):= &(-\Delta_{\mathbb{S}^3}+1)^5 \partial_{y}^{j}\chi_{Z}\tilde\eta(y,\omega,\tau).
\end{align*}
Derive from \eqref{eq:EffTilXi}
\begin{align}
\begin{split}\label{eq:govGl}
    \partial_{\tau} g_j
    =&2(K+y^2)^{-n}\Big\langle \xi_j,\ -(L+\frac{j}{2}) \xi_j\Big\rangle_{\mathbb{S}^3}+\sum_{k=1}^3 D_{1,k}\\
   =&-\Big(L_{n}+\Lambda_{n,j}\Big) g_j-\text{Positive}(\xi_{j})+\sum_{k=1}^3 D_{1,k},
   \end{split}
\end{align} where we use the commutation relation (\ref{eq:commRel}), $L_{n}$ is a differential operator, and $\Lambda_{n,j}$ is a multiplier,
\begin{align*}
L_{n}:=&-\partial_{y}^2 +\frac{1}{2}y\partial_{y}-\frac{2ny}{K+y^2}\partial_y,\\
\Lambda_{n,j}:=&-\frac{2n}{K+y^2}+\frac{2n(n-1)y^2}{(K+y^2)^2}+\frac{2ny^2}{K+y^2}+j-1,
\end{align*}
and
$\text{Positive}(\xi_{j})$ is a positive function defined as
\begin{align*}
\text{Positive}(\xi_j):=2(K+y^2)^{-n}\Big[\Big\langle \partial_{y}\xi_j, \ \partial_{y}\xi_j\Big\rangle_{\mathbb{S}^3}-\frac{1}{q_M}\Big\langle  \xi_j,\  \Delta_{\mathbb{S}^3}\xi_j\Big\rangle_{\mathbb{S}^3}\Big]
\end{align*}
and the terms $D_{1,k},\ k=1,2,3,$ are defined as,
\begin{align}
\begin{split}
    D_{1,1}:=&2(K+y^2)^{-n}\Big\langle\xi_j,\ (-\Delta_{\mathbb{S}^3}+1)^5 \partial_{y}^j \chi_{Z} F\Big\rangle_{\mathbb{S}^3},\\
    D_{1,2}:=&2(K+y^2)^{-n}\Big\langle\xi_j,\ (-\Delta_{\mathbb{S}^3}+1)^5 \partial_{y}^j \chi_{Z}SN
    \Big\rangle_{\mathbb{S}^3},\\
    D_{1,3}:=&2(K+y^2)^{-n}\Big\langle\xi_j,\ (-\Delta_{\mathbb{S}^3}+1)^5 \partial_{y}^j \mu_{Z}(\tilde\eta)\Big\rangle_{\mathbb{S}^3}\\
    =&2(K+y^2)^{-n}\Big\langle\xi_j,\  \partial_{y}^j \mu_{Z}\Big((-\Delta_{\mathbb{S}^3}+1)^5 P_{\mathbb{S}^3,N_m}\tilde\eta\Big)\Big\rangle_{\mathbb{S}^3}
\end{split}
\end{align}

For the terms on the right hand, recall the definitions of $\tilde{Z}_m$ and $\kappa$ from (\ref{def:TiZm}) and (\ref{def:kappa}),
\begin{proposition}\label{prop:maximum1}
\begin{align}
D_{1,1}\leq &\kappa  \tilde{Z}_{m}^{-n} e^{-20m\sqrt{\tau }}\ g_{l}^{\frac{1}{2}}   ,\label{eq:estD11}\\
D_{1,2}\leq &   \tilde{Z}_{m}^{-n} Y_{T_2}^{-1}\ g_{l}^{\frac{1}{2}}\Big(\kappa (1+\mathcal{M})+\mathcal{M}^2 \Big) +\kappa \text{Positive}(\xi_l) ,\label{eq:estD12}\\
D_{1,3}\leq &\delta^2  Z_{m}^{-2n-\frac{1}{2}} .\label{eq:estD13}
\end{align}
\end{proposition}
The first two estimates will be proved in Sections \ref{sec:source},\ \ref{sec:smallLinear} respectively. The third one is defined in terms of derivatives of the cutoff function $\chi_{Z}$. Different from what was discussed around \eqref{eq:diffCut}, the difficult term $(\frac{1}{2}-Z_{m}^{-1}\frac{d}{d\tau}Z_{m}) y \partial_{y}\chi_{Z}$ becomes easy to be controlled, since it is favorably nonpositive,
\begin{align}
\tilde{D}_{1,3}:=(\frac{1}{2}-Z_{m}^{-1}\frac{d}{d\tau}Z_m)\Big\langle (-\Delta_{\mathbb{S}^3}+1)^5\chi_{Z} \partial_{y}^{j}\tilde\eta,\ (-\Delta_{\mathbb{S}^3}+1)^5 (y\partial_{y}\chi_{Z}) \partial_{y}^{j}\tilde\eta\Big\rangle_{\mathcal{G}}\leq 0.
\end{align}
For the other parts we apply the same techniques used in \eqref{eq:smallGamma} to obtain the desired results.

Returning to \eqref{eq:govGl}, we collect the estimates above to find
\begin{align}\label{eq:interme}
    \partial_{\tau}g_l\leq -\Big(L_{n}+\Lambda_{n,j}-\frac{1}{2}\delta_0\Big) g_j-(1-\kappa)\text{Positive}(\xi_{j})+C_{\delta_0}\kappa^2 Z_{m}^{-2n} Y_{T_2}^{-2} \Big(1+\mathcal{M}\Big)^4. 
\end{align}

Now we need some contribution from $\text{Positive}(\xi_{j})$. Recall that $f_{n,m}$, $n=0,1,2,\cdots,$ are eigenvectors of $-\Delta_{\mathbb{S}^3}$ with eigenvalues $n(n+2).$ The condition $\xi_l(y,\cdot,\tau)\perp_{\mathbb{S}^3} f_{k,l},\ k\leq N, $ implies $$q_{M}^{-1}\Big\langle \xi_l, -\Delta_{\mathbb{S}^3}\xi_l\Big\rangle_{\mathbb{S}^3}\geq q_{M}^{-1} (N+1)(N+3) \Big\langle \xi_l, \xi_l\Big\rangle_{\mathbb{S}^3}.$$ This implies an important estimate: for any $\delta_0>0$, provided that $T_1$ and $N_m$ are large enough, 
\begin{align}
    \Lambda_{n,j}+q^{-1}_{M} (N+1)(N+3)\geq m-1-\frac{1}{2}\delta_0.
\end{align} To see this, we consider two regions: when $|y|\leq \tau^2$, then $q_{M}^{-1}\approx 6$, and $\frac{1}{6}(N+1)(N+3)\geq m-1;$ when $|y|\geq \tau^2$, $\Lambda_{n,j}\geq m-1-\frac{1}{2}\delta_0$. Recall that $\tau\geq T_1\gg 1.$

\eqref{eq:interme} becomes, for any $\delta_0>0$, suppose that $T_1$ is large enough,
\begin{align}
\partial_{\tau}g_j\leq -(m-1-\delta_0) g_j+C_{\delta_0} Z_{m}^{-2n} Y_{T_2}^{-2}  \Big(\kappa (1+\mathcal{M})+\mathcal{M}^2 \Big)^2.
\end{align}

The cutoff function $\chi_{Z}$ in the definition of $g_j$ makes $g_j(y,\tau)\equiv 0$ if $|y|\geq (1+\epsilon )Z_m.$

Apply the maximum principle to obtain
\begin{align}
\sup_{y}g_j(y,\tau)\leq & C_{\delta_0}\Big[e^{-(m-1-\delta_0)(\tau-T_1)}\sup_{y}g_{l}(y,T_1)\nonumber\\
& + \Big(\kappa (1+\mathcal{M})+\mathcal{M}^2 \Big)^2 \int_{T_1}^{\tau} e^{-(m-1-\delta_0)(\tau-\sigma)} \tilde{Z}_{m}^{-2n}Y_{T_2}^{-2}(\sigma) \ d\sigma\Big]\nonumber\\
\lesssim  & C_{\delta_0} \Big[ \delta^2  e^{-(m-1-\delta_0)(\tau-T_1)}+ \Big(\kappa (1+\mathcal{M})+\mathcal{M}^2 \Big)^2  \tilde{Z}_{m}^{-2n} Y_{T_2}^{-2}(\tau)
 \Big].
\end{align}
What is left is to take a square root to obtain the desired result.

\section{Proof of (\ref{eq:EtaKLhighN})}\label{sec:jGreat}
We will follow the steps in Section \ref{sec:MaxTiEta}.
To make the maximum principle applicable we derive a governing equation for $g_j$, defined as
\begin{align}
g_j(y,\tau):=(K+y^2)^{-n}\Big\langle\zeta_j(y,\cdot,\tau),\ \zeta_j(y,\cdot,\tau)\Big\rangle_{\mathbb{S}^3},
\end{align} 
where $K=20m^2$ is a positive constant, and $\zeta_j$ is a function, defined as
\begin{align*}
\zeta_j(y,\omega,\tau):=& (-\Delta_{\mathbb{S}^3}+1)^5 \partial_{y}^{l}\chi_{Z}\eta(y,\omega,\tau).
\end{align*}
Derive from \eqref{eq:EffTilXi},
\begin{align}
\begin{split}\label{eq:govGl2}
    \partial_{\tau} g_j
   =&-\Big(L_{n}+\Lambda_{n,j}\Big)g_j-\text{Positive}(\zeta_j)+\sum_{k=1}^3 D_{2,k},
   \end{split}
\end{align}where the differential operator $L_{n}$, the multiplier $\Lambda_{n,j}$ and the function $\text{Positive}$ are defined in \eqref{eq:govGl},
the terms $D_{2,k}$ are defined as
\begin{align}
\begin{split}
D_{2,1}:=&2(K+y^2)^{-n}\Big\langle\zeta_j,\ (-\Delta_{\mathbb{S}^3}+1)^4 \partial_{y}^j \chi_{Z}F\Big\rangle_{\mathbb{S}^3},\\
D_{2,2}:=&2(K+y^2)^{-n}\Big\langle\zeta_j,\ (-\Delta_{\mathbb{S}^3}+1)^4 \partial_{y}^j \chi_{Z}SN
   \Big\rangle_{\mathbb{S}^3},\\
D_{2,3}:=&2(K+y^2)^{-n}\Big\langle\zeta_j,\ (-\Delta_{\mathbb{S}^3}+1)^4 \partial_{y}^j \mu_{Z}(\eta)\Big\rangle_{\mathbb{S}^3}.
\end{split}
\end{align}

Now we estimate these terms, recall the definitions of $\tilde{Z}_m$ and $\kappa$ from (\ref{def:TiZm}) and (\ref{def:kappa}),
\begin{proposition}\label{prop:D3k}
The terms $D_{2,k}$ satisfy the following estimates
\begin{align}
D_{2,1}\leq &\kappa  \tilde{Z}_{m}^{-n} e^{-20m\sqrt{\tau}}\ g_{j}^{\frac{1}{2}}   ,\label{eq:estD21}\\
D_{2,2}\leq &   \tilde{Z}_{m}^{-n} Y_{T_2}^{-1}\ g_{j}^{\frac{1}{2}} \Big(\kappa (1+\mathcal{M})+\mathcal{M}^2 \Big)^2+\kappa \text{Positive}(\zeta_j) ,\label{eq:estD22}\\
D_{2,3}\leq &\delta^2  Z_{m}^{-2n-\frac{1}{2}} .\label{eq:estD23}
\end{align}
\end{proposition}
The first two estimates will be proved in Sections \ref{sec:source},\ \ref{sec:smallLinear} respectively. And \eqref{eq:estD23} can be proved by the same methods used in proving \eqref{eq:estD13}, hence we skip the details.

Returning to \eqref{eq:govGl2}, we compute directly to find that, for any pairs $(n,j)\in \Upsilon_{m}$ with $j\geq m+2,$
\begin{align}
  \Lambda_{n,j}= -\frac{2n}{K+y^2}+\frac{2n(n-1)y^2}{(K+y^2)^2}+\frac{ny^2}{K+y^2}+j-1\geq m+1-\frac{1}{20}.
\end{align}

Collect the estimates to obtain, for some $C>0,$
\begin{align}
    \partial_{\tau}g_j\leq -(m+1-\frac{1}{20})g_j+C  \tilde{Z}_{m}^{-2n} Y_{T_2}^{-2}  \Big(\kappa (1+\mathcal{M})+\mathcal{M}^2 \Big)^2.
\end{align}

Before applying the maximum principle we observe that the cutoff function $\chi_{Z}$ in the definition of $g_j$ makes
\begin{align}
g_j(y,\tau)=0\ \text{when}\ |y|\geq (1+\epsilon)Z_m(\tau).
\end{align}
Apply the maximum principle to obtain
\begin{align}
\begin{split}
\max_{y}g_j(y,\tau)\lesssim & e^{-(m+1-\frac{1}{20})(\tau-T_1)} \max_{y}g_{l}(y,T_1) \\
&+ \int_{T_1}^{\tau} e^{-(m+1-\frac{1}{20})(\tau-\sigma)} \ \tilde{Z}_{m}^{-2n}(\sigma) Y_{T_2}^{-2}(\sigma)\ d\sigma  \Big(\kappa (1+\mathcal{M})+\mathcal{M}^2 \Big)^2\\
\lesssim & \kappa^2 \Big(1+\mathcal{M}\Big)^4 \tilde{Z}_{m}^{-2n} Y_{T_2}^{-2}.
\end{split}
\end{align}
What is left is to take a square root to obtain the desired result.

%%%%%%%%%%%%%%%%%%%%%%%%%%%%%%%%%%%%%%%%%%%%%%%%%%%%%%%%%%%%%%%%%%%%%%%

\section{Proof of (\ref{eq:estBetaEff}), (\ref{eq:fkl}), (\ref{eq:estD11}) and (\ref{eq:estD21})}\label{sec:source}
To prepare for the proof we study $\chi_{Z}F=\chi_{Z}(F_1+F_2)$. It is easy to control $F_1$ by its explicit expression.

We observe some good properties for $F_2=F_2(\sqrt{q})$.
The function $q$ takes the form
\begin{align}
    q(y,\omega,\tau)=6+\tilde{q}(y,\omega,\tau)+y^m\tilde{q}_{M}(\omega,\tau)\label{eq:qqqM}
\end{align} where 
$\tilde{q}$ is a polynomial of $y$ of degree $m-1$, and $\tilde{q}_{M}$ is independent of $y$, 
\begin{align*}
    \tilde{q}(y,\omega,\tau):=&\sum_{k=0}^N \sum_l\sum_{n=0}^{m-1}c_{n,k,l}(\tau) e^{-\frac{n}{2}} y^n f_{k,l}(\omega),\\
    \tilde{q}_{M}(\omega,\tau):=& \sum_{k=0}^N\sum_{l}\alpha_{m,k,l}(\tau) f_{k,l}(\omega).
\end{align*} \eqref{eq:DiffAlBe}, \eqref{eq:n01Rapid} and \eqref{eq:dm} imply that: for some constants $\tilde{c}_{n,k,l}$, for any $k$ and $l$, and any $n\geq 2,$
\begin{align}\label{eq:cnkl1}
    e^{\frac{1}{3}\tau} |c_{0,k,l}(\tau)|+e^{\frac{5}{6}\tau}|c_{1,k,l}(\tau)|+e^{(\frac{n}{2}+\frac{1}{3})\tau}\Big| c_{n,k,l}(\tau)-\tilde{c}_{n,k,l}e^{-\frac{n}{2}\tau}\Big|&\lesssim  1;
\end{align} $\alpha_{m,k,l}$ satisfy the following estimates: for any $j\geq 1$ and any $l$, and for any constant $\delta_0>0,$ 
\begin{align}\label{eq:cnkl2}
    \Big|\alpha_{m,0,1}(\tau)-d_{m}e^{-\frac{m-2}{2}\tau}\Big|+
    |\alpha_{m,j,l}(\tau)|\leq & C_{\delta_0} e^{-\frac{m-1-\delta_0}{2}\tau}.
\end{align}

These imply that $|\tilde{q}|$ is favorably small in the considered region, more precisely, when $$|y|\leq (1+\epsilon)Z_m(\tau)=(1+\epsilon) \Big(e^{\frac{m-2}{2m}(\tau-T_1)}e^{10\sqrt{\tau-T_1}}+5\tau^{\frac{1}{2}+\frac{1}{20}}\Big),$$ there exists a constant $\epsilon_0=\epsilon_0(m)>0$ such that
\begin{align}
   |\tilde{q}(y,\omega,\tau)|\leq e^{-\epsilon_0 \tau}, \text{when}\ \tau\geq T_1\gg 1.
\end{align}

To make more preparation we observe that $F(\sqrt{6+\tilde{q}})$ is nonlinear in terms of $\tilde{q}$ and $\partial_{y}Q_{m-1}$, except $J_9(Q_{m-1},\sqrt{6+\tilde{q}})$ and $J_{10}(Q_{m-1},\sqrt{6+\tilde{q}})$ defined in \eqref{def:J9J10};  and by definition 
\begin{align*}
\partial_{y}Q_{m-1}(y,\tau)= \sum_{n=2}^{m-1}\partial_{y}H_{n}(y) e^{-\frac{n-1}{2}\tau}\Big(a_{n,1},a_{n,2},a_{n,3},a_{n,4}
\Big)^{T}.
\end{align*} 

These observations and the explicit forms of $J_9$ and $J_{10}$ make it easy to control $F(\sqrt{6+\tilde{q}})$.
By Taylor expanding in $y$, we observe that, for some large integers $M_1\geq m+1$ and $M_2$
\begin{align}
F_2(\sqrt{6+\tilde{q}})=\sum_{n=0}^{M_1} \sum_{k=0}^{M_2}\sum_l g_{n,k,l}(\tau) e^{-\frac{n}{2}}y^n f_{k,l}(\omega) +\text{Remainder}\label{eq:remainder}
\end{align} where, for some constants $\tilde{g}_{n,k,l}$,
\begin{align}
    |g_{n,k,l}(\tau)-\tilde{g}_{n,k,l}|\lesssim & e^{-\frac{1}{3}\tau},\ \text{with}\ n\geq 2,\\
    e^{\frac{\tau}{2}} \Big(|g_{0,k,l}(\tau)|+|g_{1,k,l}(\tau)|\Big)\lesssim & 1,
\end{align}
the term $\text{Remainder}$ is of order $e^{-\frac{m}{2}\tau}$ in the sense that, for any $n_1\in \mathbb{N}\cup{\{0\}}$ and $j_1\in (\mathbb{N}\cup{\{0\}})^3$, there exists a constant $a_{n_1,j_1}$ such that, in the considered region,
\begin{align}
\Big|\partial_{y}^{n_1}\nabla_{E}^{j_1} \text{Remainder}(y,\omega,\tau)\Big| \leq a_{n_1, j_1} e^{-m\tau}.
\end{align}

Now we consider $F(\sqrt{q})-F(\sqrt{6+\tilde{q}})$,  Among the terms in $q-\tilde{q}$, the main obstacle is that 
$$y^m\alpha_{m,0,1}=y^{m}\Big(e^{-\frac{m-2}{2}\tau} d_m+\mathcal{O}_{\delta_0} (e^{-\frac{m-1-\delta_0}{2}\tau})\Big)$$ 
is not uniformly bounded in the considered region. This forces us to control them by different techniques.
The proof of (\ref{eq:estBetaEff}) is easy since the rapid decay of the weight $e^{-\frac{1}{4}y^2}$ in the $\mathcal{G}-$inner product will overwhelm the modest growth of $y^{m}$. But in the proof of (\ref{eq:fkl}), (\ref{eq:estD11}) and (\ref{eq:estD21}), where it is considered in ``modestly" weighted $L^{\infty}-$norms, we have to deal with this obstacle.

Now we are ready to prove (\ref{eq:estBetaEff}).
\subsection{Proof of (\ref{eq:estBetaEff})}
Recall that $NL_{n,k,l}$ is defined as
\begin{align}
\begin{split}
NL_{n,kl}
=& \Big(\Big \langle \chi_{Z}F_2(\sqrt{q}), H_{n}f_{k,l}\Big\rangle_{\mathcal{G}}+\Big\langle \chi_{Z}SN(\eta), H_{n}f_{k,l}\Big\rangle_{\mathcal{G}}\Big)+o(e^{-\frac{1}{5}Z_{m}^2})\\
=& \Big\{\Big \langle \chi_{Z}F_2(\sqrt{6+\tilde{q}}), H_{n}f_{k,l}\Big\rangle_{\mathcal{G}}+\Big \langle \chi_{Z}\Big(F_2(\sqrt{q})-F_2(\sqrt{6+\tilde{q}})\Big), H_{n}f_{k,l}\Big\rangle_{\mathcal{G}}\\
&+\Big\langle \chi_{Z}SN(\eta), H_{n}f_{k,l}\Big\rangle_{\mathcal{G}}\Big\}+o(e^{-\frac{1}{5}Z_{m}^2})
\end{split}
\end{align} It is easy to prove that, since each term in $SN(\eta)$ is either nonlinear in terms of $\eta$, or is linear but with small coefficients, the preliminary estimate provided by \eqref{eq:DiffAlBe} and \eqref{eq:GLEtaM2} implies,
\begin{align}
|\Big\langle SN(\eta), H_n f_{k,l}\Big\rangle_{\mathcal{G}}|\leq & C_{\delta_0} e^{-\frac{m-\delta_0}{2}\tau}.
\end{align} \eqref{eq:remainder} implies that, for any $n\geq 2,$
\begin{align}
|\Big\langle \chi_{Z}F_2(\sqrt{6+\tilde{q}}), H_{n}f_{k,l}\Big\rangle_{\mathcal{G}}|\leq & A_{n,k,l} e^{-\frac{n}{2}\tau}+\mathcal{O}(e^{-(\frac{n}{2}+\frac{1}{3})\tau}),\label{eq:ankl}
\end{align} and for $n=0,1$, 
\begin{align}\label{eq:mf01}
    e^{\frac{\tau}{2}}|\Big\langle \chi_{Z}F_2(\sqrt{6+\tilde{q}}), H_{0}f_{k,l}\Big\rangle_{\mathcal{G}}|+e^{\tau}|\Big\langle \chi_{Z}F_2(\sqrt{6+\tilde{q}}), H_{1}f_{k,l}\Big\rangle_{\mathcal{G}}|\lesssim 1.
\end{align} 

Now we consider $F(\sqrt{q})-F(\sqrt{6+\tilde{q}})$. Recall that $q-6-\tilde{q}=y^m\sum_{k,l}\alpha_{m,k,l} f_{k,l}$. Among $\alpha_{m,k,l}$ the decay rate of $\alpha_{m,0,1}$ is lowest, being $e^{-\frac{m-2}{2}\tau}$. Since every term in $F(\sqrt{q})-F(\sqrt{6+\tilde{q}})$ is nonlinear in terms of $q-6-\tilde{q}$, $q$ and $\partial_{y}Q_{m-1}$, compute directly to find
\begin{align}
|\Big\langle \chi_{Z}\Big(F_2(\sqrt{q})-F_2(\sqrt{6+\tilde{q}})\Big), H_{n}f_{k,l}\Big\rangle_{\mathcal{G}}|\lesssim e^{-\frac{m-1}{2}\tau}.\label{eq:alDif}
\end{align} This estimate is sharp, since, by $J_1(v)$ defined in (\ref{def:J}), $F_2(\sqrt{q})$ contains $ (\partial_{y}Q_{m-1}\cdot \omega)\ \partial_{y}\sqrt{q}$.

Collect the estimates above to obtain the desired (\ref{eq:estBetaEff}).

For the case $d_m=0$, the decay rate of $\alpha_{m,0,1}$ improves significantly. The following result will be used to prove a part of \eqref{eq:betamm1} in Section \ref{sec:ProofTHMsecC} below. 
\begin{lemma}\label{LM:improvement} 
If $d_m=0$, then for any $\delta_0>0$ there exists a constant $C_{\delta_0}$ such that  
\begin{align}
|\alpha_{m,0,1}(\tau)|\leq & C_{\delta_0} e^{-\frac{m-\delta_0}{2}\tau}.
\end{align}
\end{lemma}
\begin{proof}
By \eqref{eq:dm}, \eqref{eq:DiffAlBe}, \eqref{eq:ankl} and (\ref{eq:alDif}), the decay rate of $NL_{m,0,1}$ becomes $e^{-\frac{m-\delta_0}{2}\tau}$, thus
\begin{align}
|\alpha_{m,0,1}(\tau)|=& \Big|e^{-\frac{m-2}{2}\tau}\Big[\alpha_{m,0,1}(0)+\int_{0}^{\infty} NL_{m,0,1}(\sigma)\ d\sigma\Big]-e^{-\frac{m-2}{2}\tau}\int_{\tau}^{\infty}e^{\frac{m-2}{2}\sigma} NL_{m,0,1}(\sigma)\ d\sigma\Big|\nonumber\\
=& \Big|e^{-\frac{m-2}{2}\tau}\int_{\tau}^{\infty}e^{\frac{m-2}{2}\sigma} NL_{m,0,1}(\sigma)\ d\sigma\Big|\leq C_{\delta_0} e^{-\frac{m-\delta_0}{2}\tau},
\end{align} where, we use that $d_m=0$ implies that $\alpha_{m,0,1}(0)+\int_{0}^{\infty} NL_{m,0,1}(\sigma)\ d\sigma=0$.
\end{proof}

%%%%%%%%%%%%%%%%%%%%%%%%%%%%%%%%%%%%%%%%%%%%%%%%%%%%%
\subsection{Proof of (\ref{eq:fkl}), (\ref{eq:estD11}) and (\ref{eq:estD21})}
We start with proving (\ref{eq:fkl}), but only the case $j=0$, i.e. no $y-$derivative is taken on $\chi_{Z}F$. This is the most difficult case, since, as $j$ increases, the wanted decay rate becomes slower, and hence is easy to obtain; on the other hand it is easy to see that $y-$derivatives improve decay rate.

For the case $j=0$ of (\ref{eq:fkl}), the following estimate is better than the desired (\ref{eq:fkl}),
\begin{align}\label{eq:betterThanAsk}
   \Big\|\langle y\rangle^{-m-1} e^{\frac{1}{8}y^2} P_m e^{-\frac{1}{8}y^2} \chi_{Z}\langle F, \ f_{k,l}\rangle_{\mathbb{S}^3}\Big\|_{\infty}\leq e^{-\frac{m-2}{2}\tau}e^{-\frac{m-2}{2m}\tau}e^{-40(m+1)\sqrt{\tau}}.
\end{align} 

In what follows we prove \eqref{eq:betterThanAsk}.

Recall that 
$F_1(\sqrt{q})=-\sum_{n=0}^m \sum_{h,i} \Big(\frac{d}{d\tau}+\frac{n-2}{2}+\frac{h(h+2)}{6}\Big)\alpha_{n,h,i} H_n f_{h,i}.$
It easy to see that the contribution from $F_1$ satisfies the desired estimates \eqref{eq:betterThanAsk} since
$$ P_m e^{-\frac{1}{8}y^2} \Big\langle F_1, \ f_{k,l}\Big\rangle_{\mathbb{S}^3}\equiv 0.$$Similar $ P_m e^{-\frac{1}{8}y^2} \chi_{Z}\Big\langle F_2(\sqrt{6+\tilde{q}}), \ f_{k,l}\Big\rangle_{\mathbb{S}^3}$ decays sufficiently rapidly since $P_{m}$ removes the slow part. Hence it satisfies (\ref{eq:betterThanAsk}).

What is left is to consider $F(\sqrt{q})-F(\sqrt{6+\tilde{q}})$. The ``worst" term in $q-6-\tilde{q}=y^m\sum_{k,l}\alpha_{m,k,l} f_{k,l}$ is $\alpha_{m,0,1} y^m$ since
\begin{align}
    |\alpha_{m,0,1}(\tau) y^m|\gg 1,\ \text{when}\ (1+\epsilon)Z_{m}(\tau)\geq  |y|\gg e^{-\frac{m-2}{2m}\tau}.\label{eq:worstTerm}
\end{align} It is considerably easier to treat the other terms since, in the considered region, they are small.

However it is not difficult to overcome this obstacle for the following two reasons:
\begin{itemize}
\item[(1)]
The decay rate of $\alpha_{m,0,1},$ which is $e^{-\frac{m-2}{2}\tau},$ is close to the wanted one in \eqref{eq:betterThanAsk},

\item[(2)]
The difficulty presented in (\ref{eq:worstTerm}) is not so bad either. Even though $\alpha_{m,0,1}y^m$ can be very large, we observe that, in the considered region, provided that $T_1$ is sufficiently large,
\begin{align}
e^{-\delta_0 \tau} |\alpha_{m,0,1}(\tau)y^m|\ll 1\ \text{for any }\ \delta_0>0.
\end{align}

Moreover a $y-$derivative of $\alpha_{m,0,1} y^m$ produces a favorably small $m \alpha_{m,0,1} y^{m-1}$; and a covariant derivative on $\mathbb{S}^3$ works even better since $\nabla_{\omega}^{\perp}\alpha_{m,0,1} y^m\equiv 0$.
\end{itemize}

Now we study two examples to illustrate these ideas, and the other terms can be treated similarly.
We consider the terms $J_1(\sqrt{q})=\partial_{y}\sqrt{q}\ \Big(\partial_{y}Q_{m-1}\cdot\omega\Big)$ and $K_2(\sqrt{q})=q^{-2} | \nabla_{\omega}^{\perp} q |^2$ from $J_1$ and $K_2$ defined in \eqref{def:J} and \eqref{def:Kterm}. Their contributions to $F_2(\sqrt{q})-F_{2}(\sqrt{6+\tilde{q}})$ are,
\begin{align*}
    \tilde{J}_1(\tilde{q},\tilde{q}_{M}):=&J_1(\sqrt{q})-J_1(\sqrt{6+\tilde{q}})=
     \frac{\partial_{y}\tilde{q}_{M} \ (\partial_{y}Q_{m-1}\cdot\omega)}{2\sqrt{q}}
    -\frac{\tilde{q}_M \partial_{y}\tilde{q}\ (\partial_{y}Q_{m-1}\cdot\omega)}{2 q^{\frac{1}{2}} (6+\tilde{q})^{\frac{1}{2}}[q^{\frac{1}{2}}+(6+\tilde{q})^{\frac{1}{2}}]},\\
    \tilde{K}_2(\tilde{q},\tilde{q}_{M}):=&K_2(\sqrt{q})-K_2(\sqrt{6+\tilde{q}})=\frac{2 \nabla_{\omega}^{\perp} \tilde{q}\cdot \nabla_{\omega}^{\perp}\tilde{q}_{M} +|\nabla_{\omega}^{\perp}\tilde{q}_{M}|^2}{q}-\frac{\tilde{q}_{M}|\nabla^{\perp}_{\omega}\tilde{q}|^2}{q(6+\tilde{q})}.
\end{align*} 

Compute directly to find a decay rate better than that in \eqref{eq:betterThanAsk},
$$\langle y\rangle^{-m-1}\Big(|\tilde{J}_1|+|\tilde{K}_2|\Big)\lesssim e^{-\frac{m-1}{2}\tau}.$$

Thus we complete the proof of \eqref{eq:betterThanAsk}, which implies the desired (\ref{eq:fkl}).

Now we prove (\ref{eq:estD11}). Recall the constant $N$ and the operator $P_{\omega,N}$ from \eqref{eq:newdec} and \eqref{def:OmePro}. For any $n\leq m$, we claim that
\begin{align}
   \Big\|P_{\omega, N}\Big\langle \chi_Z F, e^{-\frac{1}{4}y^2}H_n\Big\rangle_{L^2(\mathbb{R})}\Big\|_{L^2(\mathbb{S}^3)}=\Big\|P_{\omega, N}\Big\langle \chi_Z F_2, e^{-\frac{1}{4}y^2}H_n\Big\rangle_{L^2(\mathbb{R})}\Big\|_{L^2(\mathbb{S}^3)}\lesssim  e^{-\frac{m-1}{2}\tau}.\label{eq:sizeNm}
\end{align} This is indeed true: in the first step we observe that $P_{\omega,N}$ removes the $F_1$-part; for the second step, since $ \chi_Z F_2$ is in the governing equation for $\chi_{Z}\eta$, $\big\|P_{\omega, N}\Big\langle \chi_Z F_2, e^{-\frac{1}{4}y^2}H_n\Big\rangle_{L^2(\mathbb{R})}\big\|_{L^2(\mathbb{S}^3)}$ must decay fast enough to make it possible that $\|\chi_{Z}\eta(\cdot,\tau)\|_{\mathcal{G}}$ decays at the high rate implied by \eqref{eq:DiffAlBe} and \eqref{eq:GLEtaM2}. Recall the definition of the operator $P_{\omega,N}$ from \eqref{def:OmePro}.

\eqref{eq:sizeNm} and the properties enjoyed by $F_2(\sqrt{q})$ make the proof of (\ref{eq:estD11}) very similar to that of (\ref{eq:fkl}), hence we skip the details here.

It is easy to prove (\ref{eq:estD21}) since, after many $y-$derivatives are taken on $\chi_{Z}F$, the decay estimates improve, and then, favorably, the wanted decay rates are significantly lower. We skip this part.

%%%%%%%%%%%%%%%%%%%%%%%%%%%%%%%%%%%%%%%%%%%%%%%%%%%%%%%%%%%%%%%%%%%%%%%%%%%%%%%%%%%%%%%%%%%%%%%%%%%%%%

\section{Proof of (\ref{eq:snkl}), (\ref{eq:estD12}), (\ref{eq:estD22})}\label{sec:smallLinear}
Among the terms in $SN$, we only consider $\tilde{K}_1(v)-\tilde{K}_1(\sqrt{q})$, with $\tilde{K}_1(v)$ defined in \eqref{def:TilKJ} as
\begin{align*}
    \tilde{K}_1(v)= \frac{v(\partial_{y}v)^2 \partial_{y}^2 v}{1+(\partial_{y}v)^2+v^{-2}|\nabla_{\omega}^{\perp}v|^2}=\frac{K_1(v)}
    {1+\tilde{N}_1(v)}
\end{align*} where $\tilde{N}_1(v):=(\partial_{y}v)^2+v^{-2}|\nabla_{\omega}^{\perp}v|^2,$ and $K_1(v):=v(\partial_{y}v)^2 \partial_{y}^2 v$.

The chosen term is among the most difficult to handle. For example, $K_2$, defined in \eqref{def:Kterm}, depends on $\partial_{\omega}^{\perp} v$ and hence is easier since, in the $L^2(\mathbb{S}^3)$-inner product, one can integrate by parts. Also it is easy to control the terms listed in \eqref{def:J} since every one has a factor $\partial_{y}Q_{m-1}$, which decays uniformly in the considered region.

Recall that, around \eqref{eq:samD1D2}, we discussed some difficulties in controlling $J_5$ and $J_6$, defined in \eqref{def:J}. They can be controlled by the same techniques to be used on $D_{12}$ below.

To prove the desired (\ref{eq:snkl}), (\ref{eq:estD12}) and (\ref{eq:estD22}), as part of $SN(\eta)$, we need $\tilde{K}_1(v)-\tilde{K}_1(\sqrt{q})$ to satisfy the following estimates. Recall the definition of $Y_{T_2}$ from \eqref{def:YT2}.
\begin{lemma} For any $(n,j)\in \Upsilon_m,$
\begin{align}\label{eq:k1v}
    \Big\| \langle y\rangle^{-n}\|\partial_{y}^j (-\Delta_{\mathbb{S}^3}+1)^{5}\chi_{Z} \Big(\tilde{K}_1(v)-\tilde{K}_1(\sqrt{q})\Big)\|_{L^2(\mathbb{S}^3)} \Big\|_{\infty}\lesssim \kappa \tilde{Z}_{m}^{-n} Y_{T_2}^{-2} \Big(1+\mathcal{M}\Big)^2.
\end{align}
\end{lemma} 

In the rest of this section we prove this lemma.

Compute directly to obtain
\begin{align*}
    \tilde{K}_1(v)-\tilde{K}_1(\sqrt{q})=&\frac{K_1(v)-K_1(\sqrt{q})}{1+\tilde{N}_{1}(v)}
    -\frac{\tilde{N}_1(v)-\tilde{N}_1(\sqrt{q})}{(1+\tilde{N}_{1}(v))(1+\tilde{N}_{1}(\sqrt{q}))}K_1(\sqrt{q}).
\end{align*} 
Compared to the terms $D_1$ and $D_2$ below, it is easier to control the second term since $|K_1(\sqrt{q})|$ is uniformly small in the considered region.
Among terms in the first part, we consider two of them and it is easier to treat the other ones. The two terms are
\begin{align}
\begin{split}
D_{1}:=&v^{-2} \frac{(\partial_{y}\eta)^2 \ \partial_{y}^2\eta}{1+\tilde{N}_1(v)},\\
D_{2}:=&v^{-2} \frac{(\partial_{y}q)^2  }{1+\tilde{N}_1(v)}\partial_{y}^2 \eta.
\end{split}
\end{align}

For the norms in (\ref{eq:k1v}), we consider the cases $j=0$ and $j\geq 1$ separately.

For the case $j=0$, we compute directly to obtain
\begin{align}\label{eq:prelD11}
\Big\|\langle y\rangle^{-m-1} \big\| (-\Delta_{\mathbb{S}^3}+1)^{5} \chi_{Z}D_{1}\big\|_{L^2(\mathbb{S}^3)}\Big\|_{\infty}
\lesssim & \delta H_1 H_2
\end{align} where $H_{1}$ and $H_2$ are defined as
\begin{align*}
H_1:=&\Big\|\langle y\rangle^{-m-\frac{1}{2}} \|(-\Delta_{\mathbb{S}^3}+1)^{5} \chi_{Z}\partial_{y} \eta\|_{L^2(\mathbb{S}^3)}\Big\|_{\infty},\\
H_2:=&\Big\| \langle y\rangle^{-\frac{1}{2}} \|q_{M}^{-1}1_{\leq (1+\epsilon)Z_m}(-\Delta_{\mathbb{S}^3}+1)^{5} \partial_{y} \eta\|_{L^2(\mathbb{S}^3)} \Big\|_{\infty},
\end{align*}
and we use the embedding result in Lemma \ref{LM:embedding}, and that $v^{-1}|\partial_{y}^2\eta|\lesssim \delta$ implied by \eqref{eq:conditionZm}.
To control $H_1$, we change the orders of $\partial_{y} $ and $\chi_{Z}$ and argue as in \eqref{eq:smallGamma} to find that
\begin{align}
H_1\leq & \Big\|\langle y\rangle^{-m-\frac{1}{2}} \|(-\Delta_{\mathbb{S}^3}+1)^{5} \partial_{y}\chi_{Z} \eta\|_{L^2(\mathbb{S}^3)}\Big\|_{\infty}+Z_{m}^{-m-\frac{3}{2}} \sup_{y,\omega}\Big|1_{\leq (1+\epsilon) Z_m}\eta(y,\omega,\tau)\Big|\nonumber\\
\leq &\tilde{Z}_{m}^{-m-\frac{1}{2}} Y_{T_2}^{-1}  \mathcal{M}(\tau)+Z_{m}^{-m-1}.
\end{align} 
Fir $H_2$, by the identities $1_{\leq (1+\epsilon)Z_m}=1_{\leq \tilde{Z}_m}+1_{\tilde{Z}_m<y\leq (1+\epsilon)Z_m}$ and $\chi_{Z}1_{\leq \tilde{Z}_m}=1_{\leq \tilde{Z}_m}$, and \eqref{eq:conditionZm},
\begin{align}
    H_2\lesssim \delta \tilde{Z}_{m}^{-\frac{1}{2}}+\Big\| \langle y\rangle^{-\frac{1}{2}} \|q_{M}^{-1}1_{\leq \tilde{Z}_m}(-\Delta_{\mathbb{S}^3}+1)^{5} \partial_{y} \eta\|_{L^2(\mathbb{S}^3)} \Big\|_{\infty}
    \lesssim \delta \tilde{Z}_{m}^{-\frac{1}{2}}+\tilde{Z}_{m}^m H_1.
\end{align}

Returning to \eqref{eq:prelD11}, we collect the estimates above to prove the desired 
\begin{align}
\Big\|\langle y\rangle^{-m-1} \big\| (-\Delta_{\mathbb{S}^3}+1)^{5} \chi_{Z}D_{1}\big\|_{L^2(\mathbb{S}^3)}\Big\|_{\infty} \lesssim  \delta Z_{m}^{-m-1}Y_{T_2}^{-1}(1+ \mathcal{M}(\tau)).
\end{align}

It is easier to control $D_{2}$ since for some $\epsilon_0>0$
$$\langle y\rangle^{-1} |\partial_{y}q|^2\ll e^{-\frac{m-1}{2m}\tau}\leq e^{-\epsilon_0\tau}Z_{m}(\tau).$$ Compute directly to obtain
\begin{align}\label{eq:mD12}
\Big\|\langle y\rangle^{-m-1} \big\| (-\Delta_{\mathbb{S}^3}+1)^{5} \chi_{Z}D_{12}\big\|_{L^2(\mathbb{S}^3)}\Big\|_{\infty}\lesssim  & e^{-\frac{m-1}{2m}\tau} \Big\|\langle y\rangle^{-m} \big\| (-\Delta_{\mathbb{S}^3}+1)^{5} \chi_{Z}\partial_{y}^2 \eta\big\|_{L^2(\mathbb{S}^3)}\Big\|_{\infty}\nonumber\\
\lesssim &  e^{-\frac{m-1}{2m}\tau} \Big[Z_{m}^{-m} Y_{T_2}^{-1}  \mathcal{M}+Z_{m}^{-m-\frac{1}{2}}\Big].
\end{align}

For the cases $j\geq 1$, the problem becomes easier since, for any term in $\partial_{y}^{j}D_{1}$, one of the factors must be $\partial_{y}^{h}\eta$ with $1\leq h\leq j$. Based on this, we find that if $(n,j)\in \Upsilon_{m}$, then
\begin{align}
A:=&\Big\|\langle y\rangle^{-n} \big\|\partial_{y}^{j} (-\Delta_{\mathbb{S}^3}+1)^{5} \chi_{Z}D_{1}\big\|_{L^2(\mathbb{S}^3)}\Big\|_{\infty}\nonumber\\
\lesssim & \delta\Big[\sum_{h=1}^{j}\Big\|\langle y\rangle^{-n} \big\|\partial_{y}^{h} (-\Delta_{\mathbb{S}^3}+1)^{5} \chi_{Z}\eta\big\|_{L^2(\mathbb{S}^3)}\Big\|_{\infty}+Z_{m}^{-n-\frac{1}{2}}\Big].
\end{align}
Recall the definitions of $T_2$ and $Y_{T_2}$ in \eqref{def:YT2}. Here we discuss two possibilities, 
\begin{itemize}
\item 
if $\tau\geq T_2$, then $2\geq Z_m/\tilde{Z}_{m}\geq 1$, and hence 
\begin{align}
A\lesssim & \delta\Big( \sum_{h=1}^{j} Z_{m}^{n_{h}-n} \Big\|\langle y\rangle^{-n_{h}} \big\|\partial_{y}^{h} (-\Delta_{\mathbb{S}^3}+1)^{5} \chi_{Z}\eta\big\|_{L^2(\mathbb{S}^3)}\Big\|_{\infty}+Z_{m}^{-n-\frac{1}{2}}\Big]\nonumber\\
\lesssim & \delta \tilde{Z}_{m}^{-n} Y_{T_2}^{-1} \Big(1+\mathcal{M}(\tau)\Big),\label{eq:pureTech2}
\end{align} where in the third step $n_h\geq n$ is the unique constant s.t. for the fixed $h\geq 1,$ $(n_{h},h)\in \Upsilon_{m};$ 
\item 
when $T_1\leq \tau<T_2$, then it might happen $\tilde{Z}_{m}(\tau)/Z_{m}(\tau)\ll 1$ since, for example, when $\tau=T_1$, $\tilde{Z}_m(T_1)=1$ while $Z_{m}(T_1)\gg 1$. This forces us to divide the region $|y|\leq (1+\epsilon)Z_m$ into two parts and consider them separately:
\begin{align}
&\Big\|\langle y\rangle^{-n} \big\|\partial_{y}^{h} (-\Delta_{\mathbb{S}^3}+1)^{5} \chi_{Z}\eta\big\|_{L^2(\mathbb{S}^3)}\Big\|_{\infty}\nonumber\\
\leq &\Big\|\langle y\rangle^{-n} 1_{\geq \tilde{Z}_m}\big\|\partial_{y}^{h} (-\Delta_{\mathbb{S}^3}+1)^{5} \chi_{Z}\eta\big\|_{L^2(\mathbb{S}^3)}\Big\|_{\infty} +\Big\|\langle y\rangle^{-n} 1_{< \tilde{Z}_m}\big\|\partial_{y}^{h} (-\Delta_{\mathbb{S}^3}+1)^{5} \chi_{Z}\eta\big\|_{L^2(\mathbb{S}^3)}\Big\|_{\infty}\nonumber\\
\lesssim & \delta \tilde{Z}_{m}^{-n}+ \tilde{Z}_{m}^{n_{h}-n} \Big\|\langle y\rangle^{-n_{h}} \big\|\partial_{y}^{h} (-\Delta_{\mathbb{S}^3}+1)^{5} \chi_{Z}\eta\big\|_{L^2(\mathbb{S}^3)}\Big\|_{\infty}\label{eq:pureTech}\\
\lesssim & \tilde{Z}_{m}^{-n} Y_{T_2}^{-1} \Big(1+\mathcal{M}(\tau)\Big), \nonumber
\end{align} where we use that here $Y_{T_2}(\tau)=1$ when $\tau\in [T_1, \ T_2]$.
\end{itemize}

Collect the estimates above to find that, for all the possibilities, $A$ satisfies the desired estimate
\begin{align}
    A\lesssim  \delta \tilde{Z}_{m}^{-n} Y_{T_2}^{-1} \Big(1+\mathcal{M}(\tau)\Big).
\end{align}

To control $\Big\|\langle y\rangle^{-n} \big\|\partial_{y}^{j} (-\Delta_{\mathbb{S}^3}+1)^{5} \chi_{Z}D_{2}\big\|_{L^2(\mathbb{S}^3)}\Big\|_{\infty}$, we observe that for some constants $C_{k,l}$, $$\partial_{y}^j D_{2}=\sum_{k+l=j}C_{k,l}\Big(\partial_{y}^{k}(\chi_Z v^{-2} \frac{(\partial_{y}q)^2  }{1+\tilde{N}_1(v)})\Big)\ (\partial_{y}^{l+2}\eta).$$ We prove these terms satisfy the desired estimates, by different techniques:
\begin{itemize}
\item $l+2=2m+3$ or $2m+4$, then we use that $|\partial_{y}^{l+2}\eta|\lesssim \delta$ by (\ref{eq:conditionZm});
\item when $2m+2\geq l+2>j$, then $l+2-j= 1$ or $2$ and we apply the techniques used in (\ref{eq:mD12}); 
\item when $l+2\leq j$ we apply the same techniques used in \eqref{eq:pureTech}.
\end{itemize}

%%%%%%%%%%%%%%%%%%%%%%%%%%%%%%%%%%%%%%%%%%%%%%%%%%%%%%%%%%%%%%%%%%%%%%%%%%%%%%%%%%%%%%%%%%%%%%%%%%%%%%%%%%%%%
\section{Proof of the part A of Theorem \ref{THM:sec}}\label{sec:ATHMsec}
The Part A states that it is impossible to have the following two possibilities: (1) $d_m\not=0$ when $m$ is odd and (2) $d_m<0$ when $m$ is even. In Lemma \ref{LM:contra} below we will rule out the possibility that $d_m>0$ and $m$ is odd, the others can be treated similarly, hence we choose to skip their proof.

We start with formulating the problem for the case $m$ is odd and $d_m>0$. Recall that we discussed the ideas behind the proof for the case $m=3$ in (\ref{eq:d3Region})-(\ref{eq:rescFunForm}).  

The proof is easier than that of B-part of Theorem \ref{THM:sec}. We will take some ideas from there.

For any fixed $\epsilon_0>0$ we parametrize the rescaled MCF as $\Psi_{Q_{m-1},v}$, and study $v$ in the region
\begin{align}\label{def:YM3}
\Bigg\{y\in \mathbb{R}\ \Bigg|\ -(1+\epsilon)
\Big((1+\epsilon_0^2) \frac{6-\epsilon_0}{d_m}\Big)^{\frac{1}{m}}\leq \frac{y}{Y_{m}(\tau)}\leq (1+\epsilon)\epsilon_0^{\frac{1}{m}}\Bigg\},
\end{align} 
where $Y_{m}$ is a function defined as
\begin{align}
    Y_{m}(\tau):=e^{\frac{m-2}{2m}(\tau-T_1)}+4\tau^{\frac{1}{2}+\frac{1}{20}},
\end{align}
and $\epsilon>0$ is the small constant to appear in the definition of cutoff function, see \eqref{def:chi21} below, and $T_1$ is a large time.  

To avoid comparing the sizes of small constants $\epsilon$ and $\epsilon_0$, we require that 
\begin{align}
  0<  \frac{\epsilon}{\epsilon_0^2}\leq  1.
\end{align}

The cutoff function $\chi_{Y}$ to be used in \eqref{eq:orthom3} is defined as
\begin{align}\label{eq:chim3}
\chi_{Y}(y):=\chi_{2}(ye^{-\frac{m-2}{2m}\tau})
\end{align} and $\chi_{2}$ is a $C^{m+20}_1$ cutoff function satisfying
\begin{align}\label{def:chi21}
    \chi_{2}(x)=\left[
    \begin{array}{ll}
       1  & \text{when}\ x\in \Big[-(\frac{6-\epsilon_0}{d_m})^{-\frac{1}{m}},\ \epsilon_0^{-\frac{1}{m}}\Big]; \\
       0  &  \text{when}\ x\leq -(1+\epsilon)(\frac{6-\epsilon_0}{d_m})^{-\frac{1}{m}},\ \text{or}\ x\geq (1+\epsilon)\epsilon_0^{-\frac{1}{m}}.
    \end{array}
    \right.
\end{align}
Here we construct $\chi_2$ from $\chi$ defined in \eqref{def:chi} and \eqref{eq:chiSmooth}:
\begin{align}\label{def:chi2}
\chi_{2}(z)=\left[
\begin{array}{cc}
\chi\Big((\frac{6-\epsilon_0}{d_m})^{-\frac{1}{m}}z\Big) &\ \text{when}\ z\leq 0,\\
\chi\Big((\epsilon_0)^{-\frac{1}{m}}z\Big) &\ \text{when}\ z\geq 0.
\end{array}
\right.
\end{align}

The result is following:
\begin{proposition}\label{prop:step1Sec}
For the rescaled MCF $\Psi_{Q_{m-1},v}$ in the region defined in \eqref{def:YM3},
the function $v$ can be decomposed into the form
\begin{align}\label{eq:decomvm3}
v(y,\omega,\tau)=\sqrt{6+\sum_{n=0}^{m} \sum_{k=0}^{N}\sum_l\beta_{n,k,l}(\tau)H_{n}(y)f_{k,l}(\omega) +\xi(y,\omega,\tau)}
\end{align} where $N$ is a sufficiently large integer, and $\chi_{Y}\xi$ satisfies the orthogonality conditions
\begin{align}
\chi_{Y}\xi\perp_{\mathcal{G}} H_n f_{k,l},\ n=0,\cdots,m;\ k=0,\cdots,N,\label{eq:orthom3}
\end{align} and $\chi_{Y}$ is a cutoff function defined in \eqref{eq:chim3},
and the most important term is $\beta_{m,0,1}$, 
\begin{align}
\beta_{m,0,1}(\tau)=d_m e^{-\frac{m-2}{2}\tau} +\mathcal{O}(e^{-\frac{m-1}{2}\tau}),
\end{align} $d_m$ is the same to that in \eqref{eq:dm}, and $\beta_{n,k,l}$ enjoy the same estimates to $\gamma_{n,k,l}$ in \eqref{eq:dm}-\eqref{eq:n01Rapid}.   

The remainder $\eta$ satisfies the estimates
\begin{align}
\sum_{j+|i|\leq 2}\Big\|\partial_{y}^j \nabla_{E}^{i}\chi_{Y}\eta(\cdot,\tau)\Big\|_{\mathcal{G}}\leq & C_{\delta_0} e^{-\frac{m-\delta_0}{2}\tau} ,\label{eq:estYeta}\\
\sum_{j+|i|\leq 2}\Big|\partial_{y}^j\nabla_{E}^i\chi_{Y}\eta(\cdot,\tau)\Big|&\leq e^{-\sqrt{\tau}}.\label{eq:Ypoint}
\end{align}
\end{proposition}
We will discuss the proof shortly.

Assume the proposition holds, then we prove that the second fundamental form is unbounded. This contradicts the assumption \eqref{eq:secFun}, which implies that, $|\tilde{A}|\leq C$ for some fixed constant $C$. Recall that $\epsilon_0$ is an arbitrary, but fixed, positive constant.
\begin{lemma}\label{LM:contra}
Let $\tilde{A}(y,\omega,\tau)$ be the second fundamental form for the considered part of rescaled MCF, then we have that, for some constant $M(\epsilon_0)$ satisfying $\lim_{\epsilon_0\rightarrow 0+}M(\epsilon_0)= \infty,$
\begin{align}
    \sup_{y,\omega,\tau}|\tilde{A}(y,\omega,\tau)|\geq M(\epsilon_0).
\end{align} 
\end{lemma}
\begin{proof}
Recall that we parametrize the rescaled MCF as
\begin{align}
\Psi_{Q_{m-1}}(v)=\left[
\begin{array}{ccc}
y\\
Q_{m-1}(y,\tau)
\end{array}
\right]+v(y,\omega,\tau)\left[
\begin{array}{ccc}
-\partial_{y}Q_{m-1}(y,\tau)\cdot\omega\\
\omega
\end{array}
\right],
\end{align} and we study the region
$$-1\leq \frac{y}{Y_{m}(\tau)\Big((1+\epsilon_0^2) \frac{6-\epsilon_0}{d_m}\Big)^{\frac{1}{m}}}\leq -1+\epsilon_0.$$

The observation $|\partial_{y}Q_{m-1}(y,\tau)|\rightarrow 0$ uniformly as $\tau\rightarrow\infty $ implies that the curve $\left[
\begin{array}{ccc}
y\\
Q_{m-1}(y,\tau)
\end{array}
\right]$ varies very slowly, and locally is almost a straight line. The estimates in Proposition \ref{prop:step1Sec} imply 
\begin{align}
v(y,\omega,\tau)=\sqrt{6+ d_m e^{-\frac{m-2}{2}\tau}y^m}\Big(1+o(1)\Big)=\sqrt{\epsilon_0}\Big(1+o(1)\Big),
\end{align} when 
$ y= -Y_{m}(\tau)\Big((1+\frac{1}{2}\epsilon_0^2) \frac{6-\epsilon_0}{d_m}\Big)^{\frac{1}{m}}\Big(1+o(1)\Big).$ 
For the smoothness, the estimates in Proposition \ref{prop:step1Sec} imply that $$\lim_{\tau\rightarrow \infty}\sum_{k+|l|=1,2}|\partial_{y}^k\nabla_{E}^{l}v(\cdot,\tau)|= 0.$$ Hence the second fundamental form $\tilde{A}$ satisfies
\begin{align}
\sup_{y,\omega,\tau}|\tilde{A}(y,\omega,\tau)|\geq M(\epsilon_0)\rightarrow \infty \ \text{as}\ \epsilon_0\rightarrow 0.
\end{align}  

\end{proof}

%%%%%%%%%%%%%%%%%%%%%%%%%%%%%%%%%%%%%%%%%%%%%%%%%%%%%%%%%%%%%%%%%%%%%%%%%%%%%%%%

Next we discuss the proof of Proposition \ref{prop:step1Sec}. For a fully detailed proof, we need to go through the whole procedure similar to that in Section \ref{sec:BTHMSec}. We will skip most of the steps, and only focus on the parts requiring different techniques.

We obtain the desired estimates for $\beta_{n,k,l}$ and the estimates for $\chi_{Y}\eta$\eqref{eq:estYeta} by comparing them to $\gamma_{n,k,l}$ and $\chi_{R}\xi$ in \eqref{eq:dm}, by the techniques used in the proof of \eqref{eq:DiffAlBe}.

Now we discuss the proof of \eqref{eq:Ypoint}. Here we need the governing equation for $\chi_{Y}\eta$. To simply the notations we define new functions $q$ and $q_M$ by
\begin{align}
\begin{split}
q(y,\omega,\tau):=&6+\sum_{n=0}^{m} \sum_{k=0}^{N}\sum_{l}\gamma_{n,k,l}(\tau)H_{n}(y)f_{k,l}(\omega),\\
q_M(y,\tau):=&6+\gamma_{n,0,1}(\tau)H_n(y).
\end{split}
\end{align} 
Similar to deriving \eqref{eq:EffTilXi},
\begin{align}
\partial_{\tau}\chi_{Y}\eta=-L\chi_{Y}\eta+\chi_{Y}\Big(G+SN(\eta)\Big)+\mu_{Y}(\eta)\label{eq:etaY}
\end{align} where the linear operator $L$ is defined as, 
\begin{align*}
L:=-\partial_{y}^2+\frac{1}{2}y\partial_y-\frac{1}{q_M} \Delta_{\mathbb{S}^3}-1
\end{align*}
the function $G=G_1+G_2$ is defined as 
\begin{align*}
G_1(y,\omega,\tau):=&-\sum_m H_{n}(y)f_{k,l}(\omega) \Big[\frac{d}{d\tau}+\frac{n}{2}-1+\frac{1}{6} k(k+2)\Big]\alpha_{n,k,l}(\tau),\\
G_2(y,\omega,\tau):=&\frac{6-q}{6q}\Delta_{\mathbb{S}^3}q-\frac{1}{2q^2} |\nabla_{E}q|^2-\frac{1}{2q}|\partial_{y}q|^2+2\sqrt{q} N(\sqrt{q})+2\sqrt{q}W_{Q_{m-1}}(\sqrt{q}),
\end{align*}
and
\begin{align*}
SN(\eta):=&\frac{1}{2q^2}|\nabla_{E}q|^2-\frac{1}{2}v^{-4} |\nabla_{E}v^2|^2+\frac{1}{2q}|\partial_{y}q|^2-\frac{1}{2}v^{-2}|\partial_{y}v^2|^2+2v N(v)-2\sqrt{q} N(\sqrt{q})\\
&+2vW_{Q_{m-1}}(v)-2\sqrt{q}W_{Q_{m-1}}(\sqrt{q})+\frac{q_M-q}{q_{M}q}\Delta_{\mathbb{S}^3}\eta-\frac{\eta}{v^2 q}\Delta_{\mathbb{S}^2}v^2,
\end{align*}
and
\begin{align}
\mu_{Y}(\eta):=\frac{1}{2}\big(y\partial_{y}\chi_{Y}\big)\eta+\big(\partial_{\tau}\chi_{Y}\big)\eta-\big(\partial_{y}^2\chi_{Y}\big)\eta-2(\partial_{y}\chi_{Y})(  \partial_{y}\eta).\label{eq:Tchi3}
\end{align}

We are ready to present a difficulty. 

Observe that $\frac{1}{q_M}$ becomes $\infty$ at some  $y_0=-e^{\frac{m-2}{2m}\tau}(\frac{6}{|d_m|}e^{\frac{m-2}{2}\tau})^{\frac{1}{m}}\Big(1+o(1)\Big)$. This is adverse because, for example, we can not apply the crucial tool Lemma \ref{LM:frequencyWise}.

It is not hard to overcome this difficulty, since the difficulty takes place outside the interested region \eqref{def:YM3}. By restricting to the region \eqref{def:YM3} we can choose a ``better" multiplier $\tilde{V}(y,\tau)$ s.t.
\begin{align}
\tilde{V}\chi_{Y}\eta=q_M^{-1}\chi_{Y}\eta.
\end{align} We need $\tilde{V}$ to be a smooth function and satisfy
\begin{align}\label{eq:tildeV1}
\tilde{V}(y,\tau):=
\left[
\begin{array}{ccc}
 q_M^{-1}(y,\tau)\ &\text{when} &\ - y_{\text{min}}\leq y/Y_m\leq y_{\text{max}} ,\\
q_M^{-1}(Y_m y_{\text{max}},\tau) &\text{when} &\ y/Y_{m}\geq y_{\text{max}}+\epsilon_0^2,\\
q_M^{-1}(-Y_m y_{\text{min}},\tau)  &\text{when} &\ y/Y_m\leq -y_{\text{min}}-\epsilon_0^2, 
\end{array}
\right.
\end{align} 
and when $y +Y_m y_{\text{min}}\in Y_m [-\epsilon_0^2,0]$ and $y-Y_m y_{\text{max}}\in Y_m[0,\epsilon_0^2] $, we require that
\begin{align*}
\Big|\partial_{y}\tilde{V}(y,\tau)\Big|\leq e^{-\frac{2}{5}\tau}.
\end{align*} To construct such a function is easy, and is skipped.
Here $y_{\text{max}}$ and $y_{\text{min}}$ are constants defined as
\begin{align*}
    y_{\text{max}}:=&(1+\epsilon)\Big(\epsilon_0 e^{\frac{m-2}{2}\tau}\Big)^{\frac{1}{m}},\\
    y_{\text{min}}:=&(1+\epsilon)\Big(\frac{6-\epsilon_0}{d_m}e^{\frac{m-2}{2}\tau}\Big)^{\frac{1}{m}},
\end{align*}

Thus the equation \eqref{eq:etaY} becomes
\begin{align}
\partial_{\tau}\chi_{Y}\eta=-H_Y\chi_{Y}\eta+\chi_{Y}\Big(G+SN(\eta)\Big)+\mu_{Y}(\eta)\label{eq:m3Eta0}
\end{align} where the linear operator $H_Y$ is 
\begin{align*}
H_Y:=-\partial_{y}^2+\frac{1}{2}y\partial_{y}-1 +\tilde{V}.
\end{align*}

What is left is almost identical to the proof in Section \ref{sec:BTHMSec}. Thus we choose to skip the details.

%%%%%%%%%%%%%%%%%%%%%%%%%%%%%%%%%%%%%%%%%%%%%%%%%%%%%%%%%%%%%%%%%%%%%%%%%%%%%%%%%%%%%%%%%%%%%%%%%%%%%%%%%%%%%%%%%%%%%%%%%%%%%
\section{Proof of the part D of Theorem \ref{THM:sec}}\label{sec:ProofTHMsecC}
We will achieve the goal in two steps: we first study the old $\Psi_{Q_{m-1},v}$, but expand the function $v$ further; then we will make a normal form transformation and study the new rescaled MCF.

Here we need Proposition \ref{prop:Mconvex}, whose estimates have been proved to hold in the time interval $[T_1,\infty)$ for some $T_1\gg 1$. To simply the notations, here we shift the time to make $T_1=0.$

We are ready to start the first step, recall the definitions of $R$ and $\chi_{R}$ in (\ref{def:R}) and \eqref{eq:chiRy},
\begin{proposition}\label{prop:m1BefNorm}
For the $\Psi_{Q_{m-1},v}$, with $Q_{m-1}$ defined in \eqref{eq:qmminus1}, in the region
\begin{align}
    |y|\leq (1+\epsilon)R(\tau),
\end{align}
there exists an integer $N_{m+1}$ such that for any $N\geq N_{m+1}$, $v$ can be decomposed into
\begin{align}\label{eq:decomVm1}
    v(y,\omega,\tau)=\sqrt{6+\sum_{n=0}^{m+1}\sum_{k=0}^{N}\sum_{l} \beta_{n,k,l}(\tau)H_{n}(y)f_{k,l}(\omega)+\zeta(y,\ \omega,\tau)},
\end{align} such that $\chi_{R}\zeta$ satisfies the orthogonality conditions
\begin{align}
\chi_{R}\zeta\perp_{\mathcal{G}}  H_{n}f_{k,l},\ n=0,1,\cdots,m+1;\ \text{and}\ k=0,\cdots,N,\ \label{eq:orthoG}
\end{align} and when $n\leq m-1$, $\beta_{n,k,l}$ is closed to $\gamma_{n,k,l}$ in \eqref{eq:newdec}, for any $\delta_0>0,$ there exists a $C_{\delta_0}$ s.t.
\begin{align}
   \Big|\gamma_{n,k,l}(\tau)-\beta_{n,k,l}(\tau)\Big|\leq  & C_{\delta_0} e^{-\frac{m-1-\delta_0}{2}\tau},\label{eq:betagamma}
\end{align} here we assume that $\gamma_{n,k,l}=0$ if it is not in the decomposition \eqref{eq:newdec}.

The focus is on $\beta_{m,1,l},$ $l=1,2,3,4:$ for some constants $d_{m,l},$
\begin{align}
    \beta_{m,1,l}(\tau)=&d_{m,l} e^{-\frac{m-1}{2}\tau}+\mathcal{O}(e^{-\frac{m}{2}}),\label{eq:betaMl}
\end{align} and for the ones not included in \eqref{eq:betagamma} and \eqref{eq:betaM1}: for any $l$, for any $j\not=1$ and any $k\not=0$,
\begin{align} 
  e^{\frac{m-1}{2}\tau} |\beta_{m+1,0,1}|+(1+\tau)^{-1} e^{\frac{m}{2}\tau}|\beta_{m+1,k,l}(\tau)|+e^{\frac{m}{2}\tau}|\beta_{m,j,l}(\tau)| \lesssim  1.
\end{align} 
the remainder $\zeta$ satisfies the following estimates, for any $\delta_0>0,$ there exists a $C_{\delta_0}$ s.t.
\begin{align}
    \sum_{n+|j|\leq 2}\Big\|\partial_{y}^n\nabla_{E}^j\chi_{R} \zeta(\cdot,\tau)\Big\|_{\mathcal{G}}\leq & C_{\delta} e^{-\frac{m-\delta_0}{2}\tau},\label{eq:etaGnorm}\\
    \sum_{n+|j|\leq 2}\Big| \partial_{y}\nabla_{E}^{l}\chi_{R}\zeta(\cdot,\tau)\Big|\lesssim & e^{-\sqrt{\tau}}.\label{eq:mplus1eta}
\end{align}
\end{proposition}
The proposition will be proved in Subsection \ref{sub:m1BefNorm} below.

In the next result we remove the part $d_{m,l}e^{-\frac{m-1}{2}\tau} H_{m}\omega_l$ by reparametrizing the rescaled MCF, and then study the new MCF. The following result implies the desired D-part of Theorem \ref{THM:sec}.
\begin{proposition}
We parametrize the rescaled MCF as $\Psi_{Q_m,v}$ with $Q_m$ of the form 
\begin{align}
Q_{m}(y,\tau)=\sum_{n\geq 2}^{m} H_{n}(y) e^{-\frac{n-1}{2}\tau}\Big(a_{n,1}, \ a_{n,2},\ a_{n,3},\ a_{n,4}\Big)^{T},
\end{align} and decompose the function $v$ into the form
\begin{align}
    v(y,\omega,\tau)=\sqrt{6+\sum_{n=0}^{m+1}\sum_{k=1}^{N}\sum_{l} \gamma_{n,k,l}(\tau)H_{n}(y)f_{k,l}(\omega)+\xi(y,\ \omega,\tau)},\label{eq:decomThM}
\end{align}
where $N$ is a sufficiently large integer, the remainder $\chi_{R}\xi$ satisfies the orthogonality conditions
\begin{align}
\chi_{R}\xi\perp_{\mathcal{G}}  H_{n}f_{k,l},\ n=0,1,\cdots,m+1;\ \text{and}\ k=0,\cdots,N,\ \label{eq:orthoG10}
\end{align}
these functions satisfy the same estimates to the corresponding parts in the D-part of Theorem \ref{THM:sec}. 

\end{proposition}
The proof is very similar to those of Propositions \ref{prop:afterNorForm} and \ref{prop:m1BefNorm}.
Thus we choose to skip it.

This completes the proof of part D of Theorem \ref{THM:sec}.

In the next subsection we prove Proposition \ref{prop:m1BefNorm}.
\subsection{Proof of Proposition \ref{prop:m1BefNorm}}\label{sub:m1BefNorm}
Similar to the proof of Proposition \ref{prop:step2} in subsection \ref{sec:PropStep2}, we have to prove the existence of the integer $N_{m+1}$. This is necessary because we have to remove, from $\chi_{R}\zeta,$ any slowly decay directions $H_{n}f_{k,l}$ when $n\leq m+1,$ otherwise $\|\chi_{R}\zeta(\cdot,\tau)\|_{\mathcal{G}}$ does not decay with the desired high rate.

In the proof we will need a finer description for the function $\eta$ in Proposition \ref{prop:Mconvex}, recall the definition of the projection operator $P_{\omega, N}$ from (\ref{def:OmePro}),
\begin{lemma}
There exists an integer $M$, s.t. for any $M_1\geq M$ there exists a constant $C_{M_1}$ and,
\begin{align}
   \sum_{j+|i|\leq 2} \Big\|\partial_{y}^j \nabla_{E}^i P_{\omega,M_1}\chi_{Z}\eta(\cdot,\tau)\Big\|_{\mathcal{G}}\leq C_{M_1} e^{-\frac{m+1}{2}\tau}.
\end{align}
\end{lemma}
The proof is easy since, at least intuitively, the part at the higher frequencies of the $\omega-$space should decay faster. Since it is similar to the proof of \eqref{eq:arbiFast}, we skip the details.

Now we define $N_{m+1}$ to be the smallest integer satisfying the following conditions
\begin{align}
    N_{m+1}\geq M, \ \text{and}\ \frac{N_{m+1}(N_{m+1}+2)}{6}\geq \frac{m+1}{2}+1.
\end{align}

Next we study $\Psi_{Q_{m-1},v}$ in the region $|y|\leq (1+\epsilon)R(\tau)$, and decompose $v$ as,
\begin{align}
    v(y,\omega,\tau)=\sqrt{6+\sum_{n=0}^{m+1}\sum_{k=0}^{N }\sum_l \beta_{n,k,l}(\tau) H_{n}(y)f_{k,l}(\omega)+\zeta(y,\omega,\tau)},\label{eq:m11Decom}
\end{align} for any $N\geq N_{m+1}$, such that $\chi_{R}\zeta$ satisfies the orthogonality conditions
\begin{align*}
    \chi_{R}\zeta \perp_{\mathcal{G}} H_n f_{k,l}, \ n\leq m+1;\ k=0,\cdots, N.
\end{align*}

We provide preliminary estimates by comparing them to those in Proposition \ref{prop:Mconvex}, whose estimates have been proved to hold for $\tau\in [T_1,\infty).$ Recall that we shift the time to make $T_1=0.$
\begin{lemma}\label{LM:prem11}
For any $N\geq N_{m+1}$ and $\delta_0>0$, there exist constants $C_{\delta_0,N}$ and $C_{N}$ such that
\begin{align}
    \sum_{j+|i|\leq 2}\Big\|\partial_{y}^j \partial_{E}^i\chi_{R}\zeta(\cdot,\tau)\Big\|_{\mathcal{G}}+|\beta_{m+1,k,l,}(\tau)|\leq & C_{\delta_0,N} e^{-\frac{m-1-\delta_0}{2}\tau},\label{eq:betaM1}\\
    \sum_{j+|i|\leq 2}\Big\|\partial_{y}^j \partial_{E}^iP_{\omega, N }\chi_{R}\zeta(\cdot,\tau)\Big\|_{\mathcal{G}}\leq & C_N e^{-\frac{m+1}{2}\tau},\label{eq:largeOEta}\\
    \sum_{n\leq m}\Big|\alpha_{n,k,l}(\tau)-\beta_{n,k,l}(\tau)\Big|\ll & e^{-\frac{1}{5}R^2}.\label{eq:closedness}
\end{align} 
\end{lemma}
The proof is very similar to that of \eqref{eq:DiffAlBe}, and hence is skipped. 

This lemma and the estimates for $\alpha_{n,k,l}$ and $\chi_{Z}\eta$ in Proposition \ref{prop:Mconvex} and Corollary \ref{LM:improvement} imply the desired estimates for $\beta_{n,k,l},$ $n\leq m-1,$ and \eqref{eq:mplus1eta}.

What is left is to improve estimates for $\beta_{m,k,l}$, $\beta_{m+1,k,l}$ and the $\mathcal{G}-$norm of $\chi_{R}\zeta$ by studying their governing equations. This is similar to our treatment in subsection \ref{subsec:afterNorm3}.

To simplify the notations we define a function $\beta $ as
\begin{align*}
\beta(y,\omega,\tau):=6+\sum_{n=0}^{m+1}\sum_{k=0}^{N }\sum_l \beta_{n,k,l}(\tau) H_{n}(y)f_{k,l}(\omega).
\end{align*} 
By this we derive
\begin{align}\label{eq:EffTilXiM1}
\partial_{\tau}\chi_{R}\zeta=-L\chi_{R}\zeta+\chi_{R}\Big(F+SN\Big)+\mu_{R}(\zeta),
\end{align} where the linear operator $L$ is defined in \eqref{def:L}, and $\mu_{R}(\zeta)$ is defined in the same way as \eqref{eq:effectChiREta},
the function $F$ is independent of $\eta$ and is defined as 
\begin{align}
\begin{split}
F:=&F_1+F_2,
\end{split}
\end{align} and the functions $F_1$ and $F_2$ are defined as, 
\begin{align*}
F_1:=&-\sum_{n=0}^{m+1}\sum_{k=0}^N\sum_{l}(\frac{d}{d\tau}+\frac{n-2}{2}+\frac{k(k+2)}{6})\beta_{n,k,l} H_{n}f_{k,l},\\
F_2:=&\frac{6-\beta}{6\beta }\Delta_{\mathbb{S}^3}\beta-\frac{1}{2\beta^2} |\nabla_{E}\beta |^2-\frac{1}{2\beta }|\partial_{y}\beta |^2+2\sqrt{\beta } N(\sqrt{\beta })+2\sqrt{\beta }W_{Q_{m-1}} (\sqrt{\beta }),
\end{align*}
the term $SN(\zeta)$ contains terms nonlinear in terms of $\zeta$, or ``small" linear terms,
\begin{align*}
SN(\zeta ):=&\frac{1}{2\beta ^2}|\nabla_{E}\beta |^2-\frac{1}{2}v^{-4} |\nabla_{E}v^2|^2+\frac{1}{2\beta }|\partial_{y}\beta |^2-\frac{1}{2}v^{-2}|\partial_{y}v^2|^2
+\frac{6-\beta }{6\beta }\Delta_{\mathbb{S}^3}\xi-\frac{\zeta }{v^2 \beta }\Delta_{\mathbb{S}^3}v^2\\
&+2v N(v)-2\sqrt{\beta } N(\sqrt{\beta })+2v W_{Q_{m-1} }(v)-2\sqrt{\beta}W_{Q_{m-1} }(\sqrt{\beta }).
\end{align*}
By the orthogonality conditions imposed on $\chi_{R}\zeta$ in \eqref{eq:m11Decom},
\begin{align}
\Big(\frac{d}{d\tau}+\frac{n-2}{2}+\frac{k(k+2)}{6}\Big)\beta_{n,k,l}=NL_{n,k,l},
\end{align} with
\begin{align*}
NL_{n,k,l}:=\frac{1}{\|H_nf_{k,l}\|_{\mathcal{G}}^2}\Big\langle  F_2+SN+\mu_{R}(\eta),\ H_{n}f_{k,l}\Big\rangle_{\mathcal{G}}
\end{align*}

Now we feed the preliminary estimates in Lemma \ref{LM:prem11}
into the governing equations above to improve estimates.
\begin{lemma}\label{LM:m11}
For any $\delta_0>0$ there exists a constant $C_{\delta_0}$ such that the following estimates hold,
\begin{align}
\begin{split}\label{eq:sourceM1}
\Big|\Big\langle \chi_{R}F_2, H_{n}f_{k,l}\Big\rangle_{\mathcal{G}}-c_{n,k,l} e^{-\frac{n}{2}\tau}\Big|\leq & e^{-(\frac{n}{2}+\frac{1}{3})\tau},\ \text{if}\ n\leq m,\ \text{and}\ k\leq N_{m+1},\\
|\Big\langle \chi_{R}F_2, H_{m+1}f_{k,l}\Big\rangle_{\mathcal{G}}|\leq & C_{\delta_0}e^{-\frac{m-\delta_0}{2}\tau}, \ \text{if}\ k\leq N_{m+1},\\
\Big\|P_{m+1} e^{-\frac{1}{8}y^2}F_2\Big\|_{2} \leq & C_{\delta_0}e^{-\frac{m-\delta_0}{2}\tau},
\end{split}
\end{align} where $c_{n,k,l}$ are constants, and
\begin{align}\label{eq:LarOmPro}
\Big\|\text{P}_{\omega, N_{m+1}} \chi_{R}F_2\Big\|_{\mathcal{G}}\lesssim e^{-\frac{m+1}{2}\tau},
\end{align} the remainder $\zeta$ satisfies the estimate
\begin{align}
\sum_{j+|i|\leq 2}\Big\|\partial_{y}^{j}\nabla_{E}^i\chi_{R}\zeta(\cdot,\tau)\Big\|_{\mathcal{G}}\leq C_{\delta_0}e^{-\frac{m-\delta_0}{2}\tau}.\label{eq:gzeta}
\end{align}
\end{lemma}
\begin{proof}

\eqref{eq:LarOmPro} must hold since $\text{P}_{\omega, N_{m+1}} \chi_{R}F_2$ is in the governing equation for $\text{P}_{\omega, N_{m+1}}\chi_{R}\eta$, whose decay rate is $e^{-\frac{m+1}{2}\tau}$ by \eqref{eq:largeOEta}.

The proof of \eqref{eq:sourceM1} is similar to that of \eqref{eq:estBetaEff}. We start with decomposing $\beta$ into three parts
\begin{align}
   \beta=6+\beta_L+\beta_{H}
\end{align} with $\beta_L$ and $\beta_{H}$ defined as
\begin{align*}
\beta_{L}(y,\omega,\tau):=&\sum_{n=0}^{m-1}\sum_k\sum_l \beta_{n,k,l}(\tau) y^n f_{k,l}(\omega),\\
\beta_{H}(y,\omega,\tau):=&H_m(y) \sum_{k,l}\beta_{m,k,l}(\tau) f_{k,l}(\omega)+H_{m+1}(y)\sum_{k,l}\beta_{m+1,k,l}(\tau) f_{k,l}(\omega).
\end{align*}

Decompose $F_2$ into two parts 
\begin{align}
    F_2(\sqrt{\beta})=F_{2}(\sqrt{6+\beta_{L}})+\Big( F_{2}(\sqrt{\beta})-F_{2}(\sqrt{6+\beta_{L}})\Big).
\end{align}
By reasoning as in \eqref{eq:qqqM}-\eqref{eq:remainder} and applying \eqref{eq:closedness}, we prove the contribution from $F_{2}(\sqrt{6+\beta_{L}})$ satisfies the estimates in \eqref{eq:sourceM1}. For the part $F_{2}(\sqrt{\beta})-F_{2}(\sqrt{6+\beta_{L}})$, \eqref{eq:betaM1}, \eqref{eq:closedness} and Corollary \ref{LM:improvement} provide estimates for $\beta-6-\beta_{L}=\beta_{H}$, specifically, for any $\delta_0>0$, there exists a $C_{\delta_0}$ such that
\begin{align}
|\beta_{m,k,l}(\tau)|+|\beta_{m+1,k,l}(\tau)|\leq C_{\delta_0}e^{-\frac{m-1-\delta_0}{2}\tau}.\label{eq:betamm1}
\end{align}
Reasoning as in \eqref{eq:alDif} to prove that the contribution from this part satisfies the estimates in \eqref{eq:sourceM1}.

The third estimate in \eqref{eq:sourceM1} and the governing equation for $\chi_{R}\zeta$ imply \eqref{eq:gzeta}.
\end{proof}

What is left is to use Lemma \ref{LM:m11} to improve estimates for $\beta_{n,k,l},\ n=m,\ m+1,$ and $\chi_{R}\zeta$ by studying the governing equations, and then feed the obtained ones into the governing equations to obtain the desired estimates. This is similar to the treatment in subsection \ref{subsec:afterNorm3}, hence is skipped.

%%%%%%%%%%%%%%%%%%%%%%%%%%%%%%%%%%%%%%%%%%%%%%%%%%%%%%%%%%%%%%%%%%%%

\appendix

\section{Derivation of Equation \texorpdfstring{\eqref{eq:UPparametrization}}{}}\label{sec:derivation}
Recall that the MCF takes the form
\begin{align} 
\Phi_{\Pi}(u)= \left[
\begin{array}{ccc}
z\\
\Pi(z,t)
\end{array}
\right]+u(z,\omega,t)\left[
\begin{array}{ccc}
-(\partial_{z} \Pi(z, t))\cdot \omega\\
\omega
\end{array}
\right],
\end{align}
with $\omega\in \mathbb{S}^3$ defined as
\begin{align*}
    \omega:=\frac{(x_2,\cdots,x_5)^{T}-\Pi(z,t)}{|(x_2,\cdots,x_5)^{T}-\Pi(z,t)|}.
\end{align*}

The result is that
\begin{align}
\partial_{t}u=u^{-2}\Delta_{\mathbb{S}^3}u+\partial_{z}^2u-\frac{3}{u}+N(u)+V_{\Pi}(u)\label{eq:effectivUeqn}
\end{align} where $N$ is nonlinear in terms of $u$, and is defined as 
\begin{align}
N(u):=&-\frac{u^{-4}\sum_{k=1}^3( \nabla_{\omega}^{\perp} u\cdot e_{k})\Big( \nabla_{\omega}^{\perp} u \cdot \nabla_{\omega}^{\perp}( \nabla_{\omega}^{\perp} u\cdot e_{k})\Big)}
{1+u^{-2}|  \nabla_{\omega}^{\perp} u|^2+(\partial_{z}u)^2}\nonumber\\
&-\frac{(\partial_{z}u)^2\partial_{z}^2 u+2u^{-2}\partial_{z}u ( \nabla_{\omega}^{\perp} u\cdot \nabla_{\omega}^{\perp} \partial_{x}u)+u^{-3}|
\nabla_{\omega}^{\perp} u|^2}{1+u^{-2}|
\nabla_{\omega}^{\perp} u|^2+(\partial_{z}u)^2}\nonumber\\
=&\tilde{N}(z,\omega,t,|\bf{\tilde{x}}|)\Big|_{|{\bf{\tilde{x}}}|=u},\label{eq:tildeN}
\end{align} 
and $e_{j}\in \mathbb{R}^4, \ j=1,2,3,4,$ are the standard unit vectors
\begin{align}
    e_{j}:=\Big(\delta_{1,j},\delta_{2,j},\delta_{3,j},\delta_{4,j}
    \Big);
\end{align}
the function $V_{\Pi}(z,\omega,t)$ depends on $\Pi$ in the sense that $$V_{\Pi}\Big|_{\Pi\equiv 0}\equiv 0,$$ 
and is defined by the identity
\begin{align}
V_{\Pi}(z,\omega,t):=\widetilde{V_{\Pi}}(z,\omega,t,|{\bf{\tilde{x}}}|)\Big|_{|{\bf{\tilde{x}}}|=u}.
\end{align} $\tilde{N}$ is defined as, 
\begin{align}\label{def:tildeN}
    \tilde{N}:=\sum_{k=1}^5\Omega_{k,k}+\frac{\sum_{l=1}^5\sum_{j=1}^{5}\Omega_l \Omega_j
    \Omega_{l,j}}{\sum_{n=1}^{5}\Omega_n^2},
\end{align} and $\widetilde{V_{\Pi}}$ is defined as,
\begin{align}\label{def:widVPi}
\begin{split}
    \widetilde{V_{\Pi}}:=&\frac{Q_2}{1+Q_1}+\frac{Q_1}{1+Q_1}\Big[\sum_{k=1}^{5} (\Omega_{k,k}+\Sigma_{k,k})-\frac{\sum_{l=1}^5\sum_{j=1}^{5}(\Omega_l+\Sigma_l)(\Omega_j+\Sigma_l)
    (\Omega_{l,j}+\Sigma_{l,j})}{\sum_{n=1}^{5}(\Omega_n+\Sigma_n)^2}\Big]\\
    &-\sum_{k}\Sigma_{k,k}+\frac{\sum_{l=1}^5\sum_{j=1}^{5}(\Omega_l+\Sigma_l)(\Omega_j+\Sigma_l)
    (\Omega_{l,j}+\Sigma_{l,j})}{\sum_{n=1}^{5}(\Omega_n+\Sigma_n)^2}-\frac{\sum_{l=1}^5\sum_{j=1}^{5}\Omega_l \Omega_j
    \Omega_{l,j}}{\sum_{n=1}^{5}\Omega_n^2}.
\end{split}
\end{align}

Next we define the functions in \eqref{def:tildeN} and \eqref{def:widVPi}.

In what follows we use the following notations: $\nabla_{\omega}^{\perp}g, $ for any function $g$, is defined as
\begin{align}\label{eq:nabalPerp}
    \nabla_{\omega}^{\perp}g(\omega)= P_{\perp \omega}
    \nabla_{\omega} g(\omega);
\end{align} and the operator $P_{\perp\omega}:\mathbb{R}^4\rightarrow \mathbb{R}^4$ is an orthogonal projection such that for any vector $A\in \mathbb{R}^4,$
\begin{align}
P_{\perp \omega}A=A-(\omega \cdot A)\omega.\label{def:proOmega}
\end{align}

For the functions $\Omega_k$ and $\Sigma_k$, $k=1,\cdots,5,$
\begin{align}
\begin{split}\label{def:Om1Si1}
    \Omega_1:=&-\partial_{z}u,\\
    \Sigma_1:=&\frac{L}{1+D },
    \end{split}
\end{align} and when $k= 2,3,4,5$,
\begin{align}
\begin{split}\label{def:OmkSik}
    \Omega_k:=&\omega\cdot e_{k-1}-|{\bf{\tilde{x}}}|^{-1} \nabla_{\omega}^{\perp}u\cdot  e_{k-1},\\
    \Sigma_k:=&\frac{E    \cdot e_{k-1}}{|{\bf{\tilde{x}}}|(1+D )}.
\end{split}
\end{align} Here the functions $L$ and $D$ are defined as
\begin{align}
\begin{split}\label{def:LD}
L:=&\Big[-\partial_{z}u\Big(\partial_z u(\partial_{z}\Pi\cdot \omega)+(\partial_{z}^2\Pi\cdot \omega)u\Big)-\omega\cdot \partial_{z}\Pi\Big]\\
&+|{\bf{\tilde{x}}}|^{-1} \Big[\partial_{z}u \Big(u \|\text{P}_{\perp \omega} \partial_{z}\Pi\|^2+\partial_{z}\Pi\cdot \omega(
\nabla_{\omega}^{\perp}u \cdot \partial_{z}\Pi)\Big)+  \nabla_{\omega}^{\perp}u\cdot \partial_{z}\Pi\Big],\\
D:=&\frac{u}{|\bf{\tilde{x}}|}\|\text{P}_{\perp \omega} \partial_{z}\Pi\|^2-\partial_z u(\partial_{z}\Pi\cdot \omega)-(\partial_{z}^2\Pi\cdot \omega)u+\frac{\partial_{z}\Pi\cdot \omega}{|\bf{\tilde{x}}|}(\nabla_{\omega}^{\perp}u \cdot \partial_{z}\Pi);
\end{split}
\end{align}
$E=P_{\perp \omega}E\in \mathbb{R}^4$ is a vector orthogonal to $\omega$, and is defined as
\begin{align}\label{def:E}
\begin{split}
E:=&\Big[-\Big(\partial_{z}\Pi\cdot \omega+\partial_{z}u\Big)\Big(u \text{P}_{\perp \omega} \partial_{z}\Pi +(\partial_{z}\Pi \cdot \omega)  \nabla_{\omega}^{\perp}u\Big)\Big]\\
&+|{\bf{\tilde{x}}}|^{-1} \Big[(\nabla_{\omega}^{\perp}u\cdot \partial_{z}\Pi)\Big(u \text{P}_{\perp \omega} \partial_{z}\Pi +(\partial_{z}\Pi \cdot \omega) \nabla_{\omega}^{\perp}u\Big)\Big].
\end{split}
\end{align}

Next we define $\Omega_{l,j}$ and $\Sigma_{l,j}$, $l,j=1,2,3,4,5.$ Thy are symmetric of the indices $$\Omega_{l,j}=\Omega_{j,l}\ \text{and}\ \Sigma_{l,j}=\Sigma_{j,l},$$ 
and are defined as
\begin{align}
\begin{split}\label{def:OmSi11}
    \Omega_{11}:=&-\partial_{z}^2 u,\\
\Sigma_{11}:=&\frac{1}{1+D}\Big(\partial_{z}^2u\ D+\partial_{z}\frac{L}{1+D}\Big)-\frac{1}{(1+D)|{\bf{\tilde{x}}}|}
 \partial_{z}\Pi\cdot \Big(-\nabla_{\omega}^{\perp}\partial_{z}u+\nabla_{\omega}^{\perp}\frac{L}{1+D}\Big)\\
&-\frac{\omega\cdot\partial_{z}\Pi}{1+D}\partial_{|{\bf{\tilde{x}}}|}\frac{L}{1+D},
\end{split}
\end{align}
and when $k=2,3,4,5,$
\begin{align}
\begin{split}\label{def:OmSi1k}
    \Omega_{1k}:=&-|{\bf{\tilde{x}}}|^{-1} e_{k-1}\cdot\nabla_{\omega}^{\perp}\partial_{z}u,\\
\Sigma_{1,k}:=&\frac{(e_{k-1}\cdot\nabla_{\omega}^{\perp}\partial_{z}u)\ D}{(1+D)|{\bf{\tilde{x}}}|}
+\frac{1}{|{\bf{\tilde{x}}}|(1+D)}\partial_{z} \frac{E\cdot e_{k-1}}{|{\bf{\tilde{x}}}|(1+D)} \\
&-\frac{1}{(1+D)|{\bf{\tilde{x}}}|} \partial_{z}\Pi\cdot \Big[\text{P}_{\perp \omega} e_{k-1}- |{\bf{\tilde{x}}}|^{-1}\nabla_{\omega}^{\perp}\big( e_{k-1}\cdot \nabla_{\omega}^{\perp}u\big)+\nabla_{\omega}^{\perp}\frac{E\cdot e_{k-1}}{|{\bf{\tilde{x}}}|(1+D)}\Big]\\
&-\frac{\omega\cdot\partial_{z}\Pi}{1+D}\Big(|{\bf{\tilde{x}}}|^{-2}  e_{k-1}\cdot \nabla_{\omega}^{\perp}u+\partial_{|{\bf{\tilde{x}}}|} \frac{E\cdot e_{k-1}}{|{\bf{\tilde{x}}}|(1+D)}\Big);
\end{split}
\end{align}
and when $k,l=2,3,4,5,$
\begin{align}
\begin{split}\label{def:OmSikl}
    \Omega_{lk}:=&|{\bf{\tilde{x}}}|^{-1}\text{P}_{\perp \omega}e_{l-1}\cdot e_{k-1}-|{\bf{\tilde{x}}}|^{-2}
    e_{l-1} \cdot \nabla_{\omega}^{\perp}(\nabla_{\omega}^{\perp}u\cdot  e_{k-1})+|{\bf{\tilde{x}}}|^{-2}(\omega\cdot e_{l-1}) (\nabla_{\omega}^{\perp}u\cdot e_{k-1}),\\
\Sigma_{l,k}
:=&\frac{G\cdot e_{l-1}}{|{\bf{\tilde{x}}}|(1+D)}\Big[\frac{1}{|{\bf{\tilde{x}}}|}\partial_{z}\frac{E\cdot e_{k-1}}{1+D}-
\frac{1}{|{\bf{\tilde{x}}}|^2}\partial_{z}\Pi \cdot 
\nabla_{\omega}^{\perp} \frac{E\cdot e_{k-1}}{1+D}\Big]\\
&-\frac{G\cdot e_{l-1}}{|{\bf{\tilde{x}}}|(1+D)}\Big[\frac{1}{|{\bf{\tilde{x}}}|}\Big(\text{P}_{\perp \omega}\partial_{z}\Pi +\nabla_{\omega}^{\perp}\partial_{z}u\Big)+(\omega\cdot \partial_{z}\Pi)\Big(\frac{\nabla_{\omega}^{\perp}u}{|{\bf{\tilde{x}}}|^2}-\partial_{|{\bf{\tilde{x}}}|}\frac{E}{|{\bf{\tilde{x}}}|(1+D)}\Big)\Big]\cdot e_{k-1}\\
&+\frac{e_{l-1}}{|{\bf{\tilde{x}}}|} \cdot \nabla_{\omega}^{\perp}\frac{E\cdot e_{k-1}}{|{\bf{\tilde{x}}}|(1+D)}+\omega\cdot e_{l-1}\partial_{|{\bf{\tilde{x}}}|}\frac{E\cdot e_{k-1}}{|{\bf{\tilde{x}}}|(1+D)},
\end{split}
\end{align} and
$G=P_{\perp\omega}G$ is a four-dimensional vector-valued function defined as
\begin{align}\label{def:Gvec}
G:=
&u \text{P}_{\perp \omega} \partial_{z}\Pi +(\partial_{z}\Pi \cdot \omega)  \nabla_{\omega}^{\perp}u .
\end{align}
And lastly $Q_1$ and $Q_2$ are functions defined as
\begin{align}
\begin{split}\label{def:Q1Q2}
Q_1:=&\frac{(\partial_{z}\Pi\cdot \omega)^2+\partial_{z}u (\partial_{z}\Pi\cdot \omega)-|{\bf{\tilde{x}}}|^{-1}( \partial_{z}\Pi\cdot \nabla_{\omega}^{\perp}u) \partial_{z}\Pi\cdot \omega}{1+D},\\
Q_2:=&-\omega\cdot \partial_{t}\Pi+\frac{1}{|\bf{\tilde{x}}|}\partial_{t}\Pi\cdot \nabla_{\omega}^{\perp}u\\
&+\frac{-\omega\cdot \partial_{z}\Pi-\partial_{z}u+|{\bf{\tilde{x}}}|^{-1}  \partial_{z}\Pi\cdot\nabla_{\omega}^{\perp}u}{1+D}\times\\
&\Big(-|{\bf{\tilde{x}}}|^{-1}
\big((\partial_{z}\Pi\cdot\omega)\nabla_{\omega}^{\perp}u+u\partial_{z}\Pi\big)\cdot \text{P}_{\perp \omega}\partial_{t}\Pi+u \partial_{z}\partial_{t}\Pi\cdot \omega\Big).
\end{split}
\end{align}

\begin{remark}\label{rem:source}
In the main part of paper, among the many terms in $V_{\Pi}$, we only study a few of them. In what follows we list these terms and their sources, corresponding to those in \eqref{def:J}, \eqref{def:J7J8} and \eqref{def:J9J10}.
Recall that $V_{\Pi}=\widetilde{V_{\Pi}}\Big|_{|{\bf{\tilde{x}}}|=u}$.

\begin{itemize}
    \item[(1)] 
$v^{-1}(\partial_{z}\Pi\cdot \omega)^2$ from $Q_1\sum_{k=2}^5 \Omega_{k,k};$ 
\item[(2)] $\omega\cdot \partial_{t}\Pi$ from $Q_2;$
\item[(3)]
some terms from $\Sigma_{1,1}$: $u(\partial_{z}^2 \Pi\cdot \omega)\partial_{z}^2 u $ from 
$
\frac{(\partial_{z}^2 u) D}{1+D};
$
$(\partial_{z}^2u) |\partial_{z}\Pi\cdot \omega|^2$ and $\omega\cdot \partial_{z}^2\Pi$ from 
$
\frac{\partial_{z}L}{(1+D)^2};
$
and $u^{-2}(\omega\cdot \partial_{z}\Pi) (\nabla_{\omega}^{\perp}u\cdot \partial_{z}\Pi)$ from
$
-\frac{1}{(1+D)^2}(\omega\cdot \partial_{z}\Pi) \partial_{|{\bf{\tilde{x}}}|}L.
$
\item[(4)]
some terms from $\sum_{k=2}^5\Sigma_{k,k}$: 
$ u^{-1}\partial_{z}\Pi\cdot \partial_{z}\nabla_{\omega}^{\perp}u$ from $\sum_{k=2}^{5} \frac{G\cdot e_{l-1}}{|{\bf{\tilde{x}}}|^2}  \nabla_{\omega}^{\perp}\partial_{z}u\cdot e_{l-1}$;
 $u^{-2} \partial_{z}\Pi\cdot \nabla_{\omega}^{\perp} (\partial_{z}\Pi \cdot \nabla_{\omega}^{\perp}u) $ and $u^{-1} |\text{P}_{\perp \omega}\partial_{z}\Pi|^2
$ from $\sum_{l}\frac{e_{l-1}}{|{\bf{\tilde{x}}}|} \cdot \nabla_{\omega}^{\perp} \frac{E\cdot e_{l-1}}{|{\bf{\tilde{x}}}|}$;
and $u^{-2}(\partial_{z}\Pi\cdot \omega) (\partial_{z}\Pi\cdot \nabla_{\omega}^{\perp} u)$ from
$
    \sum_{k=2}^5|{\bf{\tilde{x}}}|^{-2} G\cdot e_{k-1} P_{\perp\omega}\partial_{z}\Pi \cdot e_{k-1};
$
and $u^{-1}\partial_{z}u (\partial_{z}\Pi\cdot \omega)$ from
$|{\bf{\tilde{x}}}|^{-2}\sum_{k }e_{k-1} \cdot \nabla_{\omega}^{\perp} (E\cdot e_{k-1}).$

\end{itemize}
\end{remark}

In the rest of this section we will prove the second identity in \eqref{eq:tildeN} in subsection \ref{sub:tildeN} below, and then derive the equation \eqref{eq:effectivUeqn} in subsection \ref{sub:effectivUeqn} below.

%%%%%%%%%%%%%%%%%%%%%%%%%%%%%%%%%%%%%%%%%%%%%%%%%%%%%%%%%%%%%%%%
\subsection{Proof of the second identity in (\ref{eq:tildeN})}\label{sub:tildeN}
The following identity will be needed: $$\sum_{l=2}^5 (A_1\cdot e_{l-1})(A_2\cdot e_{l-1})=A_1\cdot A_2 $$ for any vectors $A_1$ and $A_2$. Thus $$\sum_{l=2}^5 (P_{\perp\omega}A_1\cdot e_{l-1})(e_{l-1}\cdot \omega)=0.$$

Compute directly to find
\begin{align}
\sum_{l=2}^5 \Omega_{l}\Omega_{l,k}
=&|{\bf{\tilde{x}}}|^{-3}\nabla_{\omega}^{\perp}u\cdot 
 \Big(\nabla_{\omega}^{\perp}(\nabla_{\omega}^{\perp}u\cdot  e_{k-1})\Big)
\end{align}
and 
\begin{align}
\sum_{k=2}^5 \Omega_k \sum_{l=2}^5 \Omega_l \Omega_{l,k}=|{\bf{\tilde{x}}}|^{-3}\sum_{k=2}^5  (\omega\cdot e_{k-1}-|{\bf{\tilde{x}}}|^{-1} \nabla_{\omega}^{\perp}u\cdot  e_{k-1})\Big(\nabla_{\omega}^{\perp}u\cdot \nabla_{\omega}^{\perp}(\nabla_{\omega}^{\perp}u\cdot 
 e_{k-1})\Big),
\end{align}
where we use the following identity 
\begin{align}
&\sum_{k=2}^5  (\omega\cdot e_{k-1})\Big(\nabla_{\omega}^{\perp}u\cdot 
 \nabla_{\omega}^{\perp}(\nabla_{\omega}^{\perp}u\cdot 
  e_{k-1})\Big)\nonumber\\
=&\nabla_{\omega}^{\perp}u\cdot 
  \nabla_{\omega}^{\perp}\sum_{k=2}^5 \Big((\omega\cdot e_{k-1})(\nabla_{\omega}^{\perp}u\cdot e_{k-1}) \Big)-\sum_{k=2}^5 (\nabla_{\omega}^{\perp}u\cdot  e_{k-1})\Big(\nabla_{\omega}^{\perp}u\cdot \nabla_{\omega}^{\perp}(\omega\cdot e_{k-1})\Big)\nonumber\\
=&-|\nabla_{\omega}^{\perp}u|^2,
\end{align} where in the first step we observe the first term is identically zero.

Similarly

\begin{align*}
\Omega_{11}+\sum_{k=2}^5 \Omega_{k,k}=&-\partial_{z}^2 u+3|{\bf{\tilde{x}}}|^{-1}-|{\bf{\tilde{x}}}|^{-2}\Delta_{\mathbb{S}^3}u,\\
\Omega_1\sum_{k= 2}^5 \Omega_k  \Omega_{k,1}=&-\partial_{z}u \Big\{-\sum_{k=2}^5 \Big(\omega\cdot e_{k-1}-|{\bf{\tilde{x}}}|^{-1} \nabla_{\omega}^{\perp}u\cdot  e_{k-1}\Big)|{\bf{\tilde{x}}}|^{-1} e_{k-1}\cdot\nabla_{\omega}^{\perp}\partial_{z}u\Big\}\\
=&-\partial_{z} u\ 
 \Big(|{\bf{\tilde{x}}}|^{-2} \nabla_{\omega}^{\perp} u\cdot \nabla_{\omega}^{\perp} \partial_{z}u\Big),\\
\Omega_1^2+\sum_{k=2}^5  \Omega_k^2=&(\partial_z u)^2+\sum_{k=2}^5 \Big(\omega\cdot e_{k-1}-|{\bf{\tilde{x}}}|^{-1} \nabla_{\omega}^{\perp}u\cdot e_{k-1}\Big)^2\\
=&1+(\partial_z u)^2+|{\bf{\tilde{x}}}|^{-2} |\nabla_{\omega}^{\perp}u|^2.
\end{align*}

What is left is to feed these into the definition of $\tilde{N}$ in \eqref{def:tildeN}, and then let $u=|{\bf{\tilde{x}}}|$ to obtain the desired result.

In the rest of this section we derive (\ref{eq:effectivUeqn}).

%%%%%%%%%%%%%%%%%%%%%%%%%%%%%%%%%%%%%%%%%%%%%%%%%%%%%%%%%%%%%
\subsection{Derivation of (\ref{eq:effectivUeqn})}\label{sub:effectivUeqn}
Here we have a risk of confusions: recall that $z$, $\omega$ and $u(z,\omega,t)$ depends $t$, it is easy to confuse $\frac{\partial}{\partial_t} u\big(z(t),\omega(t),t\big)$ with $\frac{\partial}{\partial_t} u\big(z(s),\omega(s),t\big)\Big|_{s=t}$. To avoid confusion
we use the following conventions:
\begin{align}
\begin{split}\label{eq:convention}
  \frac{\partial}{\partial t}u:=&\partial_{z}u \partial_{t}z+\nabla_{\omega}^{\perp} u\cdot \partial_{t}\omega+u_t,\\
  u_t:=&\frac{\partial}{\partial_t} u(z(s),\omega(s),t)\Big|_{s=t}.
  \end{split}
\end{align}
The desired \eqref{eq:effectivUeqn} is equivalent to the following identity:
\begin{align}\label{eq:omegasigma}
    -(1+Q_1)u_t+Q_2=\sum_{k=1}^{5} (\Omega_{k,k}+\Sigma_{k,k})-\frac{\sum_{l=1}^5\sum_{j=1}^{5}(\Omega_l+\Sigma_l)(\Omega_j+\Sigma_j)
    (\Omega_{l,j}+\Sigma_{l,j})}{\sum_{n=1}^{5}(\Omega_n+\Sigma_n)^2}.
\end{align}
This will be implied by \eqref{eq:level}, \eqref{eq:tf}-\eqref{eq:klf} below.

In order to derive a governing equation for the function $u$ we define a new variable $q$ by 
\begin{align}\label{eq:defQ}
q:=z-u(z,\omega,t) \partial_{z} \Pi(z, t) \cdot \omega,
\end{align}
then the parametrization becomes, in the independent variables $q$ and $x_k,\ k=2,3,4,5$, 
\begin{align}\label{eq:NewParaQ}
\Phi_{\Pi,u}= \left[
\begin{array}{ccc}
q\\
\Pi(z(q,\omega,t),t)
\end{array}
\right]+\left[
\begin{array}{ccc}
0\\
u(z(q,\omega,t), \omega,t)\omega
\end{array}
\right],
\end{align}
and $\omega\in \mathbb{S}^3$ becomes
\begin{align}
\omega:=\frac{{\bf{\tilde{x}}}}{|{\bf{\tilde{x}}}|}:=&\frac{(\tilde{x}_2,\cdots,\tilde{x}_5)^{T}}{|(\tilde{x}_2,\cdots,\tilde{x}_5)|}
:=\frac{(x_2,\cdots,x_5)^{T}-\Pi(z,t)}{|(x_2,\cdots,x_5)^{T}-\Pi(z,t)|}\label{def:omega}
\end{align} and $\tilde{x}_k,\ k=2,3,4,5,$ are defined as
\begin{align}
\tilde{x}_k:=x_k-\Pi_k(z(q,\omega,t),t).\label{eq:tildxk}
\end{align} and $\Pi_k$ is the corresponding entry of the vector $\Pi$.

We emphasize that $\omega$ and $z$ are functions of the variables $q, \ x_{k}$ and $t$. 

To derive a governing equation for $u$ we consider the level set
\begin{align}
f(q,x_2,x_3,x_4,x_5,t):=|{\bf{\tilde{x}}}|-u(z(q,\omega,t), \omega,t)=0.
\end{align}
It was shown in \cite{MR1770903} the function $f$ satisfies the equation, with convention $q=x_1$,
\begin{align}
\partial_{t}f=\sum_{i,j}\Big(\delta_{ij}-\frac{f_{x_i}f_{x_j}}{|Df|^2}\Big)f_{x_ix_j}.\label{eq:level}
\end{align}

To prove the desired (\ref{eq:omegasigma}), we only need to prove the following identities:
\begin{lemma}
\begin{align}
\partial_t f=&-(1+Q_1)u_t+Q_2,\label{eq:tf}\\
f_{x_k}=&\Omega_k+\Sigma_k,\label{eq:kf}\\
f_{x_k, x_l}=&\Omega_{k,l}+\Sigma_{k,l}.\label{eq:klf}
\end{align}
\end{lemma}
These will be proved in subsequent subsections.

To facilitate later discussions, we derive some identities.
Take a $\partial_q$ on \eqref{eq:defQ} and \eqref{eq:tildxk} to find
\begin{align}
1=&\partial_q z\Big(1-\partial_z u(\partial_{z}\Pi\cdot \omega)-u(\partial_{z}^2\Pi\cdot \omega)\Big)-\sum_{l=2}^5\Big(\big((\partial_{z}\Pi\cdot \omega)\nabla_{\omega}^{\perp}u+u\partial_{z}\Pi\big)\cdot \partial_{\tilde{x}_l} \omega\Big)\partial_q \tilde{x}_l,\\
{\partial_q \tilde{x}_l}=&-\partial_{z}\Pi_l\ {\partial_q}z.\label{eq:qxl}
\end{align} Recall the definitions of $\nabla_{\omega}^{\perp}$ and $\text{P}_{\perp \omega}$ from \eqref{eq:nabalPerp} and \eqref{def:proOmega}.
From these two equations we solve for $\partial_q z$. This implies the following identities
\begin{align}
\partial_{\tilde{x}_l} \omega=|{\bf{\tilde{x}}}|^{-1}\text{P}_{\perp \omega}e_{l-1}\ \text{and}\ \sum_{l=2}^5 e_{l-1} \partial_{z}\Pi_l=\partial_{z}\Pi
\end{align}
and hence
\begin{align}
\begin{split}
    \sum_{l=2}^5\big(\partial_{z}\Pi\cdot \partial_{\tilde{x_l}} \omega\big)\partial_{z}\Pi_l=&\frac{1}{|\bf{\tilde{x}}|}\|\text{P}_{\perp \omega} \partial_{z}\Pi\|^2,\\
    \sum_{l=2}^5 (\nabla_{\omega}^{\perp}u \cdot \partial_{\tilde{x}_l} \omega )\partial_{z}\Pi_l=&\frac{1}{|\bf{\tilde{x}}|}(\nabla_{\omega}^{\perp}u \cdot \text{P}_{\perp \omega}\partial_{z}\Pi).
\end{split}
\end{align}
Together they imply the following results, recall the definition of $D$ from \eqref{def:LD},
 \begin{lemma}
 \begin{align}\label{eq:D1D2}
\begin{split}
\partial_q z=&\frac{1}{1+D}\\
\partial_q \tilde{x}_l =& -\frac{\partial_{z}\Pi_l}{1+D}\\
\partial_{q}|{\bf{\tilde{x}}}|=&\sum_{l=2}^5 \omega\cdot \frac{\partial {\bf{\tilde{x}}}}{\partial \tilde{x}_l} \frac{\partial \tilde{x}_l}{\partial q} 
=\sum_{l=2}^5 (\omega\cdot e_{l-1}) \frac{\partial \tilde{x}_l}{\partial q}=-\frac{ \omega \cdot \partial_{z}\Pi}{1+D},\\
\partial_{q}\omega=&-\frac{\partial_{q}z}{|{\bf{\tilde{x}}}|} \text{P}_{\perp \omega}\partial_{z}\Pi=-\frac{1}{(1+D)|{\bf{\tilde{x}}}|}\text{P}_{\perp \omega}\partial_{z}\Pi.
\end{split}
\end{align}
\end{lemma}

Next we compute the $x_k$-derivatives. 
Take a $x_k$-derive on \eqref{eq:defQ} and \eqref{eq:tildxk} to obtain
\begin{align}
\Big(1-\partial_{z}u (\partial_{z}\Pi \cdot \omega)- (\partial_{z}^2 \Pi\cdot \omega)u\Big)\partial_{x_k} z=& \sum_{j=2}^5 \Big[\big( (\partial_{z}\Pi \cdot \omega)\nabla_{\omega}^{\perp}u +u\ \partial_{z}\Pi\big)\cdot \partial_{\tilde{x}_j}\omega\Big]\partial_{x_k} \tilde{x}_j,\label{eq:zxk}\\
\partial_{x_k} \tilde{x}_j=&\delta_{j,k}-\partial_{z}\Pi_j \ \partial_{x_k} z.
\end{align} 
From these two equations we solve for $ \partial_{x_k} z$ and $\partial_{x_k} \tilde{x}_j$, and then derive more identities,
\begin{lemma}
\begin{align}
 \partial_{x_k} z=&\frac{G\cdot e_{k-1}}{|{\bf{\tilde{x}}}|(1+D)},\label{eq:parxkz}\\
\partial_{x_k}|\bf{\tilde{x}}|=&\omega \cdot e_{k-1}- \frac{G\cdot e_{k-1}}{|{\bf{\tilde{x}}}|(1+D)} \omega \cdot \partial_{z}\Pi,\label{eq:parxktilx}\\
\partial_{x_k}\omega=&|{\bf{\tilde{x}}}|^{-1} \text{P}_{\perp \omega} \big[e_{k-1}-\frac{G\cdot e_{k-1}}{|{\bf{\tilde{x}}}|(1+D)}\ \partial_{z}\Pi\big].\label{xkome}
\end{align}
\end{lemma}
Here $G=\text{P}_{\perp \omega}G\in \mathbb{R}^4$ is a vector-valued function defined in \eqref{def:Gvec}.

\subsection{Proof of (\ref{eq:kf})}

Next we compute $\partial_{x_k}u(z,\omega,t)$, $k\geq 2$. By the chain rule,
\begin{align*}
\partial_{x_k}u= &\Big[\partial_{x_k} z\partial_{z}+\partial_{x_k}\omega \cdot \nabla_{\omega}^{\perp}\Big]u
=|{\bf{\tilde{x}}}|^{-1} \nabla_{\omega}^{\perp}u\cdot  e_{k-1}+\frac{\partial_{z}u-|{\bf{\tilde{x}}}|^{-1}\nabla_{\omega}^{\perp}u\cdot \partial_{z}\Pi}{|{\bf{\tilde{x}}}|(1+D)} G\cdot e_{k-1}.
\end{align*}
This together with the expression for $\partial_{x_k}|{\bf{\tilde{x}}}|$ in \eqref{eq:parxktilx} implies that
\begin{align}\label{eq:kxu}
\partial_{x_k}\big(|{\bf{\tilde{x}}}|-u\big)=&\omega\cdot e_{k-1}-|{\bf{\tilde{x}}}|^{-1} \nabla_{\omega}^{\perp}u\cdot  e_{k-1}-\frac{(\partial_{z}\Pi\cdot \omega+\partial_{z}u)}{|{\bf{\tilde{x}}}|(1+D)} G\cdot e_{k-1}\nonumber+\frac{(\nabla_{\omega}^{\perp}u\cdot \partial_{z}\Pi)}{|{\bf{\tilde{x}}}|^2(1+D)} G\cdot e_{k-1}\nonumber\\
=&\Omega_k+\Sigma_k
\end{align} where the terms $\Omega_k$ and $\Sigma_k$ are defined in \eqref{def:OmkSik}.

Similarly, 
\begin{align*}
\partial_{q}u= &[\partial_{q} z\partial_{z}+\partial_{q}\omega\cdot \nabla_{\omega}^{\perp}]u
=\partial_{z}u-\frac{D\partial_{z}u+|{\bf{\tilde{x}}}|^{-1}\nabla_{\omega}^{\perp}u\cdot \partial_{z}\Pi}{1+D}
\end{align*}
Thus this together with the expression for $\partial_{q}|{\bf{\tilde{x}}}|$ in \eqref{eq:D1D2} implies that
\begin{align}
\partial_{q}(|{\tilde{\bf{x}}}|-u)=\Omega_1+\Sigma_1
\end{align}
where the terms $\Omega_1$ and $\Sigma_1$ are defined in \eqref{def:Om1Si1}.

\subsection{Proof of (\ref{eq:tf})}
Here we need the conventions made in \eqref{eq:convention}.

We start with considering $\partial_{t}{\bf{\tilde{x}}}.$ Compute directly from \eqref{eq:tildxk} to obtain
\begin{align}
\partial_{t}{\bf{\tilde{x}}}=&-\partial_{z}\Pi\ \partial_t z-\partial_{t}\Pi,
\end{align}
 and from the identity for $\omega$ after \eqref{eq:NewParaQ},
\begin{align}\label{eq:tomega}
\partial_{t}\omega=\frac{\partial_{t}{\bf{\tilde{x}}}}{|{\bf{\tilde{x}}}|}-{\bf{\tilde{x}}}\frac{\bf{\tilde{x}}\cdot \partial_{t}\bf{\tilde{x}}}{|{\bf{\tilde{x}}}|^3}=\frac{1}{|{\bf{\tilde{x}}}|} \text{P}_{\perp \omega}\partial_{t}{\bf{\tilde{x}}}=-\frac{1}{|{\bf{\tilde{x}}}|} \text{P}_{\perp \omega} (\partial_{z}\Pi\ \partial_t z+\partial_{t}\Pi).
\end{align}
From \eqref{eq:defQ}, we derive
\begin{align}
\partial_t z=&\Big( u_t+\partial_{z}u \partial_{t}z+\nabla_{\omega}^{\perp} u\cdot \partial_{t}\omega\Big) \partial_{z}\Pi\cdot\omega+u \partial_{z}^2 \Pi\cdot \omega \partial_t z+u \partial_{z}\partial_t \Pi\cdot \omega+u\partial_{z}\Pi \cdot \partial_{t}\omega\nonumber\\
=&(\partial_{z}\Pi\cdot\omega) u_t+\Big(\partial_{z}u(\partial_{z}\Pi\cdot\omega)+u \partial_{z}^2 \Pi\cdot \omega \Big)\partial_{t}z+\Big((\partial_{z}\Pi\cdot\omega)\nabla_{\omega}^{\perp}u+u\partial_{z}\Pi\Big)\cdot \partial_t\omega+u \partial_{z}\partial_t \Pi\cdot \omega\nonumber\\
=&(\partial_{z}\Pi\cdot\omega)\  u_t+B-D\ \partial_{t}z. \label{eq:zt}
\end{align}
where the term $B$ is defined as
\begin{align*}
B:=&-|{\bf{\tilde{x}}}|^{-1}
\big((\partial_{z}\Pi\cdot\omega)\nabla_{\omega}^{\perp}u+u\partial_{z}\Pi\big)\cdot \text{P}_{\perp \omega}\partial_t \Pi+u \partial_{z}\partial_t \Pi\cdot \omega.
\end{align*}
Solve for $\partial_{t}z$ to find
\begin{align}
\partial_{t}z=&u_t\frac{\partial_{z}\Pi\cdot\omega}{1+D}+\frac{B}{1+D}.
\end{align} Feed this to \eqref{eq:tomega} to obtain,
\begin{align}
\partial_{t}\omega=-\frac{1}{|{\bf{\tilde{x}}}|}(\partial_{t}u\frac{\partial_{z}\Pi\cdot\omega}{1+D}+\frac{B}{1+D}) \text{P}_{\perp \omega} \partial_{z}\Pi -\frac{1}{|{\bf{\tilde{x}}}|}\text{P}_{\perp \omega}\partial_{t}\Pi.
\end{align}
We are ready to compute $\frac{\partial}{\partial t}u$ and $\partial_{t} |{\bf{\tilde{x}}}|$
\begin{align*}
\frac{\partial}{\partial t}u=&\partial_{t}z\ \partial_{z}u+\partial_{t}\omega\cdot \nabla_{\omega}^{\perp}u+u_t\\
=&\Big(1+\partial_{z}u\frac{\partial_{z}\Pi\cdot \omega}{1+D}-\frac{1}{|\bf{\tilde{x}}|}\frac{(P_{\perp\omega}\partial_{z}\Pi\cdot \nabla_{\omega}^{\perp}u) \partial_{z}\Pi\cdot \omega}{1+D}\Big) u_t\\
&+\frac{B}{1+D}\partial_{z}u-\frac{1}{|\bf{\tilde{x}}|}\frac{\text{P}_{\perp \omega}\partial_{z}\Pi\cdot\nabla_{\omega}^{\perp}u}{1+D}B-\frac{1}{|\bf{\tilde{x}}|}\text{P}_{\perp \omega}\partial_{t}\Pi\cdot \nabla_{\omega}^{\perp}u,
\end{align*} and
\begin{align*}
\partial_t |{\bf{\tilde{x}}}|=&\omega\cdot \partial_{t}{\bf{\tilde{x}}}\\
=&-\frac{(\partial_{z}\Pi\cdot \omega)^2}{1+D}u_t-\omega\cdot \partial_{z}\Pi\frac{B}{1+D} -\omega\cdot \partial_{t}\Pi.
\end{align*}
Put things together to obtain
\begin{align}
\frac{\partial}{\partial t}(|{\bf{\tilde{x}}}|-u)=-\Big(1+Q_1\Big)\ u_t+Q_2
\end{align} where $Q_1$ and $Q_2$ are defined in \eqref{def:Q1Q2}.

%%%%%%%%%%%%%%%%%%%%%%%%%%%%%%%%%%%%%%%%%%

\subsection{Proof of (\ref{eq:klf})}
We start with computing $\partial_{x_l}\partial_{x_k}\big(|{\bf{\tilde{x}}}|-u\big)$, $k,l\geq 2.$ Recall that we obtained in \eqref{eq:kxu} that $$\partial_{x_k}\big(|{\bf{\tilde{x}}}|-u\big)=\omega\cdot e_{k-1}-|{\bf{\tilde{x}}}|^{-1} \nabla_{\omega}^{\perp}u\cdot  e_{k-1}+\frac{E\cdot e_{k-1}}{|{\bf{\tilde{x}}}|(1+D)},$$ which is a function of $z,$ $\omega, $ and $|{\bf{\tilde{x}}}|$.
Then by the chain rule,
\begin{align}
\partial_{x_l}\partial_{x_k}\big(|{\bf{\tilde{x}}}|-u\big)
=&\partial_{x_l}\Big[\omega\cdot e_{k-1}-|{\bf{\tilde{x}}}|^{-1} \nabla_{\omega}^{\perp}u\cdot  e_{k-1}+\frac{E\cdot e_{k-1}}{|{\bf{\tilde{x}}}|(1+D)}\Big]\nonumber\\
=&\Big[\partial_{x_l}\omega \cdot \nabla_{\omega}^{\perp}+\partial_{x_l}z\ \partial_{z}+\partial_{x_l}|{\bf{\tilde{x}}}|\ \partial_{|{\bf{\tilde{x}}}|}\Big] \Big[\omega\cdot e_{k-1}-|{\bf{\tilde{x}}}|^{-1} \nabla_{\omega}^{\perp}u\cdot  e_{k-1}+\frac{E\cdot e_{k-1}}{|{\bf{\tilde{x}}}|(1+D)}\Big]\nonumber\\
=&\Omega_{l,k}+\Sigma_{l,k}
\end{align}
where, recall the identities for $\partial_{x_l}z$, $\partial_{x_l}|{\bf{\tilde{x}}}|$ and $\partial_{x_l}\omega$ in \eqref{eq:parxkz}, \eqref{eq:parxktilx} and \eqref{xkome},
and
$\Omega_{l,k}$ and $\Sigma_{l,k}$, $l,k\geq 2$ are defined in \eqref{def:OmSikl}.

Similarly for $\partial_{q}^2\big(|{\bf{\tilde{x}}}|-u\big),$
\begin{align}
\begin{split}
\partial_{q}^2\big(|{\bf{\tilde{x}}}|-u\big)=&\partial_{q}(-\partial_{z}u+\frac{L}{1+D})\\
=&\Big(\partial_{q}z\partial_{z}+\partial_{q}\omega\cdot \nabla_{\omega}^{\perp}+\partial_{q}|{\bf{\tilde{x}}}|\partial_{|{\bf{\tilde{x}}}|}\Big)\Big(-\partial_{z}u+\frac{L}{1+D}\Big)\\
=&\Omega_{1,1}+\Sigma_{1,1}
\end{split}
\end{align}where $\Omega_{1,1}$ and $\Sigma_{1,1}$ are defined in \eqref{def:OmSi11}.

And for $\partial_{q}\partial_{x_k}(|{\bf{\tilde{x}}}|-u)$, $k=2,3,4,5,$
\begin{align}
\partial_{q}\partial_{x_k}\big(|{\bf{\tilde{x}}}|-u\big)=&\Big(\partial_{q}z\partial_{z}+\partial_{q}\omega\cdot \nabla_{\omega}^{\perp}+\partial_{q}|{\bf{\tilde{x}}}|\partial_{|{\bf{\tilde{x}}}|}\Big)\Big(\partial_{x_k}z\partial_{z}+\partial_{x_k}\omega\cdot \nabla_{\omega}^{\perp}+\partial_{x_k}|{\bf{\tilde{x}}}|\partial_{|{\bf{\tilde{x}}}|} \Big)\big(|{\bf{\tilde{x}}}|-u\big)\nonumber\\
=&\Big(\partial_{q}z\partial_{z}+\partial_{q}\omega\cdot \nabla_{\omega}^{\perp}+\partial_{q}|{\bf{\tilde{x}}}|\partial_{|{\bf{\tilde{x}}}|}\Big)\Big(\omega\cdot e_{k-1}-|{\bf{\tilde{x}}}|^{-1} e_{k-1}\cdot \nabla_{\omega}^{\perp}u+\frac{E\cdot e_{k-1}}{|{\bf{\tilde{x}}}|(1+D)} \Big)\nonumber\\
=&\Omega_{1,k}+\Sigma_{1,k},
\end{align} where $\Omega_{1,k}$ and $\Sigma_{1,k}$ are defined in \eqref{def:OmSi1k}.

%%%%%%%%%%%%%%%%%%%%%%%%%%%%%%%%%%%
\section{Proof of Proposition \ref{prop:step2}}\label{sec:exponential}
We will prove Proposition \ref{prop:step2} in Subsection \ref{sec:PropStep2} below. Before that 
we improve the decay rates of various functions in \eqref{eq:secab}. Based on the proved decay rates, it is easy to improve them to $t^{-K}$ for any $K>0.$ But to prove they actually decay exponentially fast requires a different set of techniques. 

The first result is the following proposition. Recall the cutoff function $\chi_{R}$ from \eqref{eq:chiRy}.
\begin{proposition}\label{prop:exponential}
If \eqref{eq:secab} holds, then the parametrization $\Psi_{0,v}$ works in the region, 
\begin{align*}
\Big\{y\ \Big|\ |y| \leq (1+\epsilon)R(\tau)\Big\},
\end{align*}
for some small $\epsilon>0,$ and the function $v$ is of the form
\begin{align}\label{eq:3decom}
v(y,\omega,\tau)=\sqrt{6+\sum_{n=0}^2\alpha_{n}(\tau)H_n(y)+\sum_{k=0,1} \sum_{l=1}^4 \alpha_{k,l}(\tau)H_{k}(y)\omega_l+\eta(y,\ \omega,\tau)},
\end{align}where $\chi_{R}\eta$ satisfies the orthogonality conditions
\begin{align}\label{eq:OrthoREta}
    \chi_{R}\eta\perp_{\mathcal{G}} H_{n},\ H_{k}\omega_l, \ n=0,1,2;\ k=0,1;\ l=1,2,3,4.
\end{align}
If $X_0$ is sufficiently large, then when $\tau\geq X_0,$ the $\omega$-dependent components satisfy the estimates
\begin{align}
\begin{split}\label{eq:Xinfty}
\sum_{k=0,1}\sum_{l=1}^4|\alpha_{k,l}(\tau)|\leq & X_{0}^{-\frac{1}{2}} e^{-\frac{3}{10}(\tau-X_0)},\\
\sum_{j+|i|\leq 2}\Big\|\partial_{y}^j \nabla_{E}^i\chi_{R}(1-K_1)\eta(\cdot,\tau)\Big\|_{\mathcal{G}}\leq&  X_{0}^{-\frac{1}{2}}  e^{-\frac{3}{10}(\tau-X_0)};
\end{split}
\end{align}
for the components independent of $\omega$, but is odd in the $y$-variable,
\begin{align}\label{eq:alpha2K2phi}
    |\alpha_1(\tau)|+ \sum_{j+|i|\leq 2}\Big\|\partial_{y}^j \nabla_{E}^i \chi_{R}K_2\eta(\cdot,\tau)\Big\|_{\mathcal{G}}\lesssim &  e^{-\frac{3}{10}(\tau-X_0)};
\end{align} and for the other parts
\begin{align}
\begin{split}\label{eq:alpha12Eta}
    |\alpha_{0}(\tau)|+|\alpha_{2}(\tau)|+\sum_{j+|i|\leq 2}\Big\|\partial_{y}^j \nabla_{E}^i \chi_{R}\eta(\cdot,\tau)\Big\|_{\mathcal{G}}\lesssim    e^{-\frac{3}{10}(\tau-X_0)}.
\end{split}
\end{align}
\end{proposition}
Here the operators $K_1$ and $K_2$ are defined as, for any function $\xi,$
\begin{align}\label{def:K}
(K_1\xi)(y):=&\frac{1}{\int_{\mathbb{S}^3} 1\ dS}\int_{\mathbb{S}^3}\xi(y,\omega) dS;\\
    K_2(\xi)(y):=& \frac{1}{2\int_{\mathbb{S}^3} 1\ dS}\int_{\mathbb{S}^3}\Big(\xi(y,\omega)-\xi(-y,\omega)\Big)\  dS.
\end{align}

These estimates will be proved in Subsections \ref{sub:oddOmega}, \ref{sub:alpha2K2Phi} and \ref{subsec:alpha02Eta} below.

For later purpose we derive a governing equation for $v$.
Recall that the MCF and the rescaled MCF as defined as $$\Phi_{0,u}= \left[
\begin{array}{ccc}
z\\
u(z, \omega, t)\omega
\end{array}
\right]:=\sqrt{T-t} \left[
\begin{array}{ccc}
y\\
v(y, \omega,\tau)\omega
\end{array}
\right].$$ In Appendix \ref{sec:derivation} we derive a governing equation for the function $u$,
\begin{align}\label{eq:UPparametrization}
\partial_{t}u=u^{-2}\Delta_{\mathbb{S}^3}u+\partial_{x}^2u-\frac{3}{u}+N(u),
\end{align} where 
$N$ is nonlinear in terms of $u$ and defined in \eqref{eq:tildeN}.
This makes $v$ satisfy the equation
\begin{align}\label{eq:effeV}
    \partial_{\tau}v=v^{-2}\Delta_{\mathbb{S}^3}v+\partial_{y}^2v-\frac{1}{2}y\partial_{y}v-\frac{1}{2}v-\frac{3}{v}+N(v),
\end{align} with $N(v)$ defined as
\begin{align*}
N(v):=&-\frac{v^{-4}\sum_{k=1}^3(\nabla_{\omega}^{\perp} v\cdot e_{k})\Big(\nabla_{\omega}^{\perp} v \cdot \nabla_{\omega}^{\perp}(\nabla_{\omega}^{\perp} v\cdot e_{k})\Big)}
{1+v^{-2}|\nabla_{\omega}^{\perp} v|^2+(\partial_{y}v)^2}\\
&-\frac{(\partial_{y}v)^2\partial_{y}^2 v+2v^{-2}\partial_{y}v ( \nabla_{\omega}^{\perp} v\cdot 
\nabla_{\omega}^{\perp} \partial_{y}v)+v^{-3}|\nabla_{\omega}^{\perp} v|^2}{1+v^{-2}|\nabla_{\omega}^{\perp} v|^2+(\partial_{y}v)^2},
\end{align*} and 
$\{e_k\}_{k=1}^4\subset \mathbb{S}^3$ are the standard unit vectors defined as 
\begin{align}
e_{k}:=(\delta_{1,k},\ \delta_{2,k}, \delta_{3,k},\ \delta_{4,k})\in \mathbb{R}^4,
\end{align}

%%%%%%%%%%%%%%%%%%%%%%%%%%%%%%%%%%%%%%%%%%%%%%%%%%%%%%%

Before deriving a governing equation for $\eta$, we define functions $\tilde{v}$ and $q$ as
\begin{align*}
q(y,\omega,\tau):=&6+\sum_{n=0}^3\sum_{k=0,1}\sum_{l}\alpha_{n,k,l}(\tau)H_{n}(y)f_{k,l}(\omega),\\
\tilde{v}:=&q+\eta=v^2.
\end{align*} 
The decomposition $v=\sqrt{q+\eta}$ and the governing equation for $v$ in \eqref{eq:UPparametrization} imply that
\begin{align}
\begin{split}
    \partial_{\tau}\tilde{v}=&v^{-2}\Delta_{\mathbb{S}^2}\tilde{v}+\partial_{y}^2 \tilde{v}-\frac{1}{2}y\partial_{y}\tilde{v} +\tilde{v}-6-\frac{1}{2}v^{-4} |\nabla_{E}\tilde{v}|^2-\frac{1}{2}v^{-2}|\partial_{y}\tilde{v}|^2+2v N_1(v)\\
    =&\frac{1}{q}\Delta_{\mathbb{S}^2}\tilde{v}+\partial_{y}^2 \tilde{v}-\frac{1}{2}y\partial_{y}\tilde{v} +\tilde{v}-6-\frac{\eta}{v^2 q}\Delta_{\mathbb{S}^2}\tilde{v}-\frac{1}{2}v^{-4} |\nabla_{E}\tilde{v}|^2-\frac{1}{2}v^{-2}|\partial_{y}\tilde{v}|^2+2v N(v).
    \end{split}
\end{align} By this we derive
\begin{align}
\partial_{\tau}\eta=-L\eta+F(q)+SN(\eta),
\end{align} where the linear operator $L$ is defined as
\begin{align}\label{def:L}
L:=-\partial_{y}^2+\frac{1}{2}y\partial_{y}-\frac{1}{6}\Delta_{\mathbb{S}^3}-1,
\end{align}
the function $F(q)$ is independent of $\eta$ and is defined as 
\begin{align}
\begin{split}
F:=&-\partial_{\tau}q-Lq-6+\frac{6-q}{6q}\Delta_{\mathbb{S}^3}q-\frac{1}{2q}(\partial_{y}q)^2-\frac{1}{2q^2} |\nabla_{E}q|^2+2\sqrt{q} N(\sqrt{q})\\
=&F_1+F_2,
\end{split}
\end{align} and the terms $F_1$ and $F_2$ are defined as, 
\begin{align*}
F_1:=&\sum_{n=0}^3 (-\frac{d}{d\tau}+\frac{2-n}{2})\alpha_n H_n+\sum_{l=1}^{4}\sum_{k=0}^3(-\frac{d}{d\tau}+\frac{1-k}{2})\alpha_{k,l} H_{k}\omega_l-\frac{1}{12}\alpha_{2,0,1}^2(\partial_{y}H_2)^2,\\
F_2:=&\frac{6-q}{6q}\Delta_{\mathbb{S}^3}q-\frac{1}{2q^2} |\nabla_{E}q|^2-\frac{1}{2q}|\partial_{y}q|^2+\frac{1}{12}\alpha_{2,0,1}^2(\partial_{y}H_2)^2+2\sqrt{q} N(\sqrt{q}),
\end{align*}
the term $SN(\eta)$ contains terms nonlinear in terms of $\eta$, or small linear terms,
\begin{align}
\begin{split}
SN(\eta):=&\frac{1}{2q^2}|\nabla_{E}q|^2-\frac{1}{2}v^{-4} |\nabla_{E}\tilde{v}|^2+\frac{1}{2q}|\partial_{y}q|^2-\frac{1}{2}v^{-2}|\partial_{y}\tilde{v}|^2+2v N(v)-2\sqrt{q} N(\sqrt{q})\\
&+\frac{6-q}{6q}\Delta_{\mathbb{S}^3}\eta-\frac{\eta}{v^2 q}\Delta_{\mathbb{S}^2}\tilde{v}.
\end{split}
\end{align}

Impose the cutoff function $\chi_{R}$ onto both sides to obtain
\begin{align}\label{eq:effectChiREta}
\partial_{\tau}\chi_{R}\eta=-L\chi_{R}\eta+\chi_{R}\Big(F   +SN(\eta)\Big)+\mu_{R}(\eta),
\end{align}
and the function $\mu_{R}(\eta)$ is linear in $\eta,$
\begin{align*}
\mu_{R}(\eta):=&\frac{1}{2}\big(y\partial_{y}\chi_{R}\big)\eta+\big(\partial_{\tau}\chi_{R}\big)\eta-\big(\partial_{y}^2\chi_{R}\big)\eta-2\partial_{y}\chi_{R}\partial_{y}\eta.
\end{align*}

%%%%%%%%%%%%%%%%%%%%%%%%%%%%%%%%%%%%%%%%%%%%%%%%%%%%%%%%%%%%%%%%%

\subsection{Proof of (\ref{eq:Xinfty})}\label{sub:oddOmega}
Here we apply the technique of optimal coordinates, which was used in our previous papers \cite{GaKn2014, FGKO}. Specifically,
fix a large time $X_0$, then we find an optimal coordinate, for any time $X\geq X_0$, by translating the center and tilting the axis, so that the rescaled MCF takes the form
\begin{align}\label{eq:newCoor}
U^{(X)}\Big(\left[
\begin{array}{cc}
y\\
v^{(X)}(y,\omega,\tau) \omega
\end{array}
\right]+\bf{A}^{(X)}\Big)
\end{align}where ${\bf{A}}^{(X)}\in \mathbb{R}^5$ is a time-independent vector, $U^{(X)}$ is a time-independent unitary rotation. We choose $A^{(X)},$  $U^{(X)}$ such that $v^{(X)}$ can be decomposed to the form
\begin{align}\label{eq:optimalTra}
v^{(X)}(y,\omega,\tau)=\sqrt{6+ \sum_{n=0}^2\alpha^{(X)}_n(\tau)H_n(y) +\sum_{k=0,1}\sum_{l=1 }^4 \alpha^{(X)}_{k,l}(\tau) H_{k}(y)\omega_l+\eta^{(X)}(y,\omega,\tau)}
\end{align} where $\eta^{(X)}$ satisfies the orthogonality conditions
\begin{align*}
\chi_{R}\eta^{(X)}\perp_{\mathcal{G}} \ H_n, \ H_{k}\omega_l, \ n=0,1,2; \ k=0,1;\ l=1,2,3,4, 
\end{align*} and at the time $\tau=X$, 
\begin{align}
\alpha_{k,l}^{(X)}(X)=0.\label{eq:optimal}
\end{align} These can be achieved since the directions $\omega_l$, $y\omega_l$ control center of the coordinate and tilts of axis.

Since our goal is to find decay rates for $\alpha_{n,k,l}$ and $\chi_{R}\eta$ through the decay rates for $\alpha_{n,k,l}^{(X)}$ and $\chi_{R}\eta^{(X)}$, we need to compare these functions. A preliminary results is the following:
\begin{lemma}\label{LM:odd} 
There exists a constant $C$ such that for any $X\geq X_0$ and any time $\tau\in [X_0,X]$ 
\begin{align}
\sum_{k=0}^{3}\sum_{l=1}^4 |\alpha_{k,l}^{(X)}(\tau)|+\sum_{n=0}^3 |\alpha_{n}^{(X)}(\tau)|+
\Big\|\langle y\rangle^{-3}\chi_{R}\eta^{(X)}(\cdot,\tau)\Big\|_{\infty}\leq & C\tau^{-2}
\end{align}
For any fixed $\tau\in [X_0,\infty)$ we have
\begin{align}
\begin{split}\label{eq:limitCoor}
\lim_{X\rightarrow\infty}\sum_{k=0,1}\sum_{l=1}^4\Big|\alpha_{k,l}^{(X)}(\tau)-\alpha(\tau)\Big|+&\sum_{n=0}^2\Big|\alpha^{(X)}_n(\tau)-\alpha_{n}(\tau)\Big|=0,\\
\lim_{X\rightarrow \infty}\eta^{(X)}(\cdot,\tau)=&\eta(\cdot,\tau).
\end{split}
\end{align}
\end{lemma}
\begin{proof}
Recall that we have that
\begin{align}
\sum_{n}|\alpha_n|+\sum_{k,l}|\alpha_{k,l}(\tau)|+\|\chi_{R}\eta(\cdot,\tau)\|_{\mathcal{G}}\lesssim & \tau^{-2}.
\end{align} Thus by \eqref{eq:newCoor}, the difference between the coordinate chosen for the time $\tau=X$ and the one chosen for $\tau=\infty$ is of the order $X^{-2}$
\begin{align}
|U^{(X)}-Id|,\ \|{\bf{A}^{(X)}}\|\lesssim X^{-2}.
\end{align} 

These obviously imply the desired estimates.

\end{proof}

The next result is an important step since it provides exponential decay rates.
\begin{proposition}\label{prop:step1}
If $X_0$ is sufficiently large, then for any fixed $X\geq X_0$, and for any $\tau\in [X_0,X]$,
\begin{align}
\begin{split}\label{eq:exponential}
\sum_{k=0,1}\sum_{l=1}^4|\alpha_{k,l}^{(X)}(\tau)|+\sum_{j+|i|\leq 2}\Big\|\partial_{y}^j \nabla_{E}^i\chi_{R}(1-K_1)\eta^{(X)}(\cdot,\tau)\Big\|_{\mathcal{G}}\leq & X_{0}^{-\frac{1}{2}} e^{-\frac{3}{10}(\tau-X_0)},
\end{split}
\end{align}
\end{proposition}
We will prove this proposition in the rest of the subsection.

By letting $X$ go to $\infty,$ this proposition and \eqref{eq:limitCoor} imply the desired \eqref{eq:Xinfty}.

To prepare for proving Proposition \ref{prop:step1}, we define two functions, 
\begin{align}
\begin{split}\label{eq:psi12}
\Psi_1(\tau):=&\sum_{i+|j|\leq 2}\Big\|\partial_{y}^{i}\nabla_{E}^j\chi_{R}(1-K_1)\eta^{(X)}(\cdot,\tau)\Big\|_{\mathcal{G}}^2,\\
\Psi_2(\tau):=&\sum_{k=0,1}\sum_{l=1}^4|\alpha^{(X)}_{k,l}(\tau)|.
\end{split}
\end{align}
The result is the following:
\begin{lemma}\label{LM:step1} Suppose that $X_0$ is sufficiently large. There exists a constant $C>0$ such that for any $X\geq X_0$ and any $\tau\in [X_0,X]$ the following estimates hold,
\begin{align}
\frac{d}{d\tau}\Psi_1(\tau)\leq -\frac{3}{5}\Psi_1(\tau)+\tau^{-2} \Psi_2^2(\tau),\label{eq:govPsi1}
\end{align}
and for $\alpha_{n,k},$ $n=0,1,$ 
\begin{align}
\begin{split}\label{eq:alphaOdd}
\sum_{l=1}^4|\alpha^{(X)}_{0,k}(\tau)|\leq &\int_{\tau}^{X}e^{-\frac{1}{2}(\sigma-\tau)} \sigma^{-\frac{3}{2}} \Big(\sqrt{\Psi_1(\sigma)}+\Psi_2(\sigma)\Big) \ d\sigma,\\
\sum_{l=1}^4 |\alpha^{(X)}_{1,k}(\tau)|\leq &\int^{X}_{\tau} \sigma^{-\frac{3}{2}} \Big(\sqrt{\Psi_1(\sigma)}+\Psi_2(\sigma)\Big)\ d\sigma.
\end{split}
\end{align}
\end{lemma}

This lemma will be proved in Subsection \ref{subsec:LModd} below.

Assuming this lemma holds, we prove Proposition \ref{prop:step1}.
\begin{proof}
Here we assume that $X_0$ is sufficiently large, so that Lemma \ref{LM:step1} is applicable.

We start with observing that when $X=X_0$ the proposition holds trivially since
\begin{align*}
\sum_{k=0,1}\sum_{l}|\alpha_{k,l}^{(X_0)}(X_0)|= & 0,\\
\Big\|\chi_{R}(1-K_1)\eta^{(X_0)}(\cdot,X_0)\Big\|_{\mathcal{G}}\leq & X_{0}^{-1}.
\end{align*}
Thus by continuity there exists a $X_1> X_0$ such that $\alpha_{k,l}^{(X_1)}$ and $(1-K_1)\eta^{(X_1)}$ satisfy the estimates \eqref{eq:exponential} when $\tau\in [X_0,\ X_1].$

In order to prove that \eqref{eq:exponential} hold for any $X>X_0$ we define a constant $X_{\max}\leq \infty$ as
\begin{align}
X_{\max}=\max\Big\{X\ \Big|\ X>X_0,\ \alpha_{k,l}^{(X)}\ \text{and}\ (1-K_1)\chi_{R}\eta^{(X)}\ \text{satisfy}\  \eqref{eq:exponential}\ \text{when}\ \tau\in [X_0,X]\Big\}.
\end{align}
If $X_{\max}=\infty$ then we have the desired results. 

Next we rule out the possibility $X_{\max}<\infty.$
If $X_{\max}<\infty,$ then by continuity $\alpha_{k,l}^{(X_{\max})}$ and $(1-K_1)\chi_{R}\eta^{(X_{\max})}$ satisfy \eqref{eq:exponential} when $\tau\in [X_0,X_{\max}]$.
Feed \eqref{eq:exponential} into the right hand sides of \eqref{eq:govPsi1}-\eqref{eq:alphaOdd} to obtain better estimates, specifically when $\tau\in [X_0,X_{\max}],$
\begin{align}
\begin{split}\label{eq:subBetter}
\sum_{k=0,1}\sum_{l=1}^4|\alpha_{k,l}^{(X_{\max})}(\tau)|+
\Big\|\chi_{R}(1-K_1)\eta^{(X_{\max})}(\cdot,\tau)\Big\|_{\mathcal{G}}\leq  X_{0}^{-1}  e^{-\frac{2}{5}(\tau-X_0)}.
\end{split}
\end{align}
Thus by continuity, there exists a $X_2>X_{\max}$ such that \eqref{eq:exponential} holds when $\tau\in [X_0,X_2]$. This contradicts to the definition of $X_{\max},$ thus rules out the possibility $X_{\max}<\infty.$

\end{proof}

%%%%%%%%%%%%%%%%%%%%%%%%%%%%%%%%%%%%%%%%%%%%%%%%%%%%%%%%%%%%%%%%%%%%%%%%%%%%%%%%%%%%%%%%%%%%%%%%%%%%%%%%%%%%%%%%%

\subsubsection{Proof of Lemma \ref{LM:step1}}\label{subsec:LModd}
Since here $X$ is fixed, to simplify the notations we will suppress the index $X$, so that
\begin{align}
\eta(y,\omega,\tau)=\eta^{(X)}(y,\omega,\tau), \ \alpha_{n}(\tau)=\alpha_{n}^{(X)}(\tau),\ \alpha_{k,l}(\tau)=\alpha_{k,l}^{(X)}(\tau).
\end{align}
Since the effective equation is independent of the coordinate system, the equation \eqref{eq:effectChiREta} applies.

We start with proving \eqref{eq:alphaOdd}. 

From the orthogonality conditions imposed on $\chi_{R}\eta$ we derive
\begin{align}\label{eq:evenN}
\Big(\frac{d}{d\tau}-\frac{k-1}{2}\Big)\alpha_{k,l}=&\frac{1}{\|H_{k}\omega_l\|_{\mathcal{G}}^2}\Big\langle \chi_{R}\Big(F_2+SN\Big), \ H_{k}\omega_l\Big\rangle_{\mathcal{G}}+o(e^{-\frac{1}{5}R^2})\nonumber\\
=&\frac{1}{\|H_{k}\omega_l\|_{\mathcal{G}}^2}\Big\langle \chi_{R}(1-K_1)\Big(F_2+SN\Big), \ H_{k}\omega_l\Big\rangle_{\mathcal{G}}+o(e^{-\frac{1}{5}R^2}),
\end{align} where we use that $\Big\langle -\frac{1}{12}\alpha_{2}^2(\partial_{y}H_2)^2,\ H_n\omega_{l}\Big\rangle_{\mathcal{G}}=0$, the term $o(e^{-\frac{1}{5}R^2})$ is from $\Big\langle \mu_{R}(\eta),\ H_{k} \omega_l\Big\rangle_{\mathcal{G}}$ since the estimates $\sum_{k=1,2}|\partial_{y}^k\eta|\ll 1$ and that $\mu_{R}(\eta)$ is supported by the set $|y|\geq R$ imply that
\begin{align}
|\Big\langle \mu_{R}(\eta),\ H_{k}\omega_l \Big\rangle_{\mathcal{G}}|\leq \int_{|y|\geq R} e^{-\frac{1}{4}y^2}\ dy\ll e^{-\frac{1}{5}R^2},
\end{align}
and in the last step we use that $(1-K_1)H_k \omega_l=H_k \omega_l$ and the operator $K_1$ is self-adjoint.

The operator $1-K_1$ forces $\alpha_{k,l}$ and $(1-K_1)\chi_{R}\eta$ to contribute. This, together with the estimates
$|\alpha_{n}(\tau)|+|\alpha_{k,l}|+\sum_{j+|i|\leq 2}\|\partial_{y}^{j}\nabla_{E}^{i}\chi_{R}\eta\|_{\mathcal{G}}\lesssim \tau^{-2}\ll \tau^{-\frac{3}{2}}$, implies that
\begin{align}
    |\Big(\frac{d}{d\tau}-\frac{k-1}{2}\Big)\alpha_{k,l}|\ll \tau^{-\frac{3}{2}}\Psi(\tau), 
\end{align}
which, together with that $\alpha_{k,l}(X)=0$, implies the desired \eqref{eq:alphaOdd}.

Next we prove \eqref{eq:govPsi1}. 

To simplify the notations we define two functions $\eta_1$ and $g_{j,l}$ as
\begin{align}
\begin{split}
\eta_1:=(1-K_1)\eta,\\
g_{j,l}:=\Big\langle \partial_{y}^j\chi_{R}\eta_1,\ & (-\Delta_{\mathbb{S}^3}+1)^l  \partial_{y}^j\chi_{R}\eta_1\Big\rangle_{\mathcal{G}}.
\end{split}
\end{align}

Impose the linear operator $1-K_1$ onto \eqref{eq:effectChiREta}, and take a $\mathcal{G}$-inner product with $\chi_{R}\eta_1$ to obtain
\begin{align}\label{eq:gilEqn}
\frac{1}{2}\frac{d}{d\tau} g_{j,l} 
=&-\Big\langle e^{-\frac{1}{8}y^2}\partial_{y}^j\chi_{R}\eta_1,\ (-\Delta_{\mathbb{S}^3}+1)^l  (\mathcal{L}+\frac{j}{2})e^{-\frac{1}{8}y^2}\partial_{y}^j\chi_{R}\eta_1\Big\rangle+\sum_{n=1}^{3}D_{n}+o(e^{-\frac{1}{5}R^2}),
\end{align} where the operator $\mathcal{L}$ is defined as
\begin{align}
    \mathcal{L}:=e^{-\frac{1}{8}y^2} L e^{\frac{1}{8}y^2}=-\partial_{y}^2-\Delta_{\mathbb{S}^3}+\frac{1}{16}y^2-\frac{1}{4}-1,
\end{align}
and $D_{n}, \ n=1,2,$ are defined as
\begin{align*}
    D_{n}:=\Big\langle \partial_{y}^k\chi_{R}\eta_1,\ (-\Delta_{\mathbb{S}^3}+1)^l (1-K_1)\partial_{y}^{k}\chi_{R} F_n   \Big\rangle_{\mathcal{G}},
\end{align*} and $D_3$ is defined as
\begin{align*}
    D_3:=\Big\langle \partial_{y}^k\chi_{R}\eta_1,\ (-\Delta_{\mathbb{S}^3}+1)^l (1-K_1)\partial_{y}^{k}\chi_{R} SN\Big\rangle_{\mathcal{G}}.
\end{align*}

We start with estimating $D_1$. Observe that $\chi_{R}\eta\perp_{\mathcal{G}}  H_k \omega_l$, $k=0,1,$ hence for any $j$, $\partial_{y}^j\chi_{R}\eta\perp_{\mathcal{G}} \partial_{y}^j H_k \omega_l$.---- Note that if $j\geq 2$, $\partial_{y}^j H_k=0.$
Thus
\begin{align*}
    \Big\langle \partial_{y}^j   \chi_{R}\eta_1,\  (-\Delta_{\mathbb{S}^3}+1)^l \partial_{y}^j H_k \omega_l\Big\rangle_{\mathcal{G}}=(\frac{3}{2})^l\Big\langle \partial_{y}^j   \chi_{R}\eta,\   \partial_{y}^j H_k \omega_l\Big\rangle_{\mathcal{G}}=0.
\end{align*} This directly implies that 
\begin{align}
|D_1|\leq \Big|\Big\langle \partial_{y}^j   \chi_{R}\eta_1,\  (-\Delta_{\mathbb{S}^3}+1)^l \partial_{y}^j\chi_{R}H_k \omega_l\Big\rangle_{\mathcal{G}}\Big|\ll  e^{-\frac{1}{10}R^2}\Big\|\partial_{y}^j\chi_{R}\eta_1\Big\|_{\mathcal{G}}.
\end{align}

For $D_2$ the operator $1-K_1$ forces $\alpha_{k,l}$ to contribute. Since $F_2$ is nonlinear in terms of $\alpha_{n}$ and $\alpha_{k,l}$, and that $|\alpha_{n}(\tau)|+|\alpha_{k,l}(\tau)|\lesssim \tau^{-2}\ll \tau^{-\frac{3}{2}}$, 
\begin{align}
   |D_2|\ll &\tau^{-\frac{3}{2}}\sum_{k=0,1}\sum_{l}|\alpha_{k,l}|\Big\|\partial_{y}^j\chi_{R}\eta_1\Big\|_{\mathcal{G}}\leq \tau^{-\frac{3}{2}}\Psi_2 \Big\|\partial_{y}^j\chi_{R}\eta_1\Big\|_{\mathcal{G}}
\end{align} 

For $D_3$, for some terms we need to integrate by parts, otherwise just compute directly,
\begin{align}
\begin{split}\label{eq:3NL}
D_3\ll & \tau^{-\frac{1}{4}} \sum_{i\leq j}\sum_{|n|\leq |l|}\Big\langle e^{-\frac{1}{8}y^2}\partial_{y}^{i}\chi_{R}\eta_1,  \ (-\Delta_{\mathbb{S}^3}+1)^{n}\Big(-\Delta_{\mathbb{S}^3}-\partial_{y}^2+\frac{1}{16}y^2+1\Big) e^{-\frac{1}{8}y^2}\partial_{y}^{i}\chi_{R}\eta_1\Big\rangle\\
&+\Big(\tau^{-\frac{3}{2}}\Psi_2+e^{-\frac{1}{10} R^2} \Big) \Big\|(-\Delta_{\mathbb{S}^3}+1)^{\frac{l}{2}}\partial_{y}^j\chi_{R}\eta_1\Big\|_{\mathcal{G}}.
\end{split}
\end{align} 
Here the factor $\tau^{-\frac{1}{4}}$ is from that $|\partial_{y}^{j}\nabla_{E}^{l}\eta(\cdot,\tau)|\ll \tau^{-\frac{1}{4}}$ and $|\alpha_{n}(\tau)|+|\alpha_{k,l}(\tau)|\lesssim \tau^{-2}\ll \tau^{-\frac{1}{4}}.$

Returning to \eqref{eq:gilEqn}, we collect the estimates above to obtain,
\begin{align}
    \frac{1}{2}\frac{d}{d\tau}\Psi_1\leq & -\Big(1-\tau^{-\frac{1}{4}}\Big) \sum_{j,l}\Big\langle e^{-\frac{1}{8}y^2}\partial_{y}^j\chi_{R}\eta_1,\ (-\Delta_{\mathbb{S}^3}+1)^l  (\mathcal{L}+\frac{j}{2})e^{-\frac{1}{8}y^2}\partial_{y}^j\chi_{R}\eta_1\Big\rangle\nonumber\\
    &\hspace{1cm}+\Big(\tau^{-\frac{3}{2}}\Psi_2+e^{-\frac{1}{10} R^2} \Big) \sqrt{\Psi_1}\nonumber\\
    \leq & -\frac{1-\tau^{-\frac{1}{4}}}{3} \Psi_1 +\Big(\tau^{-\frac{3}{2}}\Psi_2+e^{-\frac{1}{10} R^2} \Big) \sqrt{\Psi_1},\label{eq:geqn}
\end{align}
where, we use that, by the orthogonality conditions enjoyed by $\chi_{R}\eta$,
$$\Big\langle e^{-\frac{1}{8}y^2}\partial_{y}^j\chi_{R}\eta_1,\ (-\Delta_{\mathbb{S}^3}+1)^l  (\mathcal{L}+\frac{j}{2})e^{-\frac{1}{8}y^2}\partial_{y}^j\chi_{R}\eta_1\Big\rangle\geq \frac{1}{3} g_{j,l},$$ the constant $\frac{1}{3}$ is sharp since it is the eigenvalue of eigenfunctions $e^{-\frac{1}{8}y^2}H_{0}f_{2,l}$, see \eqref{eq:specOmega} and \eqref{eq:specL0}.

What is left is to apply the Schwartz inequality to obtain the desired \eqref{eq:govPsi1}.

\subsection{Proof of (\ref{eq:alpha2K2phi}) }\label{sub:alpha2K2Phi}
We start with deriving effective equations for the functions $\alpha_1$ and $\phi$, which is defined as $$\phi(\tau):=\sum_{j\leq 2}\Big\|\partial_{y}^j\chi_{R}K_2\eta(\cdot,\tau)\Big\|_{\mathcal{G}}.$$ Similar to deriving \eqref{eq:evenN} and \eqref{eq:geqn}, we derive, for some small positive constant $\epsilon_0\ll 1$ and some function $g$ satisfying $|g(\tau)|\leq \epsilon_0$,
\begin{align}
\begin{split}\label{eq:PhiAlpha}
    \Big|(\frac{d}{d\tau}+\frac{1}{3})\phi\Big|\leq \epsilon_0 \Big[\phi+|\alpha_1|+e^{-\frac{3}{10}(\tau-X_0)}\Big],\\
    \Big|\Big(\frac{d}{d\tau}-\frac{1}{2}-g\Big)\alpha_1\Big|\leq \epsilon_0 \Big(\phi+e^{-\frac{3}{10}(\tau-X_0)}\Big).
\end{split}
\end{align} Here the first equation holds piecewisely, since $\phi$ might not be differentiable at $\tau$ when $\phi(\tau)=0$.

We rewrite the first equation as
\begin{align}
    \phi(\tau)\leq & e^{-\frac{1-3\epsilon_0}{3}(\tau-X_0)}\phi(X_0)+\epsilon_0\int_{X_0}^{\tau}e^{-\frac{1-3\epsilon_0}{3}(\tau-\sigma)} \Big(|\alpha_1(\sigma)|+e^{-\frac{3}{10}(\sigma-X_0)}\Big)\ d\sigma\nonumber\\
    \leq &100  e^{-\frac{3}{10}(\tau-X_0)} \tilde\epsilon_0+\epsilon_0\int_{X_0}^{\tau}e^{-\frac{3}{10}(\tau-\sigma)} |\alpha_1(\sigma)|\ d\sigma,\label{eq:phitau}
\end{align} where $\tilde{\epsilon}_0$ is defined as $\tilde\epsilon_0:=\phi(X_0)+\epsilon_0\ll 1;$
and for the second, since $\lim_{\tau\rightarrow \infty}\alpha_1(\tau)= 0$,
\begin{align}
|\alpha_1(\tau)|\leq & \epsilon_0\int_{\tau}^{\infty}e^{\frac{1-2\epsilon_0}{2}(\tau-\sigma)}\Big(\phi(\sigma)+e^{-\frac{3}{10}(\sigma-X_0)}\Big)\ d\sigma\leq 5\epsilon_0 e^{-\frac{3}{10}(\tau-X_0)}+\epsilon_0\int_{\tau}^{\infty}e^{\frac{3}{10}(\tau-\sigma)}\phi(\sigma)\ d\sigma.\label{eq:alpha1De}
\end{align}
Emerge \eqref{eq:phitau} and \eqref{eq:alpha1De} into one inequality,
\begin{align}
    \phi(\tau)\leq & 200 e^{-\frac{3}{10}(\tau-X_0)} \tilde\epsilon_0 +\epsilon_0^2 \int_{X_0}^{\tau}e^{-\frac{3}{10}(\tau-\sigma)} \int_{\sigma}^{\infty} e^{\frac{3}{10}(\sigma-\sigma_1)}\phi(\sigma_1)\ d\sigma_1 d\sigma\nonumber\\
    =&f(\tau)+\epsilon_0^2 H(\phi)(\tau),\label{eq:contraction}
\end{align} where $f$ and $H$ are a function and a linear operator, they are naturally defined.
By iteration,
\begin{align}
\phi(\tau)\leq f(\tau)+\epsilon_0^2 H(f)(\tau)+\epsilon_0^4 H(H(\phi))(\tau).
\end{align} Iterating infinitely many times, we obtain that, since $\lim_{n\rightarrow \infty}\epsilon_0^{n}=0,$ $$\phi(\tau)\leq \psi(\tau)$$ with $\psi$ being the unique solution to the equation
\begin{align}
    \psi=f(\tau)+\epsilon_0^2 H(\psi)
\end{align} with condition $\psi(\tau)\lesssim \tau^{-2}$ when $\tau\geq X_0$. 

On the other hand this equation has only one solution, and it decays exponentially fast, 
\begin{align}
    \psi(\tau)\leq 400 e^{-\frac{3}{10}(\tau-X_0)} \tilde\epsilon_0.
\end{align}

This and \eqref{eq:alpha1De} imply the desired results
\begin{align}
    \phi(\tau)+|\alpha_1(\tau)|\ll e^{-\frac{3}{10}(\tau-X_0)}.
\end{align}

%%%%%%%%%%%%%%%%%%%%%%%%%%%%%%%%%%%%%%%%%%%%%%%%%%%%%%%%%%%%%%%%%%%%%%%%%%%%%%%%%%%%%%%

\subsection{Proof of (\ref{eq:alpha12Eta}) }\label{subsec:alpha02Eta}

\begin{proof}
In the previous subsections we improved estimates for the part depending on $\omega$, and that odd in the $y$-variable. Now we consider the remaining parts, specifically, $\alpha_0,\ \alpha_2$ and $\eta$. The present problem is very similar to the blowup problem of nonlinear heat equation, thus the ideas in \cite{FK1992} become relevant.

We derive the following equations for $\alpha_0$, $\alpha_2$ and $\Big\|\chi_{R}\eta(\cdot,\tau)\Big\|_{\mathcal{G}}^2$: for some constants $c$ and $\epsilon_1\ll 1,$
\begin{align}
\Big| \frac{d}{d\tau}\alpha_2+\frac{1}{3} \alpha_2^2\Big|\leq & c D, \label{eq:effAlpha2}\\
|\frac{d}{d\tau}\alpha_0-\alpha_0|\leq & c\Big(|\alpha_0|^2+|\alpha_2|^2+D\Big),\label{eq:alpha0}\\
\Big\|\chi_{R}\eta(\cdot,\tau)\Big\|_{\mathcal{G}}^2\leq & c\Big[ e^{-\frac{3}{5}(\tau-X_0)} \Big\|\chi_{R}\eta(\cdot,X_0)\Big\|_{\mathcal{G}}^2+\int_{X_0}^{\tau}e^{-\frac{3}{5}(\tau-\sigma)} \big(\alpha_2^4(\sigma)+\epsilon_1 e^{-\frac{3}{10}(\sigma-X_0)}\big)\ d\sigma\Big].\label{eq:etaG}
\end{align} Here the function $D$ is defined in terms of $\alpha_2$ and $\|\chi_R\eta\|_{\mathcal{G}}$,
\begin{align}\label{def:D}
D:=|\alpha_2|^3(\tau)+\epsilon_1 e^{-\frac{3}{10}(\tau-X_0)}+\Big(|\alpha_2|(\tau)+e^{-\frac{3}{10}(\tau-X_0)}\Big)\|\chi_{R}\eta\|_{\mathcal{G}}+\|\chi_{R}\eta\|_{\mathcal{G}}^{\frac{19}{20}}.
\end{align}

The focus is to study $\alpha_2$. To simplify the problem, we emerge \eqref{eq:effAlpha2} and \eqref{eq:etaG} into one ``autonomous" inequality and find, for some large constant $d$,
\begin{align}
    \Big| \frac{d}{d\tau}\alpha_2+\frac{1}{3} \alpha_2^2\Big|\leq d E(\alpha_2,\tau),\label{eq:dealpha2}
\end{align} where the function $E$ only depends on $\alpha_2$ and $\tau$, and is defined as,
\begin{align*}
    &E(\alpha_2,\tau)\\
    :=&\Big(|\alpha_2|(\tau)+e^{-\frac{3}{10}(\tau-X_0)}\Big)\Big[ e^{-\frac{3}{5}(\tau-X_0)} \Big\|\chi_{R}\eta(\cdot,X_0)\Big\|_{\mathcal{G}}^2+\int_{X_0}^{\tau}e^{-\frac{3}{5}(\tau-\sigma)} \Big(\alpha_2^4(\sigma)+\epsilon_1 e^{-\frac{3}{10}(\sigma-X_0)}\Big)\ d\sigma\Big]^{\frac{1}{2}}\\
    &+|\alpha_2|^3(\tau)+\epsilon_1 e^{-\frac{3}{10}(\tau-X_0)}\\
    &+  \Big[ e^{-\frac{3}{5}(\tau-X_0)} \Big\|\chi_{R}\eta(\cdot,X_0)\Big\|_{\mathcal{G}}^2+\int_{X_0}^{\tau}e^{-\frac{3}{5}(\tau-\sigma)} \Big(\alpha_2^4(\sigma)+\epsilon_1 e^{-\frac{3}{10}(\sigma-X_0)}\Big)\ d\sigma\Big]^{\frac{7}{8}}.
\end{align*}

Now we are ready to improve the decay rates. We claim that, if $X_0$ is sufficiently large, then
\begin{align}
    \alpha_2^{2}(\tau)<  10 d E(\alpha_2(\tau),\tau)\ \text{for all}\ \tau\geq X_0.\label{eq:keycri}
\end{align} 

Suppose this holds, then we prove the desired results by observing 
\begin{align}
|\alpha_2(\tau)|\leq \alpha_2^{*}(\tau)\label{eq:alalstar}
\end{align} where $\alpha_2^*>0$ is the solution to the following evolution equation, for $\tau\geq X_0,$
\begin{align}
\alpha_2^*(\tau)=&10d E (\alpha_2^*(\tau),\tau),
\end{align} with the initial condition,
\begin{align*}
\alpha_2^*(X_0)=&10d E(|\alpha_2(X_0)|,X_0)\ll 1.
\end{align*} By a standard application of the fixed point theorem we prove, for some $C>0$, 
\begin{align}
\alpha_2^*(\tau)\leq C e^{-\frac{3}{10}(\tau-X_0)}.
\end{align} This, together with \eqref{eq:alalstar} and \eqref{eq:etaG}, implies that 
\begin{align}
   |\alpha_2(\tau)|+\sum_{j+|i|\leq 2}\Big\|\partial_{y}^j \nabla_{E}^i \chi_{R}\eta(\cdot,\tau)\Big\|_{\mathcal{G}}\leq e^{-\frac{3}{10}(\tau-X_0)}.
\end{align} To estimate $\alpha_0$, we rewrite \eqref{eq:alpha0} as
\begin{align*}
    |\alpha_{0}(\tau)|\lesssim  \int_{\tau}^{\infty} e^{\tau-\sigma} \Big(|\alpha_0(\tau)|^2+e^{-\frac{3}{10}(\sigma-X_0)}\Big)\ d\sigma
\end{align*} and prove the desired result
\begin{align}
    |\alpha_{0}(\tau)|\lesssim e^{-\frac{3}{10}(\tau-X_0)}.
\end{align}

What is left is to prove \eqref{eq:keycri}. We prove it by contradiction. Suppose that for some $ \tau_1\geq X_0$, 
\begin{align}
    \alpha_2^2(\tau_1)\geq 10d E(\alpha_2(\tau_1),\tau_1).
\end{align} This, \eqref{eq:dealpha2} and $|\alpha_2(\tau)|\lesssim  \tau^{-2}$, imply that, for some small constant $\delta$,
\begin{align}
\frac{d}{d\tau} \alpha_2(\tau)=-\frac{1}{3} \alpha_2^2(\tau)\Big(1+\delta\Big) \ \text{at the time}\ \tau=\tau_1.
\end{align}
Now we have an advantage, the equation $\frac{d}{d\tau}b=-\frac{1}{3}b^2(\tau)$ can be solved exactly!
It is easy to prove that, for any time $\tau\geq \tau_1$,
 \begin{align}
 \alpha_2(\tau)=\frac{1}{\frac{1}{3}(\tau-\tau_1)+\frac{1}{\alpha_2(\tau_1)}}\big(1+o(1)\big).
 \end{align} This contradicts to that $|\alpha_2(\tau)|\lesssim \tau^{-2}.$ Thus \eqref{eq:keycri} must hold!

\end{proof}

%%%%%%%%%%%%%%%%%%%%%%%%%%%%%%%%%%%%%%%%%%%%%%%%%%%%%%%%%%%%%%%%%%%%%%%%%%%%%%%%%%%%%%%%%%%%%%%
\subsection{Proof of Proposition \ref{prop:step2}}\label{sec:PropStep2}
To simplify the notations we start the rescaled MCF from the time $\tau=X_0$, hence the estimates in Proposition \ref{prop:exponential} hold for $X_0=0.$ 

We start with a preliminary result. For any integer $M$, we define $P_{\omega, M}$ to be an orthogonal projection such that, for any function $g(\omega)=\sum_{k=0}^{\infty}\sum_l g_{k,l} f_{k,l}(\omega),$
\begin{align}\label{eq:PomegaM}
    P_{\omega, M}g=\sum_{k>M}\sum_l g_{k,l} f_{k,l}.
\end{align}
\begin{lemma}\label{LM:3mPrelim}
For any $N\geq 2$, there exist unique functions $\gamma_{n,k,l}$ such that 
\begin{align}\label{eq:decomM3AfNor}
v(y,\omega,\tau)=\sqrt{6+\sum_{n=0}^3 \sum_{k=0}^{N} \sum_{l} \alpha_{n,k,l}(\tau)H_n(y)f_{k,l}(\omega)+\eta(y,\omega,\tau)}
\end{align} and $\chi_{R}\eta$ satisfies the orthogonality conditions
\begin{align*}
\chi_{R}\eta\perp_{\mathcal{G}} H_{n} f_{k,l},\ n=0,1,2,3; \ k=0,\cdots, N.
\end{align*}
The functions satisfy the following estimates,
\begin{align}
    |\alpha_{n,k,l}(\tau)|+\|\chi_{R}\eta(\cdot,\tau)\|_{\mathcal{G}}\lesssim e^{-\frac{3}{10}\tau}.\label{eq:n3Preli}
\end{align} and for any $M\leq N$, there exist some constants $C_{M,N}$ and $\phi_{M}$ such that
\begin{align}
    |\alpha_{M,k,l}(\tau)|+\|P_{\omega, M}\chi_{R}\eta(\cdot,\tau)\|_{\mathcal{G}}\leq C_{M,N}e^{-\phi_{M} \tau}\label{eq:arbiFast}
\end{align} with $\phi_{M}\rightarrow \infty$ as $M,\ N\rightarrow  \infty.$ Here $\phi_{M}$ is independent of $N$.
\end{lemma}

It is easy to prove \eqref{eq:n3Preli} since, compared to Proposition \ref{prop:exponential}, we just expand the function $\eta$.

What is left is to prove \eqref{eq:arbiFast} by studying the governing equations.

To simplify notations we define a functions $q$ as
\begin{align*}
q(y,\omega,\tau):=&6+\sum_{n=0}^3\sum_{k=0}^{N}\sum_{l}\alpha_{n,k,l}(\tau)H_{n}(y)f_{k,l}(\omega).
\end{align*} 
From the governing equation of $v$ in (\ref{eq:effeV}) we derive
\begin{align}
\partial_{\tau}\chi_{R}\eta=-L\chi_{R}\eta+\chi_{R}\Big(G   +SN(\eta)\Big)+\mu_{R}(\eta),
\end{align}here $G=G_1+G_2$ and, 
\begin{align*}
G_1:=&-\sum_{n=0}^3\sum_{k=0}^N \sum_{l}\Big(\frac{d}{d\tau}+\frac{n-2}{2}+\frac{k(k+2)}{6}\Big)\alpha_{n,k,l} H_{n}(y)f_{k,l}(\omega),\\
G_2:=&\frac{6-q}{6q}\Delta_{\mathbb{S}^3}q-\frac{1}{2q^2} |\nabla_{E}q|^2-\frac{1}{2q}|\partial_{y}q|^2+2\sqrt{q} N(\sqrt{q}),
\end{align*}
the term $SN(\eta)$ contains terms nonlinear in terms of $\eta$, or small linear terms,
\begin{align}
\begin{split}
SN(\eta):=&\frac{1}{2q^2}|\nabla_{E}q|^2-\frac{1}{2}v^{-4} |\nabla_{E}v^2|^2+\frac{1}{2q}|\partial_{y}q|^2-\frac{1}{2}v^{-2}|\partial_{y}v^2|^2+2v N(v)-2\sqrt{q} N(\sqrt{q})\\
&+\frac{6-q}{6q}\Delta_{\mathbb{S}^3}\eta-\frac{\eta}{v^2 q}\Delta_{\mathbb{S}^2}v^2.
\end{split}
\end{align}
Through the orthogonality conditions enjoyed by $\chi_{R}\eta$ in \eqref{eq:decomM3AfNor}, we derive
\begin{align}
\begin{split}\label{eq:diffAlpha100}
\Big(\frac{d}{d\tau}+\frac{n-2}{2}+\frac{k(k+2)}{6}\Big)\alpha_{n,k,l}
=\frac{1}{\| H_n f_{k,l}\|^2_{\mathcal{G}}}\Big\langle G_2+SN+\mu_R(\eta),\ H_n f_{k,l}\Big\rangle_{\mathcal{G}}=:NL_{n,k,l}
\end{split}
\end{align} where the function $NL_{n,k,l}$ is naturally defined.

To prepare for finding the decay estimates for $\alpha_{n,k,l},$ we rewrite these equations.
Here we need to consider two (and only two) different cases separately: $\frac{n-2}{2}+\frac{k(k+2)}{6}\leq 0$ and $\frac{n-2}{2}+\frac{k(k+2)}{6}\geq \frac{1}{3}.$
When $\frac{n-2}{2}+\frac{k(k+2)}{6}\leq 0$, which includes $(n,k)=(0,0),\  (1,0),\  (2,0),\  (0,1),\  (1,1)$, since $\lim_{\tau\rightarrow \infty}\alpha_{n,k,l}(\tau)=0,$ we rewrite \eqref{eq:DiffAlBe} as
\begin{align}\label{eq:governLe0}
       \alpha_{n,k,l}(\tau)=&-\int_{\tau}^{\infty}e^{-(\frac{n-2}{2}+\frac{k(k+2)}{6})(\tau-\sigma)} NL_{n,k,l}\ d\sigma.
\end{align} 
When $\frac{n-2}{2}+\frac{k(k+2)}{6}\geq \frac{1}{3},$
\begin{align}\label{eq:3Others}
\alpha_{n,k,l}(\tau)=e^{-(\frac{n-2}{2}+\frac{k(k+2)}{6})\tau}+\int_{0}^{\tau} e^{-(\frac{n-2}{2}+\frac{k(k+2)}{6})(\tau-\sigma)}NL_{n,k,l}\ d\sigma.
\end{align}

To control the remainder $\eta$, we define a function $G_{n,m}^{(M)}$
\begin{align}
    G_{n,m}^{(M)}:=\Big\langle (-\Delta_{\mathbb{S}^3}+1)^{m}\partial_{y}^n P_{\omega,M}\chi_{R}\eta,\partial_{y}^n\chi_{R}\eta\Big\rangle_{\mathcal{G}}.
\end{align} Similar to deriving \eqref{eq:geqn}, for some constant $E_{M}$ satisfying $\lim_{M\rightarrow \infty}E_{M}= \infty,$
\begin{align}
\begin{split}\label{eq:gmmm}
   \Big( \frac{1}{2}\frac{d}{d\tau}+E_{M}\Big)G_{n,m}^{(M)}\leq & |\Big\langle (-\Delta_{\mathbb{S}^3}+1)^{m}\partial_{y}^n P_{\omega,M}\chi_{R}\eta,\  \partial_{y}^n\chi_{R}\Big(G_2+SN\Big)\Big\rangle_{\mathcal{G}}|\\
   &\hspace{2cm} + [G_{n,m}^{(M)}]^{\frac{1}{2}} \ o(e^{-\frac{1}{10}R^2}).
\end{split}
\end{align}

We are ready to prove \eqref{eq:arbiFast}. The idea is simple: \eqref{eq:governLe0}-\eqref{eq:gmmm} show that, as $k$ and $M$ increase, the functions $\alpha_{n,k,l}$ and $G_{n,m}^{(M)}$ decay faster, and moreover they interact with the lower frequency parts $\alpha_{m,i,j}$, $i<k,$ and $(1-P_{\omega,M})\chi_{R}\eta$ through the nonlinearity, in a favorable way. 

Also the definition $R(\tau)=8\tau^{\frac{1}{2}+\frac{1}{20}}$ makes $e^{-\frac{1}{10}R^2}$ decay faster than $e^{-K\tau}$ for any $K>0.$

By these ideas we prove \eqref{eq:arbiFast}. The detailed proof is easy, but tedious. We skip the details.

Now we are ready to prove the desired Proposition \ref{prop:step2}.\\
{\bf{Proof of Proposition \ref{prop:step2}}}
\begin{proof}
Define a function $G_{n,m}$, with $n+\frac{m}{2}\leq 2$, as 
\begin{align}
    G_{n,m}:=\Big\langle (-\Delta_{\mathbb{S}^3}+1)^{m}\partial_{y}^n \chi_{R}\eta,\partial_{y}^n\chi_{R}\eta\Big\rangle_{\mathcal{G}}.\label{eq:gnm0}
\end{align} Similar to deriving \eqref{eq:geqn},
\begin{align}
\begin{split}\label{eq:3gnm}
   \Big( \frac{1}{2}\frac{d}{d\tau}+1+\mathcal{O}(\tau^{-\frac{1}{4}})\Big)G_{n,m}\leq |\Big\langle (-\Delta_{\mathbb{S}^3}+1)^{m}\partial_{y}^n\chi_{R}\eta,\  \partial_{y}^n\chi_{R}    G_2\Big\rangle_{\mathcal{G}}|
   + G_{n,m}^{\frac{1}{2}} \ o(e^{-\frac{1}{10}R^2}).
\end{split}
\end{align} Here the orthogonality conditions enjoyed by $\chi_{R}\eta$ cancel contributions from the slowly decaying terms of $G_2$, provided that $N$ is large enough. This, (\ref{eq:arbiFast}) and the definition of $G_2$ imply 
\begin{align}
|\Big\langle (-\Delta_{\mathbb{S}^3}+1)^{m}\partial_{y}^n\chi_{R}\eta,\  \partial_{y}^n\chi_{R}    G_2\Big\rangle_{\mathcal{G}}|\lesssim G_{n,m}^{\frac{1}{2}}  e^{-\tau}.\label{eq:3mtau}
\end{align}

Now we improve decay estimates for $\alpha_{n,k,l}$ and $G_{n,m}$.

From \eqref{eq:n3Preli} we derive, for all the tuples $(n,k,l)$,
\begin{align}\label{eq:prelimNL}
    |NL_{n,k,l}(\tau)|\lesssim & e^{-\frac{3}{5}\tau}.
\end{align}

The slowest decaying functions are $\alpha_{0,2,l}$: \eqref{eq:3Others} and \eqref{eq:prelimNL} imply sharp estimates
\begin{align}\label{eq:02Sharp}
    |\alpha_{0,2,l}(\tau)|=|e^{-\frac{1}{3}\tau}\Big(\alpha_{0,2,l}(0)+\int_{0}^{\tau} e^{\frac{1}{3}\tau} NL_{0,2,l}(\sigma)\ d\sigma \Big)|\lesssim e^{-\frac{1}{3}\tau},
\end{align} 
When $(n,k)=(3,0)$ and $(2,1)$, $\alpha_{n,k,l}$ also decay slowly,
\begin{align}
    \alpha_{n,k,l}(\tau)=e^{-\frac{1}{2}\tau}\Big(\alpha_{n,k,l}(0)+\int_{0}^{\tau} e^{\frac{1}{2}\sigma} NL_{n,k,l}(\sigma)\ d\sigma\Big)=\mathcal{O}(e^{-\frac{1}{2}\tau}).\label{eq:3021}
\end{align}

For the others, we have $\frac{n-2}{2}+\frac{k(k+2)}{6}\geq \frac{5}{6}$ and the equality holds at $(n,k)=(1,2).$ The governing equations \eqref{eq:governLe0} and \eqref{eq:3Others} and the estimate \eqref{eq:prelimNL} imply that
\begin{align}\label{eq:preNKL}
    |\alpha_{n,k,l}(\tau)|\lesssim & e^{-\frac{3}{5}\tau},\ \text{if}\ (n,k)\not=  (0,2),\ (3,0),\ (2,1).
\end{align}

Except for the sharp estimates obtained in \eqref{eq:02Sharp} and \eqref{eq:3021}, the other ones can be improved.

If $N$ is large enough, then we feed \eqref{eq:02Sharp}-\eqref{eq:preNKL} into \eqref{eq:governLe0} and \eqref{eq:3Others}, and apply Lemma \ref{LM:3mPrelim} and \eqref{eq:3mtau}, to find the desired estimates:
for any constant $\delta_0>0$, there exists a $C_{\delta_0}>0$ such that
\begin{align}
    G_{n,m}^{\frac{1}{2}}(\tau)\leq C_{\delta_0} e^{-(1-\delta_0)\tau},\label{eq:gnm3}
\end{align} and for $n=0,1,$
\begin{align}
\begin{split}\label{eq:matFin}
    e^{\frac{1}{3}\tau}|\alpha_{0,k,l}(\tau)|+e^{\frac{5}{6}\tau}|\alpha_{1,k,l}(\tau)| \lesssim & 1, \ \text{for any}\ k,
    \end{split}
\end{align}
and for $n=2,3,$
\begin{align}
    \begin{split}
    |\alpha_{2,k,l}|\lesssim & e^{-\tau},\ \text{for any}\ k\not=1,\\
    |\alpha_{3,k,l}|\lesssim & e^{-\tau} (1+\tau), \ \text{for any}\ k\not=0,
\end{split}
\end{align} 
Here the estimates for $\alpha_{3,1,l}$ are sharp: through $-\frac{1}{2q}|\partial_{y}q|^2$ in $G_2$, $NL_{3,1,l}$ contains
\begin{align}
    -\frac{1}{\|H_3\omega_l\|_{\mathcal{G}}^2}\Big\langle q^{-1} \Big(\partial_{y} \alpha_{3,0,1} H_3\Big)\Big(\partial_{y}\alpha_{2,1,l} H_2 \omega_l\Big), H_3 \omega_l\Big\rangle_{\mathcal{G}}=\mathcal{O}(e^{-\tau}).
\end{align} This together with the equation for $\alpha_{3,1,l}$ implies that $\alpha_{3,1,l}(\tau)=\mathcal{O}(e^{-\tau}(1+\tau)).$

What is left is to prove the refined form \eqref{eq:d21Order} for $\alpha_{2,1,l}$.  
\eqref{eq:matFin} and \eqref{eq:gnm3} imply that $$|NL_{2,1,l}|\lesssim e^{-\tau}.$$ This, together with
\eqref{eq:3021} above, implies that, for $l=1,2,3,4,$
\begin{align}
\begin{split}
    \alpha_{2,1,l}(\tau)=&e^{-\frac{1}{2}\tau} \Big(\alpha_{2,l,0}(0)+\int_{0}^{\infty}e^{\frac{1}{2}\sigma} NL_{2,1,l}(\sigma) \ d\sigma\Big)-
    e^{-\frac{1}{2}\tau}\int_{\tau}^{\infty}e^{\frac{1}{2}\sigma} NL_{2,1,l}(\sigma) \ d\sigma\\
    =& d_{2,l} e^{-\frac{1}{2}\tau}+\mathcal{O}(e^{-\tau})
\end{split}
\end{align} where the constants $d_{2,l}$ are naturally defined.

\end{proof}

%%%%%%%%%%%%%%%%%%%%%%%%%%%%%%%%%%%%%%%%%%

\section{Proof of Lemma \ref{LM:frequencyWise}}
\label{sec:propagator}
We start with proving \eqref{eq:propagatorNoPro}. The key step is to find an integral kernel for the propagator.

Here we use the path integral technique, see \cite{BrKu, DGSW}, to find that, for any function $g$,
\begin{align}
U(\tau,\sigma)g(y)=e^{\tau-\sigma}\int_{\mathbb{R}^3}K_{\tau-\sigma}(y,z) \langle e^{-V}\rangle(y,z)e^{-\frac{1}{8}|z|^2} g(z)\ dz.\label{eq:IntKernel}
\end{align} where the function $\langle e^{-V}\rangle(y,z)$ is defined in terms of path integral
\begin{align*}
\langle e^{-V}\rangle(y,z):=\int e^{-\int_{\sigma}^{\tau} V(\omega_0(s)+\omega(s),s)\ ds} d\mu(\omega),
\end{align*} and $K_{\tau-\sigma}(y,z)e^{-\frac{1}{8}|z|^2}$ is the integral kernel of the operator $e^{-(\tau-\sigma)\mathcal{L}_0}$ given by Mehler's formula,
\begin{align}\label{eq:meh}
K_{\tau-\sigma}(y,z):=2\sqrt{2}\pi  (1-e^{-(\tau-\sigma)})^{-\frac{1}{2}} e^{\frac{1}{8}y^2} e^{-\frac{|y-e^{-\frac{\tau-\sigma}{2}}z|^2}{4(1-e^{-(\tau-\sigma)})}},
\end{align}
and $d\mu(\omega)$ is a probability measure on the continuous paths $\omega:[\sigma,\tau]\rightarrow \mathbb{R}^{3}$ with $\omega(\sigma)=\omega(\tau)=0$, and $\omega_0(s)$ is a path with $\omega_0(\sigma)=z$ and $\omega_0(\tau)=y$ defined as
\begin{align*}
\omega_0(s)=e^{\frac{1}{2}(\tau-s)} \frac{e^{\sigma}-e^{s}}{e^{\sigma}-e^{\tau}}y+e^{\frac{1}{2}(\sigma-s)} \frac{e^{\tau}-e^{s}}{e^{\tau}-e^{\sigma}}z.
\end{align*} Here $\mathcal{L}_0$ is a linear operator defined as
\begin{align}
\mathcal{L}_0:=-\partial_{y}^2+\frac{1}{16}y^2-\frac{1}{4}.
\end{align}
Since $V$ is nonnegative, $0\leq \langle e^{-V}\rangle\leq 1$, 
\begin{align}
|U(\tau,\sigma)g|\leq e^{-(\tau-\sigma)(\mathcal{L}_0-1)}|g|.
\end{align} This implies the desired \eqref{eq:propagatorNoPro}, see e.g. \cite{DGSW}, for any $k\geq 0,$
\begin{align}
\|\langle y\rangle^{-k} e^{\frac{1}{8}|y|^2}U(\tau,\sigma)g\|_{\infty}\leq \|\langle y\rangle^{-k} e^{\frac{1}{8}|y|^2}e^{-(\tau-\sigma)(\mathcal{L}_0-1)}g\|_{\infty} \lesssim e^{\tau-\sigma}\|\langle y\rangle^{-k} e^{\frac{1}{8}|y|^2}g\|_{\infty}.\label{eq:fi123}
\end{align}

Next we sketch a proof for \eqref{eq:propagatorMPro}. 
We will only prove the estimates for the cases $\tau-\sigma\geq 1$. When $\tau-\sigma<1$, we obtain the desired result by applying the maximum principle.

To cast the problem into a convenient form, we define a function $g$ as
$$g(y,\tau):=U_{n}(\tau,\sigma) P_{n} g= P_{n} U_{n}(\tau,\sigma) P_{n} g,$$ then $g(y,\tau)=P_{n} g(y,\tau)$ is the solution to the equation
\begin{align}
\begin{split}\label{eq:linearEvo}
\partial_{\tau}g(y,\tau)=&-P_n(\mathcal{L}_0+V)P_n g(y,\tau)\\
=&-(\mathcal{L}_0+V) g(y,\tau)+(1-P_{n})V g(y,\tau)\\
=&-(\mathcal{L}_0+V) g(y,\tau)+(1-P_{n})\Big(V(y,\tau)-V(0,\tau)\Big) g(y,\tau)
\end{split}
\end{align}here we use that $P_{n}$ commutes with $\mathcal{L}_0$, and that $(1-P_n) g=0$ in the last step. 
Rewrite \eqref{eq:linearEvo} by applying Duhamel's principle to obtain, for any times $\tau\geq \sigma,$ 
\begin{align}
g(y,\tau)=U(\tau,\sigma) P_{n}g(y,\sigma)+\int_{\sigma}^{\tau} U(\tau,s) (1-P_{n})\tilde{V} g(y,s)\ ds,\label{eq:durhamelFin}
\end{align} where $\tilde{V}$ is a function defined as
\begin{align*}
\tilde{V}(y,\tau):=V(y,\tau)-V(0,\tau).
\end{align*}

Recall that $V$ satisfies, for some $\epsilon_0\ll 1,$
\begin{align}
     \sup_{y}\Big\{ \Big|\partial_{y}V(y,\tau)\Big|+ \langle y\rangle^{-5}\Big|V(y,\tau)-V(0,\tau)\Big|\Big\}\leq \epsilon_0 (1+\tau)^{-\frac{2}{5}}.\label{eq:finiteRe}
\end{align}
We claim that the two terms in (\ref{eq:durhamelFin}) satisfy the following estimates: when $\tau\geq \sigma$ and
\begin{align}
e^{\frac{n}{2}(\tau-\sigma)}\leq \epsilon_0^{\frac{1}{2}} (1+\sigma)^{-\frac{1}{5}},\label{eq:length}
\end{align} there exists some constant $C$ independent of $\tau$, $\sigma$ and $g$, such that
\begin{align}
\Big\|\langle y\rangle^{-n-1} e^{\frac{1}{8}|y|^2}U(\tau,\sigma) P_{n}g(\cdot,\sigma)\Big\|_{\infty}\leq C e^{-\frac{n-1}{2}(\tau-\sigma)}&  \Big\|\langle y\rangle^{-n-1} e^{\frac{1}{8}|y|^2}g(\cdot,\sigma)\Big\|_{\infty}, \label{eq:firsInt}
\end{align} and 
\begin{align}
\begin{split}\label{eq:secoInt}
\Big\|\langle y\rangle^{-n-1} e^{\frac{1}{8}|y|^2}&\int_{\sigma}^{\tau} U(\tau,s) (1-P_{n}) \tilde{V} g(\cdot,s)\ ds\Big\|_{\infty}\\
\leq & C\epsilon_0^{\frac{1}{2}}\int_{\sigma}^{\tau} e^{-\frac{n-1}{2}(\tau-s)} \Big\|\langle y\rangle^{-n-1}e^{\frac{1}{8}|y|^2}g(\cdot,s)\Big\|_{\infty}\ ds. 
\end{split}
\end{align}
The claims will be proved in subsection \ref{sub:twoClaims} below.

Suppose the claims hold, then they and \eqref{eq:durhamelFin} imply that, under the condition (\ref{eq:length}),
\begin{align}
\begin{split}\label{eq:tauSigma}
\Big\|\langle y\rangle^{-n-1} e^{\frac{1}{8}|y|^2} g(\cdot,\tau)\Big\|_{\infty}\leq 2 C& e^{-\frac{n-1}{2}(\tau-\sigma)} \Big\|\langle y\rangle^{-n-1} e^{\frac{1}{8}|y|^2}g(\cdot,\sigma)\Big\|_{\infty}.
\end{split}
\end{align}

This is not the desired \eqref{eq:propagatorMPro}, but implies it: for any interval $[\sigma,\ \tau]$ we divide it into finitely many subintervals such that each of them satisfies the condition \eqref{eq:length}, then apply \eqref{eq:tauSigma} on each interval, and then put them together to obtain the desired estimate. An important observation is that as $\sigma$ increases, we can choose a larger time interval $[\sigma,\ \tau]$ to satisfy the condition \eqref{eq:length}.

\subsection{Proof of (\ref{eq:firsInt}) and (\ref{eq:secoInt})}\label{sub:twoClaims}
We start with proving the easier \eqref{eq:firsInt}. 

Recall that $1-P_n$ is the orthogonal projection onto the subspace spanned by $\{e^{-\frac{1}{8}y^2}H_{j}, \ j=0,\cdots,n\}.$
This, together with \eqref{eq:propagatorNoPro} and \eqref{eq:finiteRe}, implies that
\begin{align*}
    \Big\|\langle y\rangle^{-n-1}\int_{\sigma}^{\tau} U(\tau,s) (1-P_{n})\tilde{V} g(y,s)\ ds\Big\|_{\infty}\lesssim \epsilon_0\int_{\sigma}^{\tau}e^{\tau-\sigma} (1+\sigma)^{-\frac{2}{5}}  \Big\|\langle y\rangle^{-n-1}e^{\frac{1}{8}|y|^2}g(\cdot,s)\Big\|_{\infty}\ ds.
\end{align*}
From here we apply \eqref{eq:length} to obtain the desired \eqref{eq:firsInt}.

Next we prove \eqref{eq:secoInt}. The trick is to integrate by parts and observe that $\partial_{z}^{k}K_{\tau-\sigma}(y,z)$ has a favorably decaying factor $e^{-k\frac{\tau-\sigma}{2}}$ for any $k\in \mathbb{N}.$ Integrate by parts in $z$ to obtain
\begin{align}
    \int_{-\infty}^{\infty} K_{\tau-\sigma}(y,z) \langle e^{-V}\rangle(y,z) e^{-\frac{1}{8}z^2} & g(z) \ dz
    = \int_{-\infty}^{\infty} \Big(\partial_{z}K_{\tau-\sigma}(y,z)\Big) \langle e^{-V}\rangle(y,z) g^{-1}(z)\ dz+D_0\nonumber\\
    =&\int_{-\infty}^{\infty} \Big(\partial_{z}^{n+1} K_{\tau-\sigma}(y,z)\Big) \langle e^{-V}\rangle(y,z) g^{-n-1}(z)\ dz+\sum_{k=0}^{n}D_k,\label{eq:n1IntegBParts}
\end{align}
where the terms $D_k$, for any integer $k\geq 0,$ are defined as
\begin{align}
    D_{k}:=\int_{-\infty}^{\infty} \Big(\partial_{z}^{k}K_{\tau-\sigma}(y,z)\Big) \Big(\partial_{z}\langle e^{-V}\rangle(y,z)\Big) g^{-k-1}(z)\ dz,
\end{align}   
and $g^{-l}$, for any $l\in \mathbb{N}$, are functions defined as
\begin{align*}
    g^{-l}(z)=\int_{z}^{\infty}\int_{z_1}\cdots \int_{z_{l-1}} e^{-\frac{1}{8}z_l^2} g(z_l)\ dz_l\cdots dz_1.
\end{align*}

Recall that $g\perp e^{-\frac{1}{8}y^2} H_{k},\ k=0,\cdots, n$. Thus
$g^{-l}\perp  e^{-\frac{1}{8}y^2}H_{0},\cdots,e^{-\frac{1}{8}y^2} H_{n-l}$ if $l\leq n$,
and
\begin{align*}
    g^{-j}(z)=&-\int^{z}_{-\infty}\int_{z_1}\cdots \int_{z_{l-1}} e^{-\frac{1}{8}z_j^2} g(z_j)\ dz_j\cdots dz_1,\ \text{if}\ j=1,\cdots, n+1,
\end{align*} and furthermore by L'Hopital's rule,
\begin{align}
    |g^{-l}(z)|\leq & (1+|z|)^{n+1-l} e^{-\frac{1}{4}z^2} \Big\|\langle y\rangle^{-n-1} e^{\frac{1}{8}y^2} g  \Big\|_{\infty}.
\end{align}
This, together with $$|\partial_{z}\langle e^{V}\rangle(y,z)|\lesssim \epsilon_0 (1+\sigma)^{-\frac{2}{5}},$$ and the techniques used in proving \eqref{eq:propagatorNoPro}, implies the desired result. 
For more details see \cite{BrKu, DGSW}.

%\bibliographystyle{abbrv}
%\addcontentsline{toc}{chapter}{Bibliography}
%\bibliography{biblio}

\def\cprime{$'$}

\end{document}